\documentclass{amsart}
\usepackage[english]{babel}
\usepackage[T1]{fontenc}
\usepackage{indentfirst}
\usepackage{latexsym}
\usepackage{mathtools}
\usepackage{hyperref}
\hypersetup{colorlinks,breaklinks,
            linkcolor=[rgb]{0,0,0},
            citecolor=[rgb]{0,0,0},
            urlcolor=[rgb]{0,0,0}}
\usepackage[dvips]{graphicx}
\DeclareGraphicsExtensions{.pdf,.png,.jpg,.eps}
\usepackage{amsmath} 
\usepackage{amsfonts,amsthm,amssymb,amsbsy}
\usepackage{textcomp}
\usepackage[usenames,dvipsnames]{color}
\usepackage{color}
\usepackage{pgf}
\usepackage{enumitem}
\usepackage{pgfplots}
\usepackage{multirow}
\usepackage{esint}
\usepackage[hang,sf,small]{caption}
\usepackage{graphicx}
\usepackage{wrapfig}
\usepackage{verbatim}
\usepackage[utf8]{inputenc} 

\usepackage{fancyhdr}

\usepackage[color = white, linecolor = black, shadow]{todonotes}

\usepackage{stackengine}
\stackMath

\usepackage{color}
\usepackage[active]{srcltx}
\newcommand{\brd}[1]{\mathbb{#1}}

\newcommand{\R}{\brd{R}}

\newcommand{\N}{\brd{N}}

\newcommand{\eps}{\varepsilon}
\newcommand{\abs}[1]{\left\lvert {#1} \right\rvert}
\newcommand{\norm}[2]{\left\Vert {#1} \right\Vert_{#2}}

\newcommand{\loc}{{{\tiny{\mbox{loc}}}}}

\newcommand\ddfrac[2]{\frac{\displaystyle #1}{\displaystyle #2}}
\newtheorem{teo}{Theorem}[section]
\newtheorem{corollary}[teo]{Corollary}
\newtheorem{lemma}[teo]{Lemma}
\newtheorem{theorem}[teo]{Theorem}
\newtheorem{proposition}[teo]{Proposition}
\theoremstyle{definition}
\newtheorem{definition}[teo]{Definition}

\newtheorem{remark}[teo]{Remark}
\newcommand{\be}{\begin{equation}}
\newcommand{\ee}{\end{equation}}

\usepackage{cfr-lm}
\usepackage[osf]{libertine}

\definecolor{navyblue2}{RGB}{0, 0, 128}
\definecolor{navyblue}{RGB}{0, 25, 111}
\definecolor{blus}{RGB}{0,102,204}
\definecolor{granata}{RGB}{140,29,28}%
\definecolor{red2}{RGB}{220,10,48}
\definecolor{greeny}{RGB}{34,55,58}

\begin{document}

\author[G. Tortone]{Giorgio Tortone}\thanks{}
\address{Giorgio Tortone \newline \indent Dipartimento di Matematica
	\newline\indent
	Alma Mater Studiorum Universit\`a di Bologna
	\newline\indent
	 Piazza di Porta San Donato 5, 40126 Bologna, Italy}
\email{giorgio.tortone@unibo.it}
\title[The nodal set of solutions to some nonlocal sublinear problems]{The nodal set of solutions to some nonlocal sublinear problems}
\date{\today}


\thanks{Work partially supported by the ERC Advanced Grant 2013 n.~339958 COMPAT and by the GNAMPA project ``Esistenza e
proprietà qualitative per soluzioni di EDP non lineari ellittiche e paraboliche''.}

\maketitle
\begin{abstract}
We study the nodal set of solutions to equations of the form
$$
(-\Delta)^s u = \lambda_+ (u_+)^{q-1} - \lambda_- (u_-)^{q-1}\quad\text{in $B_1$},
$$
where $\lambda_+,\lambda_->0, q \in [1,2)$, and $u_+$ and $u_-$ are respectively
 the positive and negative part of $u$.
This collection of nonlinearities includes the unstable two-phase membrane problem $q=1$ as well as sublinear equations for $1<q<2$.\\
We initially prove the validity of the strong unique continuation property and the finiteness of the vanishing order, in order to implement a blow-up analysis of the nodal set.\\
As in the local case $s=1$, we prove that the admissible vanishing orders can not exceed the critical value $k_q= 2s/(2- q)$. Moreover, we study the regularity of the nodal set and we prove a stratification result. Ultimately, for those parameters such that $k_q< 1$, we prove a remarkable difference with the local case: solutions can only vanish with order $k_q$ and the problem admits one dimensional solutions.\\
Our approach is based on the validity of either a family of  Almgren-type or a 2-parameter family of Weiss-type
monotonicity formulas, according to the vanishing order of the solution.
\end{abstract}
\tableofcontents
\section{Introduction}
The analysis of the nodal set of solutions of elliptic equations has been the subject of an intense study in the last decades, starting from the works \cite{MR943927,MR1305956,MR1639155, MR1090434}, with a special focus on the measure theoretical features of its singular part.\\
These works provide a fairly complete picture of the geometric structure of the nodal set in the case of solutions of linear equations and they easily extend to a wide class of superlinear equations of type $-\Delta u = f(u)$, provided that the nonlinearity is locally Lipschitz continuous, that $f(0) = 0$ and that $u \in L^\infty_\loc$. From a geometric point of view,
the nodal set of a weak solution of class $C^1$ splits into a regular part, which is locally a $C^1$ graph, and a singular set which is a countable union of subsets of sufficiently smooth $(n-2)$-dimensional manifolds. Moreover these equations satisfy the strong unique continuation principle and the solutions vanish with finite integer order (see e.g. \cite{MR833393, MR882069, MR1090434}). A similar structure also holds under weaker assumptions, that is, for weak solutions of linear equations in divergence form with Lipschitz coefficients and bounded first and zero order terms (see \cite{MR1305956}).\\
Instead, the picture change drastically if we switch to  semi-linear elliptic equations with non-Lipschitz nonlinearities: given $q \in [1,2)$, let us consider for example the class of equations
\be\label{local}
-\Delta u =  \lambda_+ (u_+)^{q-1} - \lambda_- (u_-)^{q-1} \quad \mbox{in }B_1,
\ee
where $\lambda_+,\lambda_->0, q \in [1,2), B_1$ is the unit ball in $\R^n$ and $u_+=\max(u,0)$ and $u_-=\max(-u,0)$ are respectively the positive and negative part of $u$. Notice that the main feature of these equations stays in the fact that the right hand side is not locally Lipschitz continuous as function of $u$, and precisely has sublinear character for $q \in (1,2)$ and discontinuous behaviour for $q=1$.
It is well known in the literature that in the case $\lambda_+,\lambda_-\leq 0$, the features of the nodal set of
solutions are substantially different in comparison with the linear case since dead cores appear and no unique continuation can be expected.\\
However, in the unstable setting the solutions resembles some features of the linear case. Indeed, recently in \cite{soaveweth} have been proved the validity of the unique continuation principle for every $q \in [1,2)$ by controlling the oscillation of the Almgren-type frequency formula for solutions with a dead core. On the other hand, in \cite{MR3857504} has been shown that the strong unique continuation principle holds for every $q \in (1,2)$, with an alternative approach based on Carleman’s estimate: in both papers it has been emphasized that the standard aprroaches are not applicable in a standard way in the sublinear
and discontinuous cases and have to be considerably adjusted. Finally, in \cite{soavesublinear} the authors investigate the geometric properties of the nodal set and the local behaviour
of the solutions by proving the finiteness of the vanishing order at every point and by studying the regularity of the nodal set of any solution. More precisely, they show that the nodal set is a
locally finite collection of regular codimension one manifolds up to a residual singular set having Hausdorff dimension at most $(n-2)$.\\
Ultimately, the main features of the nodal set are strictly related to those of the solutions to linear (or superlinear) equations, with a remarkable difference: the admissible vanishing orders can not exceed the critical value $k_q=2/(2 - q)$. Moreover, at this threshold, they proved the non-validity of any estimates of the $(n - 1)$-dimensional measure of the nodal set of a solution in terms of the vanishing order.\\\\
The purpose of this paper is to study the structure of the nodal sets of solutions to
\be\label{ini}
(-\Delta)^s u = \lambda_+ (u_+)^{q-1} - \lambda_- (u_-)^{q-1}\quad\text{in $B_1$},
\ee
where $\lambda_+,\lambda_->0, q \in [1,2)$ and $s\in (0,1)$. This study is driven by the wish to extend the previous theory to the fractional setting emphasizing the possible difference between the two type of operators due to the nonlocal attitude of the equations. Starting from the problem of unique continuation, many result have been achieved in the study of the nodal set of solution of nonlocal elliptic equation, in particular by using local realisation of the fractional powers of the Laplacian based on the extension technique popularized by the authors in \cite{CS2007}. Also in this setting, the key tools in proving unique continuation in the linear case are based on the validity of an Almgren-type monotonicity formula (see e.g. \cite{fallfelli2,fallfelli1}), or Carleman estimates (see e.g. \cite{ruland1,ruland2}), which are not applicable in a standard way in our case.\\ In a slightly different direction,
researcher also analyzed the structure of the nodal sets from the geometric point of view by classifying the possible local behaviour of solution near their nodal set: recently in \cite{STT2020} the authors provided a stratification result for the nodal set of linear equation by applying a geometric-theoretic analysis of the nodal set of solutions
to degenerate or singular equations associated to the extension operator of the fractional Laplacian. In particular, they proved the existence of two stratified singular sets where the solution either resembles a classical harmonic function or a generic polynomial: in the first case, the stratification coincides with the one of the nodal set of solutions of local elliptic equations; in the second one a stratification still occurs but the bigger stratum is contained in a countable union of $(n-1)$-dimensional $C^{1,\alpha}$ manifolds, in contrast with the local case $s=1$ (see \cite[Section 8]{STT2020} for more detail in this direction).\\
On the other hand, the picture changes considerably in the case of solution with discontinuous nonlinearity, as in \eqref{ini}. Indeed, it is clear that in the case $\lambda_+,\lambda_-\leq 0$ (where the signs of the coefficients are opposite to ours), the features of the nodal set of solutions are substantially different in comparison with the linear case: dead cores appear
and no unique continuation can be expected. In those scenarios one may try to describe the structure and the regularity of the free boundary $\partial\{u=0\}$. When $q \in (1,2)$ we refer to \cite{yang, yijing} where the authors consider an Alt-Phillips type functional in the fractional setting for the case of non-negative solutions $u\geq 0$; while for the case $q=1$ the equation is the so called two phase obstacle problem and we refer to \cite{allenlindgrenpetro, allenpetro} and reference therein. Since in the fractional case minimisers of the two-phase obstacle problem do not change sign, we refer to \cite{MR4018099, silvestrepaper} for some general result in the one-phase setting.\\
In contrast, very little is known about the structure of the nodal sets in the case $\lambda_+,\lambda_->0$. In \cite{allengarcia} the authors considered the unstable two-phase obstacle problem $q=1$ and they proved that separation of phases does not occur in the unstable setting. Moreover, they characterized the local behaviour of minimisers near the free-boundary and they proved a bound on the Hausdorff dimension of the singular set.\\

In this paper we deal with the two phases problem \eqref{ini}, treating simultaneously the case $q = 1$, which we call unstable two phase membrane problem, in analogy with the local case, and
the case $q\in (1,2)$, a prototype of sublinear equation. Notice that our results extend the classification of blow-up limit obtained for local minimisers in \cite{allengarcia} to weak solution of \eqref{ini}.\vspace{0.35cm}\\
{\textbf{Statements of the main results. }}Exploiting the local realisation of the fractional Laplacian, through the paper we will consider solution of \eqref{ini} as a bounded weak-solution $u \in H^{1,a}_\loc(B^+_1)$ of the extend problem
\begin{equation}\label{system}
\begin{cases}
  L_a u=0 & \mbox{in } B^+_1 \\
  -\partial^a_y u = \lambda_+ (u_+)^{q-1} - \lambda_- (u_-)^{q-1} &\mbox{on } \partial^0 B^+_1,
\end{cases}
\end{equation}
where $a=1-2s \in (-1,1)$,
$$
L_a u = \mathrm{div}(y^a \nabla u), \quad \partial^a_y u(x,0) = \lim_{y \to 0^+}y^a \partial_y u(x,y)
$$
and
$$
B^+_r(X_0) = B_r(X_0) \cap \{y>0\}, \quad\partial^0 B^+_r(X_0)= B_r(X_0) \cap \{y=0\},
$$
where $B_r(X_0)$ denote the ball of center $X_0$ and radius $r$ in $\R^{n+1}$ (through the paper we will simply denote $B^+_r(0)$ with $B^+_r$). From now on, we simply write ``solution'' instead of ``weak solution'', for the sake of brevity. Through the paper we will always denote with $\Gamma(u) = \{(x,0)\colon u(x,0)=0\}$ the restriction of the nodal set of $u$ on $\{y=0\}$.\\
Inspired by \cite{soavesublinear}, we introduce two different notions of vanishing order, which will be proved a posteriori to be equal.
\begin{definition}
  Let $u \in H^{1,a}_\loc(B_1^+)$ be a solution of \eqref{system} and $X_0 \in \Gamma(u)$. The \emph{$H^{1,a}$-vanishing order} of $u$ at $X_0$ is defined as $\mathcal{O}(u,X_0) \in \R^+$, with the property that
  \begin{equation}\label{H1a.vanishing}
  \limsup_{r\to 0^+} \frac{1}{r^{2k}} \norm{u}{H^{1,a}(B_r(X_0))}^2 = \begin{cases}
      0, & \mbox{if } 0< k < \mathcal{O}(u,X_0) \\
      +\infty, & \mbox{if } k >\mathcal{O}(u,X_0).
    \end{cases}
  \end{equation}
  Moreover, if such number does not exist, i.e.
  $$
  \limsup_{r\to 0^+} \frac{1}{r^{2k}} \norm{u}{H^{1,a}(B_r(X_0))}^2 = 0 \quad\mbox{for any }k>0,
  $$
  we set $\mathcal{O}(u,X_0)=+\infty$.
\end{definition}
The advantage of this formulation relays in the fact that we have better control of both the behaviour of the trace of solutions on $\partial^0 B^+_1$ and the character of the solution in the whole extended space. Instead, we recall here the classical definition of vanishing order, which will be used as well through the paper.
\begin{definition}\label{nu}
Let $u \in H^{1,a}_\loc(B_1^+)$ be a solution of \eqref{system} and $X_0 \in \Gamma(u)$. The \emph{vanishing order} of $u$ at $X_0$ is defined as $\mathcal{V}(u,X_0) \in \R^+$, with the property that
  \begin{equation}\label{L2a.vanishing}
  \limsup_{r\to 0^+} \frac{1}{r^{n+a+2k}}\int_{\partial^+ B^+_r(X_0)}y^a u^2 = \begin{cases}
      0, & \mbox{if } 0< k < \mathcal{V}(u,X_0) \\
      +\infty, & \mbox{if } k >\mathcal{V}(u,X_0).
    \end{cases}
  \end{equation}
\end{definition}
By \eqref{poincare} we will easily deduce that $\mathcal{O}(u,X_0) \leq \mathcal{V}(u,X_0)$. The following result establishes the validity of the strong unique continuation principle for every $q \in [1,2), \lambda_+>0,\lambda_-\geq0$ and $s\in (0,1)$
\begin{theorem}\label{strong-unique}
  Let $q \in [1,2), \lambda_+>0,\lambda_-\geq 0$ and $u \in H^{1,a}_\loc(B_1^+)$ a solution of \eqref{system} such that $X_0 \in \Gamma(u)$. If $\mathcal{V}(u,X_0)=+\infty$, then necessarily $u \equiv 0$; in particular, if
for every $\beta>0$ we have
$$
\lim_{\abs{X-X_0}\to 0^+}\frac{\abs{u(x)}}{\abs{X-X_0}^\beta}=0,
$$
it follows that $u\equiv 0$.
\end{theorem}
This result implies the validity of the strong unique continuation principle also for the one-phase case $\lambda_-=0$, which resembles the result in the local setting. Moreover, in the case $\lambda_+,\lambda_->0$, we can improve the previous result by characterizing all
the admissible vanishing orders. Thus, let $\beta_q \in \N$ be the larger positive integer strictly smaller
than $k_q = 2s/(2 - q)$, that is
\be\label{betaq}
\beta_q :=
\begin{cases}
  \Big\lfloor\frac{2s}{2-q}\Big\rfloor, & \mbox{if } \frac{2s}{2-q} \not\in \N \vspace{0.1cm}\\
  \frac{2s}{2-q}-1 & \mbox{if } \frac{2s}{2-q} \in \N.
\end{cases}
\ee
Then, the admissible vanishing orders are all the positive integers smaller or equal than $\beta_q$ and the critical value $k_q$ itself.
\begin{theorem}\label{boundsopra}
  Let $q \in [1,2), \lambda_+,\lambda_->0$ and $u\in H^{1,a}_\loc(B^+)$ be a non-trivial solution of \eqref{system} with $X_0 \in \Gamma(u)$. Then
  $$
  \mathcal{V}(u,X_0) \in \{n \in \N\setminus \{0\}\colon n\leq \beta_q\} \cup \left\{\frac{2s}{2-q}\right\}.
  $$
  In particular, if $2s/(2-q)\leq 1$ then $\mathcal{V}(u,X_0)=2s/(2-q)$.
\end{theorem}
\begin{remark}
In the case $s=1$, our result coincides with the case considered in \cite{soavesublinear}. However, Theorem \ref{boundsopra} reveals a deep difference between the local and nonlocal equations for small value of $s\in (0,1)$: while the vanishing orders of solution of \eqref{local} have a universal bound $k_q=2/(2-q)$, which is always greater or equal than $1$ for $q \in (0,2)$ (see \cite{soavesublinear} for the sublinear case $q\in [1,2)$, \cite{soavesingular} for the singular case $q \in (0,1)$), in the fractional setting this is not always true even in the sublinear case and it implies, for some values of $s\in (0,1)$ and $q\in [1,2)$, the occurrence of solutions which vanish only with order $k_q\leq 1$. Moreover, this particular phenomena will also affect the structure and the regularity of the nodal set.
\end{remark}
Now, using a blow-up argument based on two different types of monotonicity formulas, we proved the validity of a generalized Taylor expansion of the solutions near the nodal set: while in the linear (and superlinear) case solutions behave like homogeneous $L_a$-harmonic functions $p \in \mathfrak{B}_k^a(\R^{n+1}_+)$ symmetric with respect to $\{y=0\}$ (see \cite{STT2020} for a complete characterization of the space $\mathfrak{B}_k^a$), in the sublinear setting this is not necessary the case.
\begin{theorem}\label{caldo}
  Let $q \in [1,2),\lambda_+,\lambda_->0$ and $u\in H^{1,a}_\loc(B^+_1)$ be a solution of \eqref{system} with $X_0 \in \Gamma(u)$. Then, the following alternative holds:
  \begin{enumerate}
    \item if $\mathcal{V}(u,X_0) \in \{n \in \N\setminus \{0\}\colon n\leq \beta_q\} $, then there exists a $k$-homogeneous entire $L_a$-harmonic function $\varphi^{X_0} \in \mathfrak{B}_k^a(\R^{n+1})$ symmetric with respect to $\{y=0\}$, such that
        \begin{equation}\label{eq.continuationa}
u(X)=\varphi^{X_0}(X-X_0) + o(\abs{X-X_0}^{k+\delta}),
\end{equation}
for some $\delta \in \N, \delta >0$;
    \item if $\mathcal{V}(u,X_0)=2s/(2-q)$, then for every sequence $r_k \searrow 0^+$ we have, up to a subsequence, that
$$
\frac{u(X_0 + r_k X)}{\norm{u}{X_0,r_k}}
\to \overline{u} \quad\mbox{in } C^{0,\alpha}_\loc(\R^{n}),
$$
for every $\alpha \in (0,\min(1,2s/(2-q)))$, where $\overline{u}$ is a $2s/(2 - q)$-homogeneous non-trivial solution to
\be\label{limite}
\begin{cases}
  L_a \overline{u}=0 & \mbox{in } \R^{n+1}_+ \\
  -\partial^a_y \overline{u} = \mu\left(\lambda_+ (\overline{u}_+)^{q-1} - \lambda_- (\overline{u}_-)^{q-1}\right) &\mbox{on } \R^n \times \{0\},
\end{cases}
\ee
for some $\mu \geq 0$. Moreover, the case $\mu = 0$ is possible if and only if $2s/(2 - q)\in \N$.
  \end{enumerate}
\end{theorem}
\begin{remark}
While the proof of the latter part of the Theorem is based on the validity of a perturbed Weiss-type monotonicity formula, as its local counterpart in \cite{soavesublinear}, in the proof of the first one of we used a completely different approach. The idea is to take advantage of the bound on the vanishing order to ensure the validity of an asymptotic limit of the Dirichlet-to-Neumann operator $\partial^a_y$ near the nodal set. Then, the blow-up analysis is based on an application of an Almgren and Monneau-type monotonicity formulas. Moreover, to improve the convergence estimate of the remainder in the Taylor expansion \eqref{eq.continuationa}, we apply a blow-up analysis on the difference between the function and its tangent map: this improvement will be crucial in order to estimate the $C^{1,\alpha}$-regularity of the strata of the nodal set.\\
While in the case of local diffusion $s=1$, it is known that a growth estimate of the Laplacian of a function near its nodal set immediately implies the validity of a Taylor expansion of the function itself in terms of harmonic polynomials, in the nonlocal setting the validity of a similar result is still unknown: we think that our strategy could be extended to a more general setting in order to prove a fractional counterpart of the fundamental Lemma in \cite[Lemma 3.1]{fermi}.
\end{remark}
This result leads to a partial stratification of the nodal set and, via the dimension reduction principle due to Federer, to an estimate of the Hausdorff dimension of the nodal and singular set.\\
In the light of the previous results, let us define with $\mathcal{R}(u)$ and $\mathcal{S}(u)$ the regular and singular part of $\Gamma(u)$ defined by
$$
\mathcal{R}(u)=\{X \in \Gamma(u) \colon \abs{\nabla u}(X)\neq 0 \}\quad\mbox{and}\quad
      \mathcal{S}(u)= \{ X \in \Gamma(u) \colon 1<\mathcal{V}(u,X)\leq \beta_q \}.
   $$
and with $\mathcal{T}(u)$ the ``purely sublinear'' part of the nodal set
$$
    \mathcal{T}(u)= \left\{ X \in \Gamma(u) \colon \mathcal{V}(u,X)=\frac{2s}{2-q}\right\}.
    $$
While in the local case $s=1$ the sets $\mathcal{S}(u)\cup \mathcal{T}(u)$ coincides with those points with vanishing gradient, in the fractional setting $s \in (0,1)$ this is not always the case since the critical value is not necessary greater than $1$. Indeed, this slightly different decomposition of $\Gamma(u)$ seems more natural in the fractional setting: by Theorem \ref{boundsopra}, we already know that if $k_q>1$ then
$$
\{X \in \Gamma(u) \colon \abs{\nabla u}(X)= 0 \} = \mathcal{S}(u)\cup \mathcal{T}(u),
$$
while if $k_q\leq 1$ we get $$
\Gamma(u)=\mathcal{T}(u).
$$
Indeed we will see that, for those value of $s\in (0,1)$ and $q\in[1,2)$ such that $k_q>1$, near the points of the nodal set where the function vanishes with order strictly less then $k_q$, the nodal set resembles the picture of the nodal set of $s$-harmonic functions.
\begin{theorem}\label{hau}
  Let $q \in [1,2),\lambda_+,\lambda_->0$ and $u \in H^{1,a}_\loc(B^+_1)$ be a solution of \eqref{system}. The nodal set $\Gamma(u)$ splits as
   $$
   \Gamma(u) = \mathcal{R}(u) \cup \mathcal{S}(u) \cap \mathcal{T}(u),
   $$
   where
  \begin{enumerate}
  \item the regular part $\mathcal{R}(u)$ is locally a $C^{1,\alpha}$-regular $(n-1)$-hypersurface on $\R^{n}$;
  \item the singular part $\mathcal{S}(u)$ satisfies
  $$
  \mathcal{S}(u)= \mathcal{S}^*(u) \cup \mathcal{S}^s(u)
  $$
  where $\mathcal{S}^*(u)$ is contained in a countable union of $(n-2)$-dimensional $C^{1,\alpha}$ manifolds and $\mathcal{S}^s(u)$ is contained in a countable union of $(n-1)$-dimensional $C^{1,\alpha}$ manifolds. Moreover
  $$
   \mathcal{S}^*(u)=\bigcup_{j=0}^{n-2} \mathcal{S}^*_j(u)\quad\mbox{and}\quad
   \mathcal{S}^s(u)=\bigcup_{j=0}^{n-1}\mathcal{S}^s_j(u),
  $$
  where both $\mathcal{S}^*_j(u)$ and $\mathcal{S}^s_j(u)$ are contained in a countable union of $j$-dimensional $C^{1,\alpha}$ manifolds.
\item the sublinear part $\mathcal{T}(u)$ has Hausdorff dimension at most $(n-1)$. Moreover, for $k_q\leq1$ the nodal set coincides with the sublinear stratum and the Haudorff estimate is optimal in the sense that there exists a collection of $2$-dimensional $k_q$-homogeneous solutions such that \be\label{like}
    u_1(x,0)=A_1\left(x_+^{k_q} - x_-^{k_q}\right) \quad\mbox{or}\quad u_2(x,0)=A_2\abs{x}^{k_q}\quad\mbox{for every }x \in \R.
    \ee
    \end{enumerate}
\end{theorem}
The result on $\mathcal{T}(u)$ is remarkably different to its local counterpart: while for $s=1$ the bound $(n-2)$ on the Hausdorff dimension is optimal, we believe that the result on the $(n-1)$-dimension of $\mathcal{T}(u)$ in the case $k_q\leq 1$ can be easily generalized to all $s\in (0,1)$ and $q \in [1,2)$, thanks to the characterization of $L_a$-harmonic function in \cite{STT2020}.\\ Moreover, we claim that a viscosity approach, based on an improvement of flatness, could give a regularity result for those points where the blow-up limit behave like \eqref{like} (see Theorem \ref{esempio} for more detail in this direction), in the case $k_q\leq 1$. At the moment, we leave it as an open problem.\vspace{0.35cm}\\
{\textbf{Structure of the paper. }}The paper is organized as follows. In Section \ref{2} we recall some embedding results and we prove some preliminary results about the optimal regularity of solutions. Moreover, we introduce the notions of vanishing order used through the paper. Next, in Section \ref{3}, we prove the validity of the unique continuation principle in measure and Theorem \ref{strong-unique} by using a 2-parameter Weiss-type monotonicity formula which allows, in Section \ref{4}, to introduce a characterisation of the threshold $k_q$.\\
Finally, in Section \ref{5} we prove the first part of Theorem \ref{caldo} and Theorem \ref{hau} by developing a blow-up analysis based on the validity of two Almgren-type formulas for those points with vanishing order smaller than $k_q$ and, in Section \ref{6}, we complete the proof of Theorem \ref{caldo} by applying a blow-up analysis on those points with vanishing order equal to $k_q$. As byproduct, we will recover Theorem \ref{boundsopra}.\\
Finally, in Section \ref{7} we prove the existence of $k_q$-homogeneous solutions of the form \eqref{like}, for those values of $s$ and $q$ so that $k_q<1$. This result will lead to the Hausdorff estimate of $\mathcal{T}(u)$ in Theorem \ref{hau}.
\section{Preliminaries}\label{2}
In this section we start by showing preliminary results related to the trace embedding of the $H^{1,a}$-space and the optimal regularity of the solution of our problem. As we mentioned in the Introduction, we deal with weak-solution $u \in H^{1,a}_\loc(B^+_1)$ of the problem
$$
\begin{cases}
  L_a u=0 & \mbox{in } B^+_1 \\
  -\partial^a_y u = \lambda_+ (u_+)^{q-1} - \lambda_- (u_-)^{q-1} &\mbox{on } \partial^0 B^+_1,
\end{cases}
$$
where $a=1-2s \in (-1,1)$ and $q\in [1,2)$. With a slight abuse of notations, we will always denote with $u$ the $L_a$-harmonic extension of the solution of \eqref{ini}.\\
The existence of solution follows by standard methods of the calculus of variations and a straightforward application of the following trace embedding. Through the paper, for $X_0 \in \partial^0 B^+_1$ and $r \in (0,1-\abs{X_0})$, we will always consider the space $H^{1,a}(B_r^+(X_0))$ endowed with the norm
$$
\norm{u}{H^{1,a}(B_r^+(X_0))} = \left(\frac{1}{r^{n+a-1}}\int_{B^+_r(X_0)}{y^a \abs{\nabla u}^2 \mathrm{d}X} + \frac{1}{r^{n+a}}\int_{\partial^+B^+_r(X_0)}{y^a u^2 \mathrm{d}\sigma}\right)^{1/2}.
$$
From now on, we often use the notation $\norm{\cdot}{X_0,r}$ to simplify the notation of the norm in $H^{1,a}(B_r^+(X_0))$.
The equivalence of this norm with the classic one is a consequence of the trace theory and the Poincaré inequality. For the sake of completeness
\begin{lemma}\label{lem.poin}
Let $u\in H^{1,a}(B^+)$ and $q \in [1,2^\star]$, where $2^\star = 2n/(n-2s) = 2n/(n+a-1)$ is Sobolev's exponent for the fractional Laplacian. There exists a constant $C_1=C_1(n, p,a)$ such that
\begin{equation}\label{poincare}
  \left(\frac{1}{r^{n}}\int_{\partial^0 B^+_r}{\abs{u}^q\mathrm{d}x}\right)^{\frac{1}{q}} \leq C_1 \norm{u}{H^{1,a}(B_r^+)},
\end{equation}
for every $0 < r < 1$. Namely, the space $H^{1,a}(B^+_r(X_0))$ is continuously embedded in $L^q(\partial^0 B^+_r(X_0))$, for every $r \in (0,1)$
\end{lemma}
\begin{proof}
This result is a direct consequence of the characterization of the class of trace of $H^{1,a}(B^+_r)$, with $r\in (0,1)$, and the Sobolev embedding in the context of fractional Sobolev-Slobodeckij spaces.

For the first inequality \eqref{poincare}, by \cite[Theorem 2.11]{Nekvinda}, the traces of $H^{1,a}(B^+)$ function of the set $\partial^0B_r^+$ coincides with the Sobolev-Slobodeckij space $H^s(\partial^0B^+_r)$. This is defined as the set of all functions $v : \partial^0B_r^+ \to \R$ with a finite norm
\[
  \norm{v}{H^s(\partial^0B^+_r)} := \left(\int_{\partial^0B^+_r}{\abs{v}^2\mathrm{d}x} + \frac{C(n,s)}{2} \int_{\partial^0B^+_r}\int_{\partial^0B^+_r}{\frac{\abs{v(x)-v(z)}^2}{\abs{x-z}^{n+2s}}\mathrm{d}x\mathrm{d}z} \right)^{1/2},
\]
where the term
\begin{equation}\label{gagliardo}
  \left[v\right]_{H^s(\partial^0B^+_r)}=\left(\frac{C(n,s)}{2}\int_{\partial^0B^+_r}\int_{\partial^0B^+_r}{\frac{\abs{v(x)-v(z)}^2}{\abs{x-z}^{n+2s}}\mathrm{d}x\mathrm{d}z} \right)^{1/2}
\end{equation}
is the Gagliardo seminorm of $v$ in $H^s(\partial^0B^+_r)$. Since $\partial^0 B^+_r$ is a Lipschitz domain with bounded boundary, the fractional Sobolev inequality states that
\[
  \norm{v}{L^q(\partial^0 B^+_r)} \leq C \norm{v}{H^s(\partial^0B^+_r)},
\]
for every $q \in [1,2^\star]$, where $2^\star = 2n/(n-2s) = 2n/(n+a-1)$.
\end{proof}
Since $\partial^0 B^+_r$ is a Lipschitz domain with bounded boundary in $\R^n$, the compact embedding in the fractional Sobolev spaces implies the following remark (see \cite{MR2944369} for further details).
\begin{lemma}\label{dafare}
Let $u\in H^{1,a}(B^+)$ and $q \in [1,2^\star)$, where $2^\star = 2n/(n-2s) = 2n/(n+a-1)$ is Sobolev's exponent for the fractional Laplacian. Then $H^{1,a}(B^+_r(X_0))$ is compactly embedded in $L^q(\partial^0 B^+_r(X_0))$, for every $r \in (0,1)$.
\end{lemma}
In this remaining part of the Section we consider the problem of the optimal regularity of solutions of problem \eqref{system}. Since the solutions of \eqref{system} are bounded in $L^\infty_\loc$ by a Moser's iteration argument, we easily deduce that solutions are locally H\"older continuous.
\begin{theorem}\label{holer.reg}
For any compact set $K\subset B$ we get $u \in C^{0,\alpha}(K \cap B^+)$, for every
$$
\alpha < \min\left\{1, \frac{2s}{2-q}\right\}.
$$
Moreover, if $2s/(2-q)<1$, then $u \in C^{0,\alpha^*}_\loc(\overline{B^+})$, with
      $$
      \alpha^* =\frac{2s}{2-q}.
      $$
\end{theorem}
\begin{proof}
It is not restrictive to assume that $K= B_{1/2}$. Hence, we need  to prove that for every $\alpha \in (0,\alpha^*)$ we have
$$
\sup_{X_1,X_2 \in \overline{B}_{1/2}}\frac{\abs{u(X_1)-u(X_2)}}{\abs{X_1-X_2}^\alpha} \leq C,
$$
for some positive constant $C>0$. \\Thus, we proceed by contradiction and develop a blow-up argument: suppose there exists $\alpha \in (0,\alpha^*)$ such that, up to a subsequence, we have
$$
L_k = {\frac{\abs{(\eta u)(X_{1,k})-(\eta u)(X_{2,k})}}{\abs{X_{1,k}-X_{2,k}}^\alpha}} \to \infty
    $$
    where $(X_{1,k},X_{2,k}) \in \overline{B^+_1}\times \overline{B^+_1}$ and $\eta\in C^\infty_c(B_1)$ is a smooth function such that
    $$
    \begin{cases}
      \eta(X)=1, & 0 \leq \abs{X}\leq 1/2 \\
      0<\eta(X)\leq 1, & 1/2\leq \abs{X}\leq 1 \\
      \eta(X)=0, & \abs{X}=1.
    \end{cases}
    $$
Given $r_k=\abs{X_{1,k}-X_{2,k}}$ we can prove, as $k\to \infty$, that
    \begin{itemize}
      \item $r_k\to 0$
      \item $\ddfrac{\mbox{dist}(X_{1,k},\partial^+ B^+_1)}{r_k}\to \infty$, $\ddfrac{\mbox{dist}(X_{2,k},\partial^+ B^+_1)}{r_k}\to \infty$.
    \end{itemize}
    Before to continue, let us fix the notations $X_{1,k}=(x_{1,k},y_{1,k})$ and $X_{2,k}=(x_{2,k},y_{2,k})$.
    Now, since the solution are bounded we deduce
    \begin{equation}\label{lip.1}
    L_k \leq \ddfrac{\norm{u}{L^\infty(B^+_1)}}{r_k^\alpha}\left(\eta(X_{1,k})-\eta(X_{2,k})\right),
    \end{equation}
    which immediately implies that $r_k\to 0$. Now, since $\eta$ is compactly supported in $B_1$ and it vanishes on $\partial^+ B^+_1$, for every $X \in \overline{B^+_1}$ we have
    $$
    \eta(X) \leq \mbox{dist}(X, \partial^+ B^+_1) \mathrm{Lip}(\eta),$$
     where obviously $\mathrm{Lip}(\eta)$ denotes the Lipschitz constant of $\eta$. Finally, the inequality \eqref{lip.1} becomes
    $$
    \ddfrac{\mbox{dist}(X_{1,k},\partial^+ B^+_1)}{r_k}+\ddfrac{\mbox{dist}(X_{2,k},\partial^+ B^+_1)}{r_k} \geq\ddfrac{L_k r_k^{\alpha-1}}{\mathrm{Lip}(\eta)\norm{u}{L^\infty(B^+_1)}}\to \infty
    $$
    and the result follows by recalling that $\alpha <1$. The proof is based on two different blow-up sequences, indeed we introduce the auxiliary sequences
    $$
    w_{k}(X) = \eta(P_k) \frac{u(P_k + r_k X)}{L_k r_k^\alpha} \quad \mbox{and} \quad \overline{w}_{k}(X)=\frac{(\eta u)(P_k + r_k X)}{L_k r_k^\alpha}
    $$
    for $X \in B^+_{P_k,r_k}$ and $P_k=(p_{x,k},p_{y,k})$ a suitable sequence of points that will be choose later. On one hand the sequence $(w_k)_k$ has an uniform bound on the  $\alpha$ - H\"older seminorm, i.e.
     $$
     \sup_{X_1 \neq X_2 \in \overline{B^+_{P_k,r_k}}} \frac{\abs{\overline{w}_k(X_1)-\overline{w}_k(X_2)}}{\abs{X_1-X_2}^\alpha}\leq \abs{\overline{w}_k\left(\frac{X_{1,k}-P_k}{r_k}\right)-\overline{w}_k\left(\frac{X_{2,k}-P_k}{r_k}\right)}=1,
     $$
     while on the other hand
     \begin{equation}\label{blowup.merda.1}
\begin{cases}
-L_a^k w_k = 0 &\mbox{in } B^+_{P_k,r_k}\\
-\partial_y^{a,k} w_{k} = M_k \left( \lambda_+ (w_k)_+^{q-1} - \lambda_- (w_k)_-^{q-1}\right)& \mbox{on } \partial^0 B^+_{P_k,r_k}
\end{cases}
\end{equation}
with
$$
L_a^k=\mbox{div}\left(\left(y+\frac{p_{y,k}}{r_k}\right)^a\nabla\right),\quad
\partial_y^{a,k} = \lim_{y\to -\frac{p_{y,k}}{r_k}} \left(y+\frac{p_{y,k}}{r_k}\right)^a\partial_y,
$$
and
\begin{equation}\label{mk}
M_k = \left(\frac{r_k^{\alpha^*-\alpha}}{L_k}\eta(P_k)\right)^{2-q} \to 0^+,
\end{equation}
since $q \in [1,2)$ and $\alpha \in (0,\alpha^*)$.
Now, the importance of these two sequences lies in the fact that they have asymptotically equivalent behaviour. Namely, since
    \begin{align}
    \begin{aligned}\label{lip2.1}
      \abs{w_{k}(X)-\overline{w}_{k}(X)} \leq & \ddfrac{\norm{u_k}{L^\infty(B_1)}}{r_k^\alpha L_k}\abs{\eta(P_k +r_k X)-\eta(P_k)} \\
      \leq  & \frac{\mathrm{Lip}(\eta)r_k^{1-\alpha}}{ L_k}\norm{u_k}{L^\infty(B_1)}\abs{X}
    \end{aligned}
    \end{align}
    we get, for any compact $K\subset \R^{n+1}$, that
    \begin{equation}\label{asympt.1}
    \max_{X\in K\cap B^+_{P_k,r_k}}\abs{w_k(X)-\overline{w}_k(X)}\longrightarrow 0.
    \end{equation}
    Moreover, since $w_k(0)=\overline{w}_k(0)$  we note by \eqref{lip2.1} that
    \begin{align*}
    \abs{w_{k}(X)-w_{k}(0)} &\leq\abs{w_{k}(X)-\overline{w}_{k}(X)} + \abs{\overline{w}_{k}(X)-\overline{w}_{k}(0)}\\
    &\leq C\left(\frac{r_k^{1-\alpha}}{L_k}\abs{X}+\abs{X}^\alpha\right)
    \end{align*}
    and consequently, there exists $C=C(K)$ such that $\abs{w_k(X)-w_k(0)}\leq C$, for every $X \in K$.\\
    Let us prove that it is not restrictive to choose $P_k \in \Sigma$ in the definitions of the sequences $(w_k)_k$ $(\overline{w}_k)_k$, showing that $X_{1,k},X_{2,k}$ must converge to $\partial^0 B^+_1$, i.e. there exists $C>0$ such that, for $k$ sufficiently large,
    $$
    \ddfrac{\mbox{dist}(X_{1,k},\partial^0 B^+_1)+\mbox{dist}(X_{2,k},\partial^0 B^+_1)}{r_k}\leq C.
    $$
    The following proof follows directly the one of \cite[Lemma 4.5]{tvz2}) but for the sake of complexness we report some details. Arguing by contradiction, suppose that
    $$
    \ddfrac{\mbox{dist}(X_{1,k},\partial^0 B^+_1)+\mbox{dist}(X_{2,k},\partial^0 B^+_1)}{r_k}\longrightarrow \infty
    $$
    and let us choose $P_k = X_{1,k}$ in the definition of $w_k,\overline{w}_k$ so that $B^+_{P_k,r_k}\to \R^{n+1}$ and $p_{y,k}^{-1}r_k\to 0^+$. Given $W_k = w_k -w_k(0)$ and $\overline{W}_k = \overline{w}_k -\overline{w}_k(0)$, by construction   $\overline{W}_k$ is a sequence of functions which share the same bound on the $\alpha$ - H\"older seminorm and they are uniformly bounded in every compact $K\subset \R^{n+1}$ since $\overline{W}_k(0)=0$. Thus, by the Ascoli-Arzel\'a theorem, there exists $W \in C(K)$ which, up to a subsequence, is the uniform limit of $\overline{W}_k$. By \eqref{asympt.1}, we also find that $W_k \to W$ uniformly con compact sets.\\ In order to reach a contradiction we can prove that $W$ is a nonconstant globally H\"older harmonic function with $\alpha \in (0,\alpha^*)$.
    Since we already know that $W\in C^{0,\alpha}(\R^{n+1})$ it reamins to prove the harmonicity of the limit function. To this purpose, let $\varphi \in C^\infty_c(\R^{n+1})$ be a compactly supported smooth function and $\overline{k}$ be sufficiently large so that $\mbox{supp}\,\varphi \subset B^+_{P_k,r_k}$, for all $k\geq \overline{k}$. Fixed
$i = 1,\dots,h$, by testing the first equation in \eqref{blowup.merda.1} with $\varphi$ we get
$$
\int_{\R^{n+1}}{\mbox{div}\left(\left(1+y\frac{r_k}{p_{y,k}}\right)^a\nabla\varphi\right)w_{k}\mathrm{d}X} =0.
$$
Passing to the uniform limit and observing that
$$
\left(1+y\frac{r_k}{p_{y,k}}\right)^a \to 1\quad\mbox{in } C^\infty\left(\mbox{supp}\,\varphi\right),
$$
we deduce that $W$ is indeed harmonic. The contradiction follows by the classical Liouville Theorem once we show that $W$ is globally $\alpha$ - H\"older continuous and not constant. Hence, since $P_k =X_{1,k}$ then, up to a subsequence,
$$
\frac{X_{2,k}-P_k}{r_k}=\frac{X_{2,k}-X_{1,k}}{\abs{X_{2,k}-X_{1,k}}}
 \to X_2 \in \partial B_1.
$$
Finally, by the equicontinuity and the uniform convergence, we conclude
$$
\abs{\overline{W}_{k}\left(\frac{X_1-P_k}{r_k}\right)-\overline{W}_{k}\left(\frac{X_2-P_k}{r_k}\right)}=1 \longrightarrow\abs{\overline{W}(0)-\overline{W}(X_2)}=1.
$$
At this point, the choice $P_k=(x_{1,k},0)$ for every $k \in \N$ guarantees the convergence of the rescaled domains $B^+_{P_k,r_k} \to \R^{n+1}_+$, while for any compact set $K\subset\R^{n+1}$
 \begin{equation*}
    \max_{X\in K\cap B^+_{P_k,r_k}}\abs{w_k(X)-\overline{w}_k(X)}\longrightarrow 0.
    \end{equation*}
Hence, we are left with two possibilities:
    \begin{itemize}
      \item for any compact set $K \subset \R^n \times \{0\}$ we have $w_{k}(X)\neq 0$ for every $k \geq k_0$ and $X \in K$;
      \item there exists a sequence $(X_k)_k \subset \R^n \times \{0\}$ such that $w_k(X_k)=0$, for every $k\in \N$.
    \end{itemize}
In the first case, if we define again $W_k=w_k-w_k(0)$ and $\overline{W}_k = \overline{w}_k - \overline{w}_k(0)$ we obtain that the last sequence is uniformly bounded in $C^{0,\alpha}$ and hence $(W_k)_k$ converges uniformly on compact set to a nonconstant globally $\alpha$ - H\"older continuous $L_a$-harmonic function $W$ such that $\partial^a_y W\equiv 0$. Now, extending properly the $W$ to the whole $\R^{n+1}$ with an even reflection with respect to $\{y=0\}$, we find a contradiction with the Liouville theorem for entire $L_a$-harmonic function symmetric in the variable $y$, since $\alpha< 1$.\\
Similarly, in the second case $(w_k)_k$ itself does converge uniformly on compact sets to a nonconstant globally $\alpha$ - H\"older continuous function $w$, and the contradiction follows in the same way.\\\\
Moreover, if $2s/(2-q)<1$, we can repeat the proof as before with
$
\alpha=\alpha^*.
$
Indeed we can proceed by contradiction and develop a blow-up argument based on two subsequences
$$
    w_{k}(X) = \eta(P_k) \frac{u(P_k + r_k X)}{L_k r_k^{\alpha^*}} \quad \mbox{and} \quad \overline{w}_{k}(X)=\frac{(\eta u)(P_k + r_k X)}{L_k r_k^{\alpha^*}}.
    $$
    More precisely, it can be proved that the sequence $(w_k)_k$ does converge uniformly on compact set to a nonconstant global $\alpha^*$-H\"{o}lder continuous function, in contradiction with the Liouville type theorem con $L_a$-harmonic function even with respect to $\{y=0\}$.
\end{proof}
\section{Strong unique continuation principle}\label{3}
This section is devoted to the proof of the validity of the strong unique continuation principle for solution of \eqref{system}. In order to achieve the main result we start our analysis by proving the unique continuation principle in measure: if a solution $u$ is identically zero in a neighborhood in $\R^{n}\times \{0\}$ of a point $X_0 \in \partial^0 B^+_1$, then necessary $u\equiv 0$ on $\partial^0 B^+_1$. Moreover, since $q \in [1,2)$, this will also imply that $u\equiv 0$ on $B^+_1$ (see \cite{STT2020}).\\

Our proof of the unique continuation is deeply based on the validity of an Almgren-type monotonicity formula. Indeed, let
$$
F_{\lambda_+,\lambda_-}(u)=\lambda_+ (u_+)^q + \lambda_- (u_-)^q,
$$
then for $X_0 \in \Gamma(u)$ and $r \in (0,\mbox{dist}(X_0,\partial^+ B^+_1))$, we introduce the functionals
\begin{align}\label{E.H}
\begin{aligned}
E(X_0,u,r)&=\frac{1}{r^{n-1+a}}\left[\int_{B^+_r(X_0)}y^a\abs{\nabla w}^2 \mathrm{d}X+ \int_{\partial^0 B^+_r(X_0)}w\partial^a_y w \mathrm{d}x\right]\\
&=\frac{1}{r^{n+a-1}}\left[\int_{B^+_r(X_0)}{y^a \abs{\nabla u}^2\mathrm{d}X} - \int_{\partial^0 B^+_r (X_0)}{F_{\lambda_+,\lambda_-}(u)\mathrm{d}x}\right]\\
H(X_0,u,r)&=\frac{1}{r^{n+a}}\int_{\partial^+ B^+_r(X_0)}{y^a u^2 \mathrm{d}\sigma},
\end{aligned}
\end{align}
and the associated Almgren-type formula
\begin{equation}\label{almgren}
N(X_0,u,r)= \frac{E(X_0,u,r)}{H(X_0,u,r)}.
\end{equation}
Through the paper we will often abuse the notation $E(u,r),H(u,r)$ and $N(u,r)$ when it is not restrictive to assume that $X_0=0$. On one side, by the Gauss-Green formula we immediately obtain
$$
E(X_0,u,r)= \frac{1}{r^{n+a-1}}\int_{\partial^+ B^+_r (X_0)}{y^a u \partial_r u\mathrm{d}\sigma},
$$
while, if we differentiate the functions $r\mapsto H(X_0,u,r)$, we get
\begin{equation}\label{h.derivative}
  \frac{d}{dr}H(X_0,u,r) = \frac{d}{dr}\left(\int_{\partial^+ B^+_1}{y^a u^2(X_0+rx)\mathrm{d}\sigma} \right) =\frac{2}{r^{n+a}}\int_{\partial^+ B^+_r}{y^a u\partial_r u \mathrm{d}\sigma}= \frac{2}{r} E(X_0,u,r),
\end{equation}
While the derivative of the denominator of the Almgren-type quotient follows by the a direct computation, in the case of the energy $r\mapsto E(X_0,u,r)$ we need to take care of the sublinear term on the boundary $\partial^0 B^+_r$. 
Now, inspired by direct computations, we get
\begin{proposition}\label{E.derivative}
 Let $X_0 \in \Gamma(u)$ and $r \in (0,\mathrm{dist}(X_0,\partial^+ B^+_1))$. Then, it holds
 \begin{align*}
 \frac{d}{dr} E(X_0,u,r) = &\, \frac{2}{r^{n+a-1}} \int_{\partial^+ B^+_r(X_0)}{y^a (\partial_r u)^2\mathrm{d}\sigma}+\\
 &\, +  \frac{1}{r^{n+a-1}}\left[ \frac{2-q}{q}\int_{S^{n-1}_r(X_0)}{F_{\lambda_+,\lambda_-}(u)\mathrm{d}\sigma}-\frac{C_{n,q}^s}{q r}\int_{\partial^0 B^+_r(X_0)}{F_{\lambda_+,\lambda_-}(u)\mathrm{d}x}\right],
 \end{align*}
 where $C^s_{n,q} = 2n-q(n-2s)>0$.
\end{proposition}
\begin{proof}
  Up to translation, let us assume that $X_0=0$ and consider
  \begin{align*}
  \frac{d}{dr}E(u,r) =&\, \frac{1-n-a}{r^{n+a}}\left[\int_{B^+_r}{y^a \abs{\nabla u}^2\mathrm{d}X} - \int_{\partial^0 B^+_r}{F_{\lambda_+,\lambda_-}(u)\mathrm{d}x}\right]+\\
  &\, + \frac{1}{r^{n+a-1}}\left[\int_{\partial^+B^+_r}{y^a \abs{\nabla v}^2\mathrm{d}\sigma} - \int_{S^{n-1}_r}{F_{\lambda_+,\lambda_-}(u)\mathrm{d}\sigma}\right].
  \end{align*}
  By multiplying the first equation in \eqref{system} with $\langle X,\nabla u\rangle$ and integrating by parts over $B^+_r$ we obtain
  $$
  \int_{B^+_r}{y^a \langle \nabla u, \nabla \langle X,\nabla u\rangle \rangle \mathrm{d}X} = r \int_{\partial^+ B^+_r}{y^a (\partial_r u)^2\mathrm{d}\sigma} + \int_{\partial^0 B^+_r}{\langle x, \nabla_x u\rangle (-\partial^a_y u)\mathrm{d}x}.
  $$
  Using the known identity
  $$
  \langle \nabla u, \langle X, \nabla u\rangle \rangle = \abs{\nabla u}^2 + \left\langle X , \nabla \left(\frac{1}{2}\abs{\nabla u}^2
\right)\right\rangle,$$
we finally get
$$
\frac{1-n-a}{2}\int_{B^+_r}{y^a \abs{\nabla u}^2\mathrm{d}X} + \frac{r}{2}\int_{\partial^+ B^+_r}{y^a \abs{\nabla u}^2\mathrm{d}\sigma} = r\int_{\partial^+ B^+_r}{y^a (\partial_r u)^2\mathrm{d}\sigma} + \int_{\partial^0 B^+_r}{\langle x, \nabla_x u\rangle (-\partial^a_y u)\mathrm{d}x}.
$$
Since $\nabla_x F_{\lambda_+,\lambda_-}(u) = q(-\partial^a_y u) \nabla_x u$ in $\partial^0 B^+_1$, we get
\begin{align}\label{gauss.green}
\begin{aligned}
\int_{\partial^0 B^+_r}{\langle x, \nabla_x u\rangle (-\partial^a_y u)\mathrm{d}x} & = \frac{1}{q}\int_{\partial^0 B^+_r}{\langle x, \nabla_x F_{\lambda_+,\lambda_-}(u)\rangle \mathrm{d}x}\\
&= \frac{r}{q}\int_{S^{n-1}_r}{F_{\lambda_+,\lambda_-}(u)\mathrm{d}\sigma} -\frac{n}{q}\int_{\partial^0 B^+_r}{F_{\lambda_+,\lambda_-}(u)\mathrm{d}x}.
\end{aligned}
\end{align}
Summing together the previous equalities, we finally get the claimed result. We remark that the previous computations are also valid in the case $q = 1$, but require some justification. More precisely, as observed in \cite[5, Proposition 2.7]{soaveweth}, the Gauss-Green formula holds for all vector fields $Y \in C(\overline{B_r},\R^{n+1})$ with $\mathrm{div}Y \in L^1(B_r)$. In particular in \eqref{gauss.green}, the Gauss-Green formula is applied to the vector fields
$$
Y_1 =  F_{\lambda_+,\lambda_-}(u) (x,0) = (\lambda_+ u_+ + \lambda_- u_-) (x,0),
$$
where
$$
\mathrm{div}Y_1 = \mathrm{sign}(\lambda_+ u_+ - \lambda_- u_-)\langle x, \nabla_x u\rangle +n(\lambda_+ u_+ + \lambda_- u_-) \quad\mbox{a.e. in } \partial^0 B^+_1.
$$
The previous quantity is absolutely integrable in $\partial^0 B^+_r$ as a direct consequence of the characterization of the class of trace of $H^{1,a}(B^+_r)$ with $r\in (0,1)$ (see \cite[Theorem 2.11]{Nekvinda}).
\end{proof}
Now, combining the previous estimate, we finally get a lower bound for the derivative of the Almgren-type frequency formula.
\begin{corollary}\label{N.derivative}
Let $X_0 \in \Gamma(u)$ and $r_1,r_2 \in (0,\mathrm{dist}(X_0,\partial^+ B^+_1))$ such that $H(X_0,u,r) \neq 0$ for a.e. $r\in (r_1,r_2)$. Then
$$
\frac{d}{dr}N(X_0,u,r) \geq 
\ddfrac{r\left(\frac{2-q}{q}\right)\int_{S^{n-1}_r(X_0)}{F_{\lambda_+,\lambda_-}(u)\mathrm{d}\sigma} - \frac{C^s_{n,q}}{q}\int_{\partial^0 B^+_r(X_0)}{F_{\lambda_+,\lambda_-}(u)\mathrm{d}x} }{\int_{\partial^+ B^+_r(X_0)}y^a u^2\mathrm{d}\sigma}
$$
for a.e. $r \in (r_1,r_2)$.
\end{corollary}
\begin{proof}
The proof follows essentially the ideas of the similar results in literature and it is based on a straightforward combination \eqref{h.derivative}, Proposition \ref{E.derivative} and the validity of the Cauchy-Schwarz inequality on $\partial^+ B^+_r(X_0)$.
\end{proof}
Now, we are ready to show an intermediate statement in the existence of the vanishing order of out solution. The following is the classical unique continuation principle for solution of the sub-linear nonlocal equation.
\begin{theorem}\label{ucp}
  Let $q \in [1,2), \lambda_+>0, \lambda_-\geq 0$ and $u \in H^{1,a}_\loc(B^+_1)$ be a weak solution of \eqref{system} which vanishes in a neighbourhood in $\R^n\times \{0\}$ of a point on  $\Gamma(u)$. Then $u \equiv 0$ in $\partial^0 B^+_1$.
\end{theorem}
\begin{proof}
  The proof follows the main idea of its local counterpart in \cite{soaveweth}. Hence, let us define the vanishing set on $\R^n\times \{0\}$ as
  $$
  U= \{ x \in \partial^0 B^+_1 \colon u\equiv 0 \mbox{ in a neighborhood of }x\} \subset \R^n.
  $$
  By hypothesis, since $u\not\equiv 0$ and $\partial^0 B^+_1$ is connected, we already know that $U$ is open, non-empty and $\partial U \cap \partial^0 B^+_1 \neq \emptyset$, where $\partial U$ is the topological boundary of $U$ as a subset of $\R^n$. \\Thus, given $X^* \in \partial U \cap \partial^0 B^+_1$ so that $u(X^*)=0$, we claim the existence of $R\in (0,1)$ such that,
  $$
  \mbox{for every }r \in (0,R),\,\, \left\{F_{\lambda_+,\lambda_-}(u) >0 \right\} \cap \partial^0 B^+_r(X^*) \neq \emptyset.
  $$
  This statement shows how the vanishing set of the nonlinearity $F_{\lambda_+,\lambda_-}(u)$ affects the one of the solution $u$. Indeed, while in the case $\lambda_+, \lambda_- >0$ it is obvious that $u \not\equiv 0$ implies the claimed result, if $\lambda_-=0$ the implication is not trivial. Suppose by contradiction that there exists $\overline{r}>0$ such that $u\leq 0$ in $\partial^0 B^+_{\overline{r}}(X^*)$, then by \eqref{system} the function $u$ satisfies
  \begin{equation}\label{sharmonic.ext}
  \begin{cases}
  L_a u=0 & \mbox{in } B^+_{\overline{r}}(X^*) \\
  -\partial^a_y u = 0 &\mbox{on } \partial^0 B^+_{\overline{r}}(X^*).
\end{cases}
  \end{equation}
  Since $X^*\in \partial U$, there exists an open set $\Omega\subsetneqq \partial^0 B^+_{\overline{r}}(X^*)$ such that $u\equiv 0$ in $\Omega$. As a consequence, the unique continuation principle for solution of \eqref{sharmonic.ext} (see e.g. \cite{CS2007, STT2020,fallfelli1}) yields that $u\equiv 0$ in the whole $\partial^0 B^+_{\overline{r}}(X^*)$, which ultimately imply that $X^* \in U$, in contradiction with the fact the $U$ is open.\\
  Finally, as in \cite{soavesublinear} let us consider $X_1 \in U$ such that $\abs{X_1-X^*} < \mathrm{dist}(X_1,S^{n-1}_1)$. Without loss of generality, we can assume that $X_1=0$ and $u\equiv 0$ on $B_R$, for some $R \in (0,1)$. Given
  $$
  d(r)=\frac{1}{q}\int_{\partial^0 B^+_r}{F_{\lambda_+,\lambda_-}(u)\mathrm{d}x}, \quad\mbox{such that}\quad d'(r)=\frac{1}{q}\int_{S^{n-1}_r}{F_{\lambda_+,\lambda_-}(u)\mathrm{d}x},
  $$
  let us set $r_0 = \sup \{ r\geq 0: d(r) =0\}\in (0,1)$. By definition, the function $r\mapsto d(r)$ is not identically zero on $(0,1)$ and, by the monotonicity of $d$, we get $d(r)=0$ for every $r \in (0,r_0)$ and $d(r)>0$ for $r>r_0$.\\Moreover, we claim that $E(u,r) = 0 $ for every $r \in (0,r_0)$: since $R\leq r_0$ and $F_{\lambda_+,\lambda_-}(u)$ is nonnegative, the solution $u$ satisfies
   $$
  \begin{cases}
  L_a u=0 & \mbox{in } B^+_{r} \\
  -\partial^a_y u = 0 &\mbox{on } \partial^0 B^+_{r}.
\end{cases}
  $$
  for every $r \in (0,r_0)$, and $u \equiv 0$ on $\partial B^+_R$. By the unique continuations principle, the solution is identically zero on $\partial^0 B^+_r$. Finally, by \cite[Proposition 5.9]{STT2020}, the function $u$ is identically zero in $B^+_{r_0}$.\\
  Now, by Proposition \ref{E.derivative} we get
  $$
  \frac{d}{dr} E(u,r) \geq \frac{2-q}{r^{n+a-1}}d'(r)-\frac{C_{n,q}^s}{ r^{n+a}}d(r),
  $$
  for every $r \in (0,1)$. In particular, since $q <2$, for $r \in (r_0,1)$
  $$
  \frac{d}{dr} E(u,r) \geq (2-q)d'(r)-C_1d(r) \quad\mbox{with } C_1= \frac{C^s_{n,q}}{r_0^{n+a}}.
  $$
  and integrating in $(r_0,r)$, with $r \in (r_0,1)$ we obtain
  \begin{align*}
  E(u,r) &\geq E(u,r_0) + (2-q)\left(d(r)-d(r_0)\right) - C_1 \int_{r_0}^r d(t)\mathrm{d}t\\
  &\geq (2-q)d(r) - C_1 (r-r_0)d(r),
  \end{align*}
  where we used that $E(u,r_0)=d(r_0)=0$ and the monotonicity of $d$. By choosing $r_1 \in (r_0,1)$ such that $C_1(r_1-r_0) < (2-q)/2$, we infer that
  \begin{equation}\label{estim}
      E(u,r) \geq \frac{2-q}{2}d(r)\quad\mbox{for }r \in (r_0,r_1).
  \end{equation}
  Now, since $u\not\equiv 0$ in $\partial^0B^+_{r_1}$, there must exist $r_2\in (r_0,r_1)$ such that $H(u,r_2)\neq 0$. Fixed the upper bound $r_2$, let us define
  $$
  r_3 = \inf\{r \in (0,r_2) : H(u,t) >0 \mbox{ for every }t \in (r,r_2)\}.
  $$
  Obviously $r_3\geq r_0 >0$ and by the continuity of $u$ we get $H(u,r_3)=0$ and $H(u,r)>0$ for $r\in (r_3,r_2]$. \\Finally, passing through the logarithmic derivative of the Almgren-type formula, by Corollary \ref{N.derivative} we get
$$
\frac{d}{dr}\log N(u,r) \geq -\frac{C^s_{n,q} d(r)}{r^{n+a}E(u,r)}\\
\geq -\frac{2C^s_{n,q}}{(2-q)r^{n+a}} \quad\mbox{for }r \in (r_3,r_2],
$$
where in the second inequality we use the estimate \eqref{estim} on the lower bound of energy $E(u,r)$ on $(r_0,r_1)$. Since $r_3\geq r_0>0$, by integrating the previous inequality we proved that
\begin{equation}\label{almg}
r \mapsto N(u,r) e^{C_3 r}, \quad\mbox{with }\,C_3= \frac{2C^s_{n,q}}{(2-q)r^{n+a}_0},
\end{equation}
is monotone non-decreasing in $(r_3,r_2]$. Since $H(u,r) \neq 0 $ on $(r_3,r_2]$, combining the monotonicity in \eqref{almg} with the estimate \eqref{h.derivative}, we get for $r \in (r_3,r_2]$
\begin{equation}\label{estim2}
\frac{d}{dr}\log H(u,r) = 2\frac{N(u,r)}{r} \leq 2\frac{N(u,r_2)e^{C_3r_2}}{r}\leq 2\frac{N(u,r_2)e^{C_3r_2}}{r_0}.
\end{equation}
The claimed result follows from the definition of $r_3$, namely since $H(u,r_3)=0$ we have
$$
\lim_{r\to r_3^+}\log H(u,r) = -\infty,
$$
in contradiction with the estimate \eqref{estim2}.
\end{proof}

Now, in order to improve the previous result by showing the validity of the strong unique continuation principle, we introduce several monotonicity formulas. These preliminary results  will be also crucial in the study of the local behavior of solutions near nodal points.
Thus, we introduce the fundamental objects of our analysis: inspired by \cite{soavesublinear}, for $X_0 \in \Gamma(u)$ and $r \in (0,\mbox{dist}(X_0,\partial^+ B^+_1))$ we consider the functional
\begin{equation}\label{E.t}
E_t(X_0,u,r) =  \frac{1}{r^{n+a-1}}\left[\int_{B^+_r(X_0)}{y^a \abs{\nabla u}^2\mathrm{d}X} - \frac{t}{q}\int_{\partial^0 B^+_r (X_0)}{F_{\lambda_+,\lambda_-}(u)\mathrm{d}x}\right],
\end{equation}
and similarly we introduce the two-parameters families of functionals
\begin{align}\label{two-para.monotonicity}
\begin{aligned}
  N_t(X_0,u,r) & = \frac{E_t(X_0,u,r)}{H(X_0,u,r)},\\
  W_{k,t}(X_0,u,r) & = \frac{H(X_0,u,r)}{r^{2k}}\left(N_t(X_0,u,r) - k \right)\\
  & = \frac{E_t(X_0,u,r)}{r^{2k}} - k \frac{H(X_0,u,r)}{r^{2k}}.
\end{aligned}
\end{align}
More precisely, the functionals $r\mapsto N_t(X_0,u,r)$ and $r\mapsto W_{k,t}(X_0,u,r)$ are respectively an Almgren-type frequency and a Weiss-type formula. Indeed, for $t=q$ we recover the functionals in \eqref{E.H} and their associated Almgren-type formula.\\
Proceeding exactly as in Proposition \ref{E.derivative} and Corollary \ref{N.derivative}, we get
\begin{equation}\label{E.derivative.t}
 \frac{d}{dr} E_t(X_0,u,r) = \frac{2}{r^{n+a-1}} \int_{\partial^+ B^+_r(X_0)}{y^a (\partial_r u)^2\mathrm{d}\sigma}+ R_t(X_0,u,r)
\end{equation}
where
\begin{equation}\label{E.derivative.t.resto}
  R_t(X_0,u,r) = \frac{1}{r^{n+a-1}}\left[ \frac{2-t}{q}\int_{S^{n-1}_r(X_0)}{F_{\lambda_+,\lambda_-}(u)\mathrm{d}\sigma}-\frac{C_{n,t}^s}{q r}\int_{\partial^0 B^+_r(X_0)}{F_{\lambda_+,\lambda_-}(u)\mathrm{d}x}\right]
\end{equation}
and $C^s_{n,t}= 2n-t(n-2s)$. Moreover, we get
\begin{proposition}\label{weiss.mon}
Let $X_0 \in \Gamma(u)$ and $r \in (0,\mathrm{dist}(X_0,\partial^+ B^+_1))$. Then we have
\begin{align}\label{weiss.general}
\begin{aligned}
  \frac{d}{dr}W_{k,t}(X_0,u,r) =&\, \frac{2}{r^{n+a-1+2k}}\int_{\partial^+ B^+_r(X_0)}{y^a \left(\partial_r u -\frac{k}{r}u\right)^2 \mathrm{d}\sigma} + \\
  &\,+  \frac{1}{r^{n+a-1+2k}}\left[ \frac{2-t}{q}\int_{S^{n-1}_r(X_0)}{F_{\lambda_+,\lambda_-}(u)\mathrm{d}\sigma} -\frac{C_{n,t}^s-2k(t-q)}{q r}\int_{\partial^0 B^+_r(X_0)}{F_{\lambda_+,\lambda_-}(u)\mathrm{d}x}\right].
  \end{aligned}
  \end{align}
In particular, for $t=2$ and $k\geq 2s/(2-q)$ the function $r\mapsto W_{k,2}(X_0,u,r)$ is monotone non-decreasing.
\end{proposition}
\begin{proof}
  Up to translation, let us suppose $X_0=0$ and $r \in (0,1)$. In order to simplify the notations, through the proof we will omit the dependence of the functionals with respect to $u$ and $X_0$. A direct computation gives
  \begin{align}
  \begin{aligned}\label{bohpoi}
    \frac{d}{dr}W_{k,t}(u,r)  = &\, -2k\left( \frac{E_t(u,r)}{r^{2k+1}} - k \frac{H(u,r)}{r^{2k+1}}\right) + \frac{1}{r^{2k}}\left(\frac{d}{dr}E_t(u,r) -k\frac{d}{dr}H(u,r) \right)\\
    = &\,\frac{1}{r^{2k}}\frac{d}{dr}E_t(u,r) -4k\frac{E_t(u,r)}{r^{2k+1}}
        +2k^2\frac{H(u,r)}{r^{2k+1}}
  \end{aligned}
  \end{align}
  where in the second inequality we used the estimate \eqref{h.derivative}. By the Gauss-Green formula in \eqref{E.t} we get
   $$
   E_t(u,r)= \frac{1}{r^{n+a-1}}\left[\int_{\partial^+ B^+_r}{y^a u \partial_r u\mathrm{d}X}- \frac{t-q}{q}\int_{\partial^0 B^+_r }{F_{\lambda_+,\lambda_-}(u)\mathrm{d}x}\right]
   $$
   and taking care of the estimate \eqref{E.derivative.t}, we finally obtain
  \begin{align*}
  \frac{d}{dr}W_{k,t}(u,r) =&\, \frac{2}{r^{n+a-1+2k}}\int_{\partial^+ B^+_r}{y^a \left(\partial_r u -\frac{k}{r}u\right)^2 \mathrm{d}\sigma} +\\
  &\,+  \frac{1}{r^{n+a-1+2k}}\left[ \frac{2-t}{q}\int_{S^{n-1}_r}{F_{\lambda_+,\lambda_-}(u)\mathrm{d}\sigma}-\frac{C_{n,t}^s-2k(t-q)}{q r}\int_{\partial^0 B^+_r}{F_{\lambda_+,\lambda_-}(u)\mathrm{d}x}\right].
  \end{align*}
Now, for $t=2$ and $k\geq 2s/(2-q)$ the monotonicity follows straightforwardly by the previous computations. Indeed, we have
\begin{equation}\label{remainder}
\frac{d}{dr}W_{k,t}(u,r) \geq  -\frac{C_{n,t}^s-2k(t-q)}{qr^{n+a+2k}}
\int_{\partial^0 B^+_r}{F_{\lambda_+,\lambda_-}(u)\mathrm{d}x},
\end{equation}
and in our specific case $C_{n,2}^s - 2k(2-q)\leq 0$ if and only if $k\geq 2s/(2-q)$.
\end{proof}
Thus, as a simple corollary, we deduce the following results.
\begin{corollary}\label{homo}
  Let $X_0\in \Gamma(u)$ and $k\geq 2s/(2-q)$. Then, there exists the limit
$$
W_{k,2}(X_0,u,0^+) = \lim_{r \to 0^+}W_{k,2}(X_0,u,r).
$$
Moreover, for $k=2s/(2-q)$, the map $r\mapsto W_{k,2}(X_0,u,r)$ is constant if and only if $u$ is $2s/(2-q)$-homogeneous in $\overline{\R^{n+1}}$ with respect to $X_0$.
\end{corollary}
\begin{lemma}\label{absurd.gamma}
  For $X_0 \in \Gamma(u)$, there exists $k\geq 2s/(2-q)$ such that $$W_{k,2}(X_0,u,0^+)<0.$$ Moreover, if $W_{k_1,2}(X_0,u,0^+)<0$ then $W_{k_2,2}(X_0,u,0^+)=-\infty$ for every $k_2>k_1$.
\end{lemma}
\begin{proof}
Up to translation, let us consider $X_0=0$ and $r \in (0,1)$.
By Theorem \ref{ucp}, since $u\not \equiv 0$ there exists $r_1 \in (0,1)$ such that $H(u,r_1)\neq 0$. Now, there exists $k\geq 2s/(2-q)$ sufficiently large, such that
$$
W_{k,2}(u,r_1)= \frac{E_2(u,r_1)}{r_1^{2k}} - k \frac{H(u,r_1)}{r_1^{2k}} <0,
$$
and by the monotonicity result in Proposition \ref{weiss.mon} we obtain $W_{k,2}(u,0^+) \leq W_{k,2}(u,r_1) <0$, for $k$ sufficiently large.\\
Now, fixed $k_1>0$ such that $W_{k_1,2}(u,0^+)<0$, let us consider $k_2>k_1$. Thus, for $r \in (0,1)$
\begin{align*}
W_{k_2,2}(u,r) &= \frac{E_2(u,r)}{r^{2k_2}} - k_2 \frac{H(u,r)}{r^{2k_2}}\\
&= \frac{1}{r^{2(k_2-k_1)}}\left[\frac{E_2(u,r)}{r^{2k_1}} - k_1 \frac{H(u,r)}{r^{2k_1}} \right] - \frac{k_2-k_1}{r^{2k_2}}H(u,r)\\
&\leq \frac{1}{r^{2(k_2-k_1)}}W_{k_1,2}(u,r),
\end{align*}
which implies the claimed conclusion.
\end{proof}
As corollary of Lemma \ref{absurd.gamma}, we get
\begin{corollary}\label{overline.k}
  For every  $X_0 \in \Gamma(u)$ such that $\mathcal{O}(u,X_0) \geq 2s/(2-q)$, there exists finite
  $$
  \overline{k} = \inf\{k>0\colon W_{k,2}(X_0,u,0^+)=-\infty\} \in \left[\frac{2s}{2-q},+\infty\right).
  $$
  Moreover, the limit $W_{k,2}(X_0,u,0^+)$ exists for every $k\geq 0$ and it satisfies
  $$
  \begin{cases}
    W_{k,2}(X_0,u,0^+)=0 & \mbox{if } 0<k<\frac{2s}{2-q} \\
    W_{k,2}(X_0,u,0^+)\geq0 & \mbox{if } \frac{2s}{2-q}\leq k < \overline{k} \\
    W_{k,2}(X_0,u,0^+)=-\infty & \mbox{if } k > \overline{k}.
  \end{cases}
  $$
\end{corollary}
\begin{proof}
  The existence of $\overline{k} \geq 0$ follows by Lemma \ref{absurd.gamma}. Now, let us consider separately the cases $k<2s/(2-q)$ and $k\geq 2s/(2-q)$.\\
  In the first case, since $\mathcal{O}(u,X_0)\geq 2s/(2-q)$, by \eqref{H1a.vanishing} there exists $\varepsilon>0$ such that
  $$
  k < \frac{2s}{2-q}-\varepsilon < \mathcal{O}(u,X_0),
  $$
  and two constant $C>0,r_0>0$, depending on $\varepsilon$, such that
  $$
  \norm{u}{H^{1,a}(B^+_r(X_0))}^2 \leq C r^{2\left(\frac{2s}{2-q}-\varepsilon\right)},
  $$
  for every $r\in (0,r_0)$. By definition of the Weiss-type formula, we get
  \begin{align*}
\abs{W_{k,2}(X_0,u,r)} &\leq C\frac{1}{r^{2k}}\left((1+k)\norm{u}{H^{1,a}(B_r(X_0))}^2 + \frac{2}{q} r^{1-a}\norm{u}{H^{1,a}(B_r(X_0))}^q  \right)\\
& \leq C\frac{1}{r^{2k}}\left(r^{2\alpha} + r^{2s+q\alpha} \right),
\end{align*}
with $\alpha = 2s/(2-q) -\varepsilon$. Finally, since $q\in [1,2)$, we get
$$
\abs{W_{k,2}(X_0,u,r)}\leq C r^{2\left(\frac{2s}{2-q}-k-\varepsilon\right)} + C r^{2\left(\frac{2s}{2-q}-k-\frac{q\varepsilon}{2}\right)}
$$
which leads to the claimed result ad $r\to 0^+$. In particular, this estimate suggests that $\overline{k}\geq 2s/(2-q)$.\\
Instead, in the case $k> 2s/(2-q)$ the existence of a non-negative limit for $k< \overline{k}$ follows by the monotonicity result in Proposition \ref{weiss.mon} and by Lemma \ref{absurd.gamma}.
\end{proof}
The previous result emphasizes an hidden relation between the notion of $H^{1,a}$-vanishing order and the transition exponent $\overline{k}$ defined in Corollary \ref{overline.k}, which will be deeply examined in the following section. Finally, we can prove the main result of the section
\begin{proof}[Proof of Theorem \ref{strong-unique}]
  By contradiction, suppose that $u\not\equiv 0$ on $\partial^0 B^+_1$ and $\mathcal{O}(u,X_0)=+\infty$, i.e.
  $$
   \limsup_{r\to 0^+} \frac{1}{r^{2\beta}} \norm{u}{H^{1,a}(B_r(X_0))}^2 = 0, \quad \mbox{for any }\beta>0.
  $$
   In particular, given $\overline{k}>0$ as in Corollary \ref{overline.k}, let us fix $k>\overline{k}$ and $\beta=2k/q$. Thus, there exists $r_0>0$ and $C>0$ such that
  \begin{equation}\label{beta}
  \frac{1}{r^{n+a-1}}\int_{B^+_r(X_0)}{y^a \abs{\nabla u}^2 \mathrm{d}X} + \frac{1}{r^{n+a}}\int_{\partial^+B^+_r(X_0)}{y^a u^2 \mathrm{d}\sigma} \leq C r^{\frac{4}{q}k} \quad \mbox{for every }r \in (0,r_0).
  \end{equation}
  On one side, since $2k/q>k$ for $q \in [1,2)$, by the previous inequality we easily have
  $$
  H(X_0,u,r) \leq C r^{2k}\quad \mbox{for every }r \in (0,r_0)
  $$
  while, by an integration by parts, fixed $\Lambda= \max\{\lambda_+,\lambda_-\}$ we get
  $$
  \frac{1}{r^{n+a-1}}\int_{\partial^0 B^+_r (X_0)}{F_{\lambda_+,\lambda_-}(u)\mathrm{d}x}  \leq \frac{\Lambda}{r^{n+a-1}}   \int_{\partial^0 B^+_r (X_0)}{\abs{u}^q \mathrm{d}x}
  \leq  C r^{1-a}  \norm{u}{H^{1,a}(B_r(X_0))}^{q}
  \leq  C r^{2k},
  $$
  where in the second inequality we use Lemma \ref{lem.poin} and in the last one \eqref{beta}. \\Finally, collecting the previous estimate, for every $r \in (0,r_0)$ we have
  $$
  W_{k,2}(X_0,u,r) \geq -\frac{1}{r^{n+a-1+2k}}  \frac{2}{q}\int_{\partial^0 B^+_r (X_0)}{F_{\lambda_+,\lambda_-}(u)\mathrm{d}x} - \frac{k}{r^{2k}}H(X_0,u,r) \geq -\left(\frac{2}{q}+ k\right)C,
  $$
  and in particular $W_{k,2}(X_0,u,0^+)>-\infty$, in contradiction with the fact that, being $k >\overline{k}$, by Corollary \ref{overline.k} we must have $W_{k,2}(X_0,u,0^+)=-\infty$ for any $k > \overline{k}$.
\end{proof}
\section{
The transition exponent for the Weiss-type formula
}\label{4}
In this section we prove a different characterization of the $H^{1,a}$-vanishing order in terms of the transition exponent for the Weiss-type monotonicity formula, taking care of the different case of $\mathcal{O}(u, X_0) < 2s/(2-q)$ and $\mathcal{O}(u, X_0) \geq 2s/(2-q)$.\\
First, let us consider the latter case and let us prove the following characterization of the transition exponent previously introduced.
\begin{proposition}\label{claim}
  For every  $X_0 \in \Gamma(u)$ such that $\mathcal{O}(u,X_0) \geq 2s/(2-q)$, we have
  $$
 \inf\left\{k>0\colon W_{k,2}(X_0,u,0^+)=-\infty\right\} = \frac{2s}{2-q}.
  $$
\end{proposition}
Moreover, combining the previous estimate with Corollary \ref{overline.k} we will deduce that $W_{k,2}(X_0,u,0^+)$ exists for every $k\geq 0$ and
\begin{equation}\label{transition.up}
  \begin{cases}
    W_{k,2}(X_0,u,0^+)=0 & \mbox{if } 0<k<\frac{2s}{2-q} \\
    W_{k,2}(X_0,u,0^+)=-\infty & \mbox{if } k > \frac{2s}{2-q}.
  \end{cases}
\end{equation}
Following the strategy presented in \cite{soavesublinear}, this result will be a consequence of the following Lemmata.
\begin{remark}\label{soto}
We observe that in the following lemmata we never used the assumption $\mathcal{O}(u,X_0)\geq 2s/(2-q)$, but only the absurd hypothesis $\overline{k}\geq 2s/(2-q)$. Indeed, this results will be used in the study of the local behaviour near point of the nodal set with $\mathcal{O}(u,X_0)< 2s/(2-q)$.
\end{remark}
\begin{lemma}\label{trans1}
Let $X_0 \in \Gamma(u)$ and $\overline{k}\geq 2s/(2-q)$. Then
$$
E_2(X_0,u,r) \geq 0\quad\mbox{and}\quad H(X_0,u,r)>0,
$$
for every $r \in (0,\mathrm{dist}(X_0,\partial^+ B^+_1))$. Moreover, if $k>\overline{k}$, we deduce
$$
\limsup_{r\to 0^+} N_2(X_0,u,r) \leq k \quad\mbox{and}\quad \liminf_{r\to 0^+} \frac{H(X_0,u,r)}{r^{2k}} = +\infty.
$$
\end{lemma}
\begin{proof}
  Fixed $k \in (2s/(2-q),\overline{k}]$, we already know by Proposition \ref{weiss.mon} that $r \mapsto W_{k,2}(X_0,u,r)$ is monotone non-decreasing and, by Corollary \ref{overline.k}, that $W_{k,2}(X_0,u,0^+)\geq 0$.\\ Hence, for every $r \in (0,\mathrm{dist}(X_0,\partial B^+))$ we get
  $$
  0 \leq W_{k,2}(X_0,u,r) \leq  \frac{1}{r^{2k}} E_2(X_0,u,r).
  $$
  Moreover, since $W_{k,q}(X_0,u,r) \geq W_{k,2}(X_0,u,r) \geq 0$ for every $r \in (0,\mathrm{dist}(X_0,\partial B^+))$, by \eqref{h.derivative} we deduce
  \begin{equation}\label{mon.H.weiss}
  \frac{d}{dr} \frac{H(X_0,u,r)}{r^{2k}} = \frac{2}{r} W_{k,q}(X_0,u,r) \geq 0.
  \end{equation}
  Finally, if $H(X_0,u,r_1)=0$ for some $r_1 >0$, by the monotonicity of \eqref{mon.H.weiss}, we deduce that $u\equiv 0$ in $B^+_{r_1}(X_0)$, in contradiction with Theorem \ref{ucp}.\\ Hence, collecting the previous inequality, we get $N_2(X_0,u,r)\geq 0$ and in particular, since $k> \overline{k}$ we get
  $$
  -\infty = W_{k,2}(X_0,u,0^+) = \lim_{r \to 0^+} \frac{H(X_0,u,r)}{r^{2k}}(N_2(X_0,u,r)-k).
  $$
  Also, since $H(X_0,u,r)/r^{2k}\geq 0$, we finally deduce
  $$
  -k \leq \liminf_{r\to 0^+} (N_2(X_0,u,r)-k) \leq \limsup_{r \to 0^+} (N_2(X_0,u,r)-k) \leq 0,
  $$
  which implies the desired claim.
\end{proof}
As a consequence, for every $t \in (0,2)$ the associated Almgren-type formula $N_t(X_0,u,r)$ is non-negative for every $r \in (0,\mathrm{dist}(X_0,\partial^+ B+))$. \\Since in this Section we are proceeding by contradiction by assuming that $\overline{k}>2s/(2-q)$, let $t$ be the medium point between $2s/(2-q)$ and $\overline{k}$.
\begin{lemma}\label{mon.t}
  Let $X_0 \in \Gamma(u)$ and $\overline{k}> 2s/(2-q)$. Given
  $$
  \widetilde{k} = \frac{1}{2}\left(\frac{2s}{2-q}+\overline{k} \right) \quad\mbox{and}\quad \tilde{t} = \frac{2n + 2 kq}{2k+n-2s} \in (q,2),
  $$
  then the map $r \mapsto W_{k,\tilde{t}}(X_0,u,r)$ is monotone non-decreasing in $(0,\mathrm{dist}(X_0,\partial^+ B^+))$.
\end{lemma}
\begin{proof}
  The proof is a direct corollary of Proposition \ref{weiss.mon}. More precisely, since $q \in [1,2)$ and $\tilde{k}>2s/(2-q)$ we get that
  $$
  \overline{t} = \frac{2n + 2 \tilde{k}q}{2\tilde{k}+n-2s} \longleftrightarrow C_{n,\overline{t}}^s - 2\tilde{k}(\tilde{t}-q) =0,
  $$
  which implies, by \eqref{remainder}, the claimed result.
\end{proof}
Inspired by Corollary \ref{overline.k} let us consider the same quantity associated to the limit $W_{k,\tilde{t}}(X_0,u,0^+)$.
\begin{lemma}\label{trans2}
If $\overline{k}>2s/(2-q)$, then
\begin{equation}\label{kk}
\overline{k} = \inf\{k \geq \tilde{k} \colon W_{k,\tilde{t}}(X_0,u,0^+)=-\infty \}.
\end{equation}
In particular, for every $k > \overline{k}$ we get
$$
\limsup_{r\to 0^+} N_{\tilde{t}}(X_0,u,r) \leq k.
$$
\end{lemma}
\begin{proof}
  Following the reasoning in Lemma \ref{absurd.gamma}, we can immediately deduce the existence of $k \geq \tilde{k}$ such that $W_{k,\tilde{t}}(X_0,u,0^+) <0$. Hence, we can reasonably define the quantity
  $$
  \overline{\overline{k}} =  \inf\{k \geq \tilde{k} \colon W_{k,\tilde{t}}(X_0,u,0^+)=-\infty \} ,
  $$
  for which
  $$
  \begin{cases}
      W_{k,\tilde{t}}(X_0,u,0^+)\geq0 & \mbox{if } \overline{k}\leq k < \overline{\overline{k}} \\
    W_{k,\tilde{t}}(X_0,u,0^+)=-\infty & \mbox{if } k > \overline{\overline{k}}.
  \end{cases}
  $$
  Since $\tilde{t}<2$, we fist have $W_{k,\tilde{t}}(X_0,u,r) \geq W_{k,2}(X_0,u,r)$ for every $0<r<R$ and $k>0$. Now, on one side $W_{k,\tilde{t}}(X_0,u,0^+)=-\infty$ implies $W_{k,2}(X_0,u,0^+)=-\infty$ and hence $\overline{\overline{k}}\geq \overline{k}$.
  So, let us suppose by contradiction that $\overline{\overline{k}}>\overline{k}$, hence there exists $k \in (\overline{k},\overline{\overline{k}})$ such that $W_{k,\tilde{t}}(X_0,u,0^+)\geq 0$.\\
  By the monotonicity result in Lemma \ref{mon.t} we get $W_{k,\tilde{t}}(X_0,u,r) \geq 0$ for $r>0$ and, since $\tilde{t}\in (q,2)$, we deduce
  \begin{equation}\label{mon.H.weiss.q}
  W_{k,q}(X_0,u,r) \geq   W_{k,\tilde{t}}(X_0,u,r) \geq 0,
  \end{equation}
  for every $r \in (0,\mathrm{dist}(X_0,\partial^+ B^+))$. Finally, recalling the relation in \eqref{mon.H.weiss}, by \eqref{mon.H.weiss.q} it follows that $r \mapsto r^{-2k}H(X_0,u,r)$ is monotone non-decreasing and in particular there exists finite
  $$
  \lim_{r\to 0^+} \frac{H(X_0,u,r)}{r^{2k}} \in (0,+\infty),
  $$
  which contradicts Lemma \ref{trans1}.
\end{proof}
\begin{lemma}\label{r}
Let $X_0 \in \Gamma(u)$ and $\overline{k}\geq 2s/(2-q)$. There exists a sequence $(r_n)_n$ such that $r_i \to 0^+$ and
$$
\frac{1}{r_n^{n+a-1+2\overline{k}}} \int_{\partial^0 B^+_{r_n}(X_0)}{F_{\lambda_+,\lambda_-}(u)\mathrm{d}x} \to 0.
$$
\end{lemma}
\begin{proof}
  Let $k \in [2s/(2-q),\overline{k})$, by Corollary \ref{weiss.mon} and Corollary \ref{overline.k} we have $W_{k,2}(X_0,u,r)\geq 0$ for every $r \in (0,\mathrm{dist}(X_0,\partial^+ B^+))$. Since that, for any fixed radius $r>0$, the function $k \mapsto W_{k,2}(X_0,u,r)$ is continuous, we infer as $k \to \overline{k}^-$ that $W_{\overline{k},2}(X_0,u,r)\geq 0$, which implies by continuity that $W_{\overline{k},2}(X_0,u,0^+)\geq 0$.\\ Thus, for any $\overline{r}\in (0,\mathrm{dist}(X_0,\partial^+ B^+))$ we get
  $$
  0 \leq \int_0^{\overline{r}}{\frac{d}{dr}W_{\overline{k},2}(X_0,u,s) \mathrm{d}s} = W_{\overline{k},2}(X_0,u,\overline{r})- W_{\overline{k},2}(X_0,u,0^+) < +\infty.
  $$
  On the other side, by \eqref{remainder} we deduce
  \begin{equation}\label{would}
  \int_0^{\overline{r}}{\frac{1}{s}\left(\frac{1}{s^{n+a-1+2\overline{k}}}
\int_{\partial^0 B^+_s(X_0)}{F_{\lambda_+,\lambda_-}(u)\mathrm{d}x}\right)\mathrm{d}s} < +\infty,
  \end{equation}
  which implies, combined with the non-integrability of $s\mapsto s^{-1}$ in $0$, that if
  $$
  \liminf_{r\mapsto 0^+}\frac{1}{r^{n+a-1+2\overline{k}}}
\int_{\partial^0 B^+_r(X_0)}{F_{\lambda_+,\lambda_-}(u)\mathrm{d}x}>0,
  $$
  then \eqref{would} would not be true. Thus, this implication suggests that the previous liminf has to be null.
\end{proof}
\begin{proof}[Proof of Proposition \ref{claim}]
  The proof is based on a blow-up argument: indeed, assuming that $\overline{k}> 2s/(2-q)$, for $X_0 \in \Gamma(u)$ let us consider the sequence $(r_n)_n$ introduced in Lemma \ref{r} and the associated blow-up sequence
  $$
  u_n(X) = \frac{u(X_0+r_nX)}{\sqrt{H(X_0,u,r_n)}} \quad\mbox {for } X \in B_{R/r_n}^+
  $$
  where $R= \mathrm{dist}(X_0,\partial^+ B^+)$. Thanks to Lemma \ref{trans1}, we have $H(X_0,u,r_n)>0$ and $E_2(X_0,u,r_n)\geq 0$, which lead to
  $$
  \int_{\partial^+ B^+_1}{y^a u_n^2 \mathrm{d}\sigma}= 1 \quad\mbox{and}\quad \int_{B^+_1}{y^a \abs{\nabla u_n}^2 \mathrm{d}X} = \ddfrac{\frac{1}{r_n^{n+a-1}}\int_{ B^+_{r_n}(X_0)}{y^a \abs{\nabla u}^2\mathrm{d}X}}{\frac{1}{r_n^{n+a}}\int_{\partial^+ B^+_{r_n}(X_0)}{y^a u^2\mathrm{d}\sigma}}.
  $$
  On the other hand by Lemma \ref{trans1} we deduce
  $$
  \int_{B^+_{r_n}(X_0)}{y^a\abs{\nabla u}^2\mathrm{d}X}\geq \frac{2}{q}\int_{\partial^0 B^+_{r_n}(X_0)}{F_{\lambda_+,\lambda_-}(u)\mathrm{d}x},
  $$
  which implies, since $\tilde{t}<2$ that
  \begin{align*}
  \frac{1}{r_n^{n+a-1}}\int_{B^+_{r_n}(X_0)}{y^a\abs{\nabla u}^2\mathrm{d}X} &\leq \frac{2}{2-\tilde{t}}  \frac{1}{r_n^{n+a-1}}\left(\int_{B^+_{r_n}(X_0)}{y^a\abs{\nabla u}^2\mathrm{d}X}-\frac{\tilde{t}}{q}\int_{\partial^0 B^+_{r_n}(X_0)}{F_{\lambda_+,\lambda_-}(u)\mathrm{d}x}\right)\\ &\leq \frac{2}{2-\tilde{t}}E_{\tilde{t}}(X_0,u,r_n).
  \end{align*}
  As a consequence of the previous estimates and Lemma \ref{trans2}, we get
  $$
  \int_{B^+_1}{y^a \abs{\nabla u_n}^2 \mathrm{d}X} \leq \frac{2}{2-\tilde{t}}N_{\tilde{t}}(X_0,u,r_n) \leq C.
  $$
  Thus, since the sequence $(u_n)_n$ is uniformly bounded in $H^{1,a}(B^+_1)$, the compactness of the Sobolev embedding implies that $(u_n)_n$ converges weakly in $H^{1,a}(B_+^1)$ and strongly in $L^{2,a}(\partial^+ B_1^+)$ to a function $\overline{u}\in H^{1,a}(B^1_+)$.\\
  Moreover, since by \cite[Theorem 2.11]{Nekvinda} the traces of functions in $H^{1,a}(B^+)$ on the set $\partial^0B^+$ coincides with the Sobolev-Slobodeckij space $H^s(\partial^0B^+)$ and, since $\partial^0 B^+$ itself is a Lipschitz domain with bounded boundary, by the Riesz–Frechet–Kolmogorov theorem, the trace operator
  \[
  H^{1,a}(B^+_1) \hookrightarrow\hookrightarrow L^p(\partial^0 B^+_1)
  \]
  is well defined an compact for every $p\in [1,2]$ (see Lemma \ref{dafare}). Hence, since $q \in [1,2)$, we get
  \begin{equation}\label{limit}
  \int_{\partial^+ B^+_1}{y^a \overline{u}^2\mathrm{d}\sigma} = 1 \quad\mbox{and}\quad \lim_{n\to \infty}\int_{\partial^0 B^+_1}{F_{\lambda_+,\lambda_-}(u_n)\mathrm{d}x} = \int_{\partial^0 B^+_1}{F_{\lambda_+,\lambda_-}(\overline{u})\mathrm{d}x}.
  \end{equation}
  Since the first equality implies that $\overline{u}\not\equiv 0$ on $\partial^+ B^+_1$, we deduce by the trace embedding that  $\overline{u}\not\equiv 0$ on the whole $B^+_1$.    
  On the other side, we get
$$
  \int_{\partial^0 B^+_1}{F_{\lambda_+,\lambda_-}(u_n)\mathrm{d}x} =   \left(\frac{r^{(a-1+2\overline{k})/q}}{\sqrt{H(X_0,u,r_n)}}\right)^{q} \frac{1}{r_n^{n+a-1+2\overline{k}}} \int_{\partial^0 B^+_{r_n}(X_0) }{F_{\lambda_+,\lambda_-}(u)\mathrm{d}x}
  $$
  Since we are assuming $\overline{k}> 2s/(2-q)$, we have $2(2\overline{k}+a-1)/q > 2 \overline{k}$ and, for $n$ sufficiently large, we get
  $$
  \int_{\partial^0 B^+_1}{F_{\lambda_+,\lambda_-}(u_n)\mathrm{d}x} \leq \left(\frac{H(X_0,u,r_n)}{r_n^{2\overline{k}}} \right)^{-q/2}\frac{1}{r_n^{n+a-1+2\overline{k}}} \int_{\partial^0 B^+_{r_n}(X_0) }{F_{\lambda_+,\lambda_-}(u)\mathrm{d}x}
  $$
  where the right hand side goes to $0$ as $n\mapsto +\infty$ by Lemma \ref{trans1} and Lemma \ref{r}. By \eqref{limit} we infer that
  $$
  \int_{\partial^0 B^+_1}{F_{\lambda_+,\lambda_-}(\overline{u})\mathrm{d}x} = 0\ \longleftrightarrow \ \overline{u} \equiv 0 \ \mbox{ on }\ \partial^0 B^+_1.
  $$
  By standard argument, since $(u_n)_n$ is uniformly bounded in $H^{1,a}(B_1^+)$ and $u_n \rightharpoonup \overline{u}$ weakly in $H^{1,a}$, from
  $$
-\partial^a_y u_n = \left(\ddfrac{r_n^{2\overline{k}}}{H(X_0,u,r_n)} \right)^{\frac{2-q}{2}}\left(\lambda_+ (u_n)_+^{q-1} - \lambda_- (u_n)_-^{q-1} \right) \quad \mbox{on }\, \frac{\partial^0 B^+_R-X_0}{r_n},
$$ we deduce that
  the limit function $\overline{u} \in H^{1,a}_\loc(\R^{n+1}_+)$ is a weak solution of
  \begin{equation}\label{zero.sharm}
  \begin{cases}
  L_a \overline{u}=0 & \mbox{in } \R^{n+1}_+ \\
  \partial^a_y \overline{u} = 0 &\mbox{on } \R^{n}\times \{0\}\\
  \overline{u} = 0 &\mbox{on } \R^{n}\times \{0\},
  \end{cases}
  \end{equation}
  such that $\overline{u}\not\equiv 0$ on $\R^{n+1}_+$. The contradiction follows immediately by the unique continuation principle for the traces of $L_a$-harmonic functions, see \cite[Proposition 5.9]{STT2020}.
\end{proof}
The following result allows to characterize the $H^{1,a}$-vanishing order in terms of the Weiss-type functional in the case $\mathcal{O}(u,X_0)\geq 2s/(2-q)$. More precisely, it implies that the solutions of \eqref{system} can vanish with order less or equal than $\overline{k}=2s/(2-q)$, with $\overline{k}$ the transition exponent.
\begin{proposition}\label{2s/(2-q)}
  Let $u$ be a solution of \eqref{system} and $X_0 \in \Gamma(u)$ such that $\mathcal{O}(u,X_0)\geq 2s/(2-q)$. Then, the vanishing order $\mathcal{O}(u,X_0)$ is characterized by
  $$
  \mathcal{O}(u,X_0)= \inf\left\{\ k >0 \colon W_{k,2}(X_0,u,0^+) = -\infty \right\} = \frac{2s}{2-q}.
  $$
  Furthermore, we get
  $$
  \begin{cases}
    W_{k,2}(X_0,u,0^+)=0 & \mbox{if }\, 0<k< \mathcal{O}(u,X_0) \\
    W_{k,2}(X_0,u,0^+)=-\infty & \mbox{if }\, k > \mathcal{O}(u,X_0).
  \end{cases}  $$
\end{proposition}
\begin{proof}
The proof of this result follows the one of its local counterpart in \cite{soavesublinear}. For the sake of simplicity, let us denote with $\norm{\cdot}{H^{1,a}(B^+_r(X_0))} = \norm{\cdot}{X_0,r}$. Now, fixed $X_0 \in \Gamma(u)$, let us prove that
  \be\label{eq.2s/(2-q)}
  \liminf_{r\to 0^+} \frac{\norm{u}{X_0,r_n}^2}{r^{2\overline{k}}} >0,
\ee
  where $\overline{k}=2s/(2-q)$. After that, the result will follow by Proposition \ref{claim} and \eqref{transition.up}.\\
By contradiction, let us suppose there exists a sequence $r_n\to 0^+$ such that
  \begin{equation}\label{abs}
  \lim_{n\to \infty}   \frac{\norm{u}{X_0,r_n}^2}{r^{2\overline{k}}_n}=0.
    \end{equation}
  Then, consider the blow-up sequence associated to the $H^{1,a}$-norm, defined as
  \begin{equation}\label{unif.norm}
  u_r(X) = \frac{u(X_0 + r X)}{\norm{u}{X_0,r}},\quad \mbox{such that }\norm{u_n}{0,1} = 1.
  \end{equation}
  As we deduce in the proof of Proposition \ref{claim}, since the blow-up sequence $(u_n)_n$ is uniformly bounded in $H^{1,a}(B^+_1)$, the compactness of the Sobolev embedding implies that $(u_n)_n$ converges weakly in $H^{1,a}(B_+^1)$ and strongly in $L^{2,a}(\partial^+ B_1^+)$ to a function $\overline{u}\in H^{1,a}(B^1_+)$. Moreover, the traces on $\partial^0 B^+_1$ converge strongly in $L^q(\partial^0 B^+_1)$, for every $q\in [1,2)$. In particular,
  \begin{align*}
  \lim_{n\to \infty} W_{\overline{k},2}(X_0,u,r_n) = &\,\lim_{n\to \infty}\left[ \frac{\norm{u}{X_0,r_n}^2}{r_n^{2\overline{k}}}\left(\int_{B^+_1}{y^a \abs{\nabla u_n}^2\mathrm{d}X} - \overline{k}\int_{\partial^+ B^+_1}{y^a u_n^2\mathrm{d}\sigma}\right)+\right. \\
  &\, \left.- \frac{2}{q}\left(\frac{\norm{u}{X_0,r_n}^2}{r_n^{2\overline{k}}}\right)^{q/2} \int_{\partial^0 B^+_1}{F_{\lambda_+,\lambda_-}(u_n)\mathrm{d}x}\right]= 0.
  \end{align*}
  Thus, since the limit $W_{\overline{k},2}(X_0,u,0^+)$ exists, by the monotonicity result in Proposition \ref{weiss.mon}, we get that $W_{\overline{k},2}(X_0,u,r)\geq 0$ and $W_{\overline{k},q}(X_0,u,r)\geq 0$ for every $r \in (0,\mathrm{dist}(X_0,\partial^+ B^+_1))$ and $q< 2$. First, by Lemma \ref{trans1}, we know that $E_2(X_0,u,r)\geq 0$ and $H(X_0,u,r)>0$ for every $r \in (0,\mathrm{dist}(X_0,\partial^+ B^+_1))$ and, for every $k > \overline{k}$ we get
  \begin{equation}\label{liminf}
  \liminf_{r\to 0^+} \frac{H(X_0,u,r)}{r^{2k}} = +\infty.
  \end{equation}
Now, let us compute the same limit in the case $\overline{k}$. Since the function $r \mapsto H(X_0,u,r)/r^{2\overline{k}}$ is monotone non-decreasing, there exists the limit as $r\to 0^+$ and, by \eqref{abs}, we get
$$
0\leq \frac{H(X_0,u,r_n)}{r_n^{2\overline{k}}} \leq \frac{\norm{u}{X_0,r_n}^2}{r_n^{2\overline{k}}} \to 0
$$
which implies
\begin{equation}\label{sppos}
\lim_{r\to 0^+} \frac{H(X_0,u,r)}{r^{2\overline{k}}} =0.
\end{equation}
In order to reach a contradiction, in the end of the proof we will prove that the blow-up limit satisfies $\overline{u}\equiv 0$, in contradiction with the normalization \eqref{unif.norm}.
\end{proof}
\begin{lemma}\label{nic.lemma}
  Fixed $X_0 \in \Gamma(u)$ and $\overline{k}=2s/(2-q)$ let us suppose that \eqref{sppos} holds true.
  Then, we get
  \begin{equation}\label{nic}
  \liminf_{r \to 0^+} \frac{W_{\overline{k},q}(X_0,u,r) r^{2\overline{k}}}{H(X_0,u,r)} = 0 \quad \mbox{and}\quad \lim_{r \to 0^+}\frac{W_{\overline{k},2}(X_0,u,r) r^{2\overline{k}}}{H(X_0,u,r)} = 0.
  \end{equation}
\end{lemma}
\begin{proof}
  Let us consider first the limit associated to the case $t=q$ and, by contradiction, assume that $\varepsilon >0$ and $r_0 \in (0,\mathrm{dist}(X_0,\partial^+ B^+_1))$ such that
  $$
  \frac{W_{\overline{k},q}(X_0,u,r) r^{2\overline{k}}}{H(X_0,u,r)}  \geq \varepsilon \quad \mbox{for every } r \in (0,r_0).
  $$
  By \eqref{h.derivative}, we deduce that
  $$
  \frac{d}{dr} \log \left( \frac{H(X_0,u,r)}{r^{2\overline{k}}} \right)= \frac{2}{r}\frac{W_{\overline{k},q}(X_0,u,r) r^{2\overline{k}}}{H(X_0,u,r)} \geq \frac{2 \varepsilon}{r},
  $$
  and integrating by parts the previous inequality between $r \in (0,r_0)$ and $r_0$ we  get
  $$
  \frac{H(X_0,u,r)}{r^{2\overline{k}+2\varepsilon}} \leq   \frac{H(X_0,u,r_0)}{r_0^{2\overline{k}+2\varepsilon}} < \infty \quad \mbox{for every }r \in (0,r_0).
  $$
  In particular
  $$
  \limsup_{r\to 0^+} \frac{H(X_0,u,r)}{r^{2\overline{k}+2\varepsilon}} <+\infty,
  $$
  in contradiction with \eqref{liminf} with $k = \overline{k}+\varepsilon$.\\
  Now, for $t=2$ and $\overline{k}=2s/(2-q)$ we already know by Proposition \ref{weiss.mon} that
  \[
  \frac{d}{dr}W_{\overline{k},2}(X_0,u,r) =  \frac{2}{r^{n+a-1+2\overline{k}}}\int_{\partial^+ B^+_r(X_0)}{y^a \left(\partial_r u -\frac{\overline{k}}{r}u\right)^2 \mathrm{d}\sigma}.
  \]
  In the remaining part of the proof, for the sake of simplicity we omit the dependence with respect to $u$ and $X_0$. Hence, combining the previous derivative with \eqref{mon.H.weiss} we get
  \begin{align*}
  \left(\frac{H(r)}{r^{2\overline{k}}}\right)^2 \frac{d}{dr}\left(\frac{r^{2\overline{k}}W_{\overline{k},2}(r)}{H(r)} \right) 
   &=\frac{H(r)}{r^{2\overline{k}}} \frac{d}{dr}W_{\overline{k},2}(r) - \frac{2}{r}W_{\overline{k},2}(r) W_{\overline{k},q}(r),
  \end{align*}
and since $0\leq W_{\overline{k},2}(r) \leq W_{\overline{k},q}(r)$ we infer that
\begin{equation*}
\begin{split}
&\left(\frac{H(r)}{r^{2\overline{k}}}\right)^2 \frac{d}{dr}\left(\frac{r^{2\overline{k}}W_{\overline{k},2}(r)}{H(r)}\right) \geq\\
&\qquad \geq \frac{2}{r^{2n+2a-1+4\overline{k}}} \int_{\partial^+ B^+_r}{y^a u^2 \mathrm{d}\sigma}\int_{\partial^+ B^+_r}{y^a \left(\partial_r u -\frac{\overline{k}}{r}u\right)^2 \mathrm{d}\sigma} - \frac{2}{r}\left(W_{\overline{k},q}(r)\right)^2 \\
& \qquad\geq \frac{2}{r^{2n+2a-1+4\overline{k}}} \left[
\int_{\partial^+ B^+_r}{y^a u^2\mathrm{d}\sigma}\int_{\partial^+ B^+_r}{y^a (\partial_r u)^2\mathrm{d}\sigma}-\left(\int_{\partial^+ B^+_r}{y^a u \partial_r u} \mathrm{d}\sigma\right)^2 \right],
\end{split}
\end{equation*}
which is non-negative by the Cauchy-Schwarz inequality. Since $H(r)>0$ and $0\leq W_{\overline{k},2}(r) \leq W_{\overline{k},q}(r)$, the previous part of the proof yields that the second limit in \eqref{nic} exists and is equal to zero.
  \end{proof}
  \begin{proof}[Conclusion of the proof of Proposition \ref{2s/(2-q)} ]
    Since $\norm{u}{X_0,r}^2 \geq H(X_0,u,r)$, by Lemma \ref{nic.lemma}, there exists a sequence $r_m\to 0^+$ such that
    \begin{equation}\label{nic2}
    \lim_{m\to \infty}\frac{r_m^{2\overline{k}}W_{\overline{k},q}(X_0,u,r_m)}{\norm{u}{X_0,r_m}^2} =     \lim_{m\to \infty}\frac{r_m^{2\overline{k}}W_{\overline{k},2}(X_0,u,r_m)}{\norm{u}{X_0,r_m}^2}=0.
  \end{equation}
    Now, fixed $u_m$ the blow-up sequence in \eqref{unif.norm} associated to the sequence $(r_m)_m$, we already know by the $H^{1,a}$-normalization that $(u_m)_m$ converges weakly in $H^{1,a}(B^+_1)$ and strongly in $L^q(\partial^0 B^+_1)$ to a limit function $\overline{u}$. First, by \eqref{nic} we infer
    \begin{align*}
    0\leq \left(\frac{r_m^{2\overline{k}} }{H(r_m)}\right)^{\frac{2-q}{2}}\int_{\partial^0 B^+_1}{F_{\lambda_+,\lambda_-}(u_m)\mathrm{d}x}
    &\leq \frac{r_m^{2s} \norm{u}{X_0,r_m}^q}{H(r_m)}\int_{\partial^0 B^+_1}{F_{\lambda_+,\lambda_-}(u_m)\mathrm{d}x}\\
    &= \frac{1}{r_m^{n+a-1}H(r_m)}\int_{\partial^0 B^+_{r_m} (X_0)}{F_{\lambda_+,\lambda_-}(u)\mathrm{d}x}\\
     &= \frac{q}{2-q} \frac{r_m^{2\overline{k}}\left(W_{k,q}(r_m)-W_{k,2}(r_m) \right)}{H(r_m)}\to 0^+,
    \end{align*}
    which implies by the strong convergence and \eqref{sppos} that $\overline{u} \equiv 0$ on $\partial^0 B^+_1$. On the other side, by \eqref{nic2} we deduce that
    \begin{align*}
      0&\leq \frac{r_m^{2s}}{\norm{u}{X_0,r_m}^{2-q}}\int_{\partial^0 B^+_1}{F_{\lambda_+,\lambda_-}(u_m)\mathrm{d}x}\\ &= \frac{1}{r_m^{n+a-1}\norm{u}{X_0,r_m}^{2}}\int_{\partial^0 B^+_{r_m}(X_0)}{F_{\lambda_+,\lambda_-}(u) \mathrm{d}x}\\
      &= \frac{q}{2-q} \frac{r_m^{2\overline{k}}\left(W_{k,q}(X_0,u,r_m)-W_{k,2}(X_0,u,r_m) \right)} {\norm{u}{X_0,r_m}^{2}}\to 0^+,
    \end{align*}
    as $m\to +\infty$. Therefore, collecting the previous result we get
    \begin{align*}
    0 &= \lim_{m\to \infty}  \frac{r_m^{2\overline{k}}W_{\overline{k},2}(X_0,u,r_m)}{\norm{u}{X_0,r_m}^2}\\
    &= \lim_{m \to \infty} \left(\int_{B^+_1}{y^a \abs{\nabla u_m}^2\mathrm{d}X} - \frac{2r_m^{1-a}}{q\norm{u}{X_0,r_m}^{2-q}}\int_{\partial^0 B^+_1}{F_{\lambda_+,\lambda_-}(u_m)\mathrm{d}x}-\overline{k}\int_{\partial^+ B^+_1}{y^a u_m^2\mathrm{d}\sigma}\right)\\
    &= \lim_{m \to \infty} \left(\int_{B^+_1}{y^a \abs{\nabla u_m}^2\mathrm{d}X} -\overline{k}\int_{\partial^+ B^+_1}{y^a u_m^2\mathrm{d}\sigma}\right),
    \end{align*}
    which implies that $\norm{u_m}{0,1}^2 \to (\overline{k}+1)\norm{\overline{u}}{L^{2,a}(\partial^+ B^+_1)}^2$. Since by \eqref{unif.norm} $\norm{u_m}{0,1}=1$ for every $m$, we immediately deduce that $\overline{u} \not \equiv 0$ in $B^+_1$. Finally, the conclusion follows as in the proof of Proposition \ref{claim}.
  \end{proof}
  \section{Blow-up analysis for $\mathcal{O}(u,X_0)< k_q$}\label{5}
  In this Section we initiate the blow-up analysis of the nodal set starting from those points with vanishing order smaller than the critical value $k_q=2s/(2-q)$. The main idea is to develop a blow-up argument based on the validity of two Almgren-type monotonicity formulas, which provide a Taylor expansion of the solutions near the nodal set in terms of $L_a$-harmonic polynomials symmetric with respect to $\{y=0\}$. As noticed in \cite{STT2020}, this class of polynomials can be treated as the fractional counterpart of harmonic polynomials.\\
  In the following result, using the upper bound on the $H^{1,a}$-vanishing order of $u$, we prove the validity of a monotonicity result for the Almgren-type functional $N(X_0,u,r)=N_q(X_0,u,r)$ introduced in \eqref{almgren}.
\begin{proposition}\label{almgren.lower}
Let $K \subset \subset \partial^0 B^+_1$ and suppose there exists $\delta >0$ such that
\be\label{vanishing.bound}
\mathcal{O}(u,X_0) \leq k_q -\delta \quad \text{for every $X_0 \in \Gamma(u) \cap K$}.
\ee
Then there exists $r_0>0$ such that for every $X_0 \in \Gamma(u)\cap K$
$$
r \mapsto e^{\tilde{C} r^\alpha}\!\left(N(X_0,u,r)+1\right)
$$
is monotone non-decreasing for $r \in (0,\min(r_0,\mathrm{dist}(K,\partial^0 B^+)))$, for some constant $\alpha = \alpha(\delta,n,s,q)$ and $\tilde C=\tilde C(\delta,n,s,q)$. Moreover, for every $X_0 \in \Gamma(u)$ such that $\mathcal{O}(u,X_0)< k_q$ there exits the limit
$$
N(X_0,u,0^+)= \lim_{r \to 0^+}e^{\tilde{C} r^\alpha}\left(N(X_0,u,r)+1\right)-1
$$
and the map $X_0\mapsto N(X_0,u,0^+)$ is upper semi-continuous on $\Gamma(u)$.
\end{proposition}
\begin{proof}
Let $K\subset\subset \partial^0 B^+_1$ and $\alpha>0$ to be made precise later. Let $X_0 \in K$ and, for the sake of simplicity, we omit the dependence of the functionals with respect to $u$ and $X_0$. By Corollary \ref{N.derivative}, we easily get
\begin{align}\label{deriv.N1}
\begin{aligned}
\frac{d}{dr}\log ( N(r) +1 ) 
& \geq \frac{1}{E(r) + H(r)}\frac{1}{r^{n+a-1}}\left[\frac{2-q}{q}\int_{S^{n-1}_r}F_{\lambda_+,\lambda_-}(u)\mathrm{d}\sigma - \frac{C^s_{n,q}}{qr} \int_{\partial^0 B^+_r}{F_{\lambda_+,\lambda_-}(u)\mathrm{d}x}\right]\\
& \geq -\frac{C^s_{n,q}}{q(E(r) + H(r))}\frac{1}{r^{n+a}}\int_{\partial^0 B^+_r}{F_{\lambda_+,\lambda_-}(u)\mathrm{d}x}
\end{aligned}
\end{align}
with $C^s_{n,q} = 2n - q(n-2s)$. By Lemma \ref{lem.poin}, we get
\begin{align}\label{upper E+H}
\begin{aligned}
E(r) + H(r) &\geq \norm{u}{r}^q \left(\norm{u}{r}^{2-q} - C_1 r^{2s}\right)\\
& \geq \frac{C}{r^{n}}\left(\norm{u}{r}^{2-q} - C_1 r^{2s}\right)\int_{\partial^0 B^+_r}{\abs{u}^q\mathrm{d}x},
\end{aligned}
\end{align}
where $\norm{\cdot}{r}=\norm{\cdot}{H^{1,a}(B^+_r)}$. Now, we want to show that there exists $\alpha, r_0, C_2>0$ such that
\be\label{utiledopo}
\ddfrac{\norm{u}{r}^{2-q}}{r^{2s}} - C_1 > C_2\frac{1}{r^\alpha},
\ee
for every $r\in (0,r_0)$. Then, combining the previous inequality with \eqref{deriv.N1} and \eqref{upper E+H}, we will get
$$
  \frac{d}{dr}\log(N(r)+1) \geq -\ddfrac{\tilde{C}}{r\left(\frac{\norm{u}{r}^{2-q}}{r^{2s}}-C_1\right)}\geq -\tilde{C} r^{\alpha-1},
$$
as we claimed. First, by \eqref{vanishing.bound}, let us choose $\alpha = \delta/2$ and consider $$k_2=\frac{2s}{2-q}-\alpha \geq \mathcal{O}(u,X_0),$$
for every $X_0 \in \Gamma(u)\cap K$.
Indeed, by the definition of $H^{1,a}$-vanishing order, there exists $r_2>0$ and $C_2>0$ such that, for every $r\in (0,r_2)$
\be \label{alpha}
\norm{u}{r} \geq C_2 r^{k_2} \longleftrightarrow \ddfrac{\norm{u}{r}^{2-q}}{r^{2s}} \geq C_2 r^{(2-q)k_2-2s} = C_2 r^{-\alpha}.
\ee
Since $\delta=\delta(K)$, the constant $C_2,\alpha$ and $r_2$ depend only on the choice of the compact $K$. Finally, the upper semi-continuity follows by a standard argument.
\end{proof}
Using this monotonicity result we can prove the equivalence between the notion of $H^{1,a}$-vanishing order $\mathcal{O}(u,X_0)$ and the one introduced in Definition \ref{nu}.
\begin{corollary}\label{equivalence1}
  Let $X_0 \in \Gamma(u)$ be such that $\mathcal{O}(u,X_0)<k_q$. Then
  $$
  \mathcal{O}(u,X_0) = \mathcal{V}(u,X_0).
  $$
\end{corollary}
\begin{proof}
Suppose by contradiction that $\mathcal{O}(u,X_0) < \mathcal{V}(u,X_0)$ and consider $k \in (\mathcal{O}(u,X_0),\mathcal{V}(u,X_0))$. Let us write
$$
k = \frac{2s-\alpha}{2-q},
$$
for some $\alpha>0$. Now, let $r \in (0,\mathrm{dist}(X_0,\partial^0 B^+_1))$, by \eqref{upper E+H} we get
  \begin{equation}\label{compare1}
  \norm{u}{X_0,r}^q \left(\norm{u}{X_0,r}^{2-q} - C_1 r^{2s}\right) \leq E(X_0,u,r)+H(X_0,u,r)=H(X_0,u,r)(N(X_0,u,r)+1)
  \end{equation}
  which implies
  \begin{equation}\label{compare2}
  \frac{\norm{u}{X_0,r}^2}{r^{2k}} \leq \left[\ddfrac{N(X_0,u,r)+1}{\norm{u}{X_0,r}^{2-q}-C_1r^{2s}} r^{k(2-q)}\right]^{2/q} \left(\frac{H(X_0,u,r)}{r^{2k}}\right)^{2/q}.
  \end{equation}
  As in \eqref{alpha}, in the proof of Proposition \ref{almgren.lower}, there exists $r_0>0$ and $C_0>0$ such that
  $$
  \norm{u}{X_0,r}^{2-q}-C_1r^{2s} \geq C_0 r^{2s-\alpha} = C_0 r^{k(2-q)},
  $$
  for every $r\in (0,r_0)$. With a slight abuse of notations, it is not restrictive to assume that $r_0$ corresponds to the radius introduced in Proposition \ref{almgren.lower}.\\ Finally, by the monotonicity result, fixed $R=\min\{r_0,\mathrm{dist}(X_0,\partial^0 B^+)\}$ we deduce, for every $r \in (0,R)$, that
  \begin{align*}
\frac{\norm{u}{X_0,r}^2}{r^{2k}} &\leq C \left[(N(X_0,u,r)+1)\right]^{2/q} \left(\frac{H(X_0,u,r)}{r^{2k}}\right)^{2/q}\\
& \leq C \left[e^{\tilde{C}R}(N(X_0,u,R)+1)\right]^{2/q}  \left(\frac{H(X_0,u,r)}{r^{2k}}\right)^{2/q}
  \end{align*}
where $C>0$ depends only on $C_0$. Thus, by Definition \ref{nu} we get that $\mathcal{O}(u,X_0)\geq \mathcal{V}(u,X_0)$ that, in combination with the opposite inequality, implies the desired result.
\end{proof}
Similarly, we show that in the case $\mathcal{V}(u,X_0)< 2s/(2-q)$, the possible vanishing orders correspond to the possible limits of the Almgren-type frequency formula. For the sake of completeness, we report the proof of this result which is deeply based on the validity of the Almgren-type monotonicity result.
\begin{corollary}\label{equivalence2}
    Let $X_0 \in \Gamma(u)$ be such that $\mathcal{V}(u,X_0)<k_q$. Then $\mathcal{V}(u,X_0) = N(X_0,u,0^+).$
\end{corollary}
\begin{proof}
  By \eqref{E.H} and Definition \ref{nu}, we claim that
$$
  \limsup_{r\to 0^+} \frac{H(X_0,u,r)}{r^{2k}} = \begin{cases}
      0, & \mbox{if } 0< k < N(X_0,u,0^+)\\
      +\infty, & \mbox{if } k >N(X_0,u,0^+).
    \end{cases}
$$
It is not restrictive to assume that $X_0=0$ and $r \in (0,R)$, for some $R>0$ that will be choose later. By definition of $r\mapsto H(0,u,r)=H(u,r)$ we immediately get for every $r \in (0,R)$ that
\be\label{derivative.H}
\frac{d}{dr}\log H(u,r) = \frac2r N(u,r)
\ee
and in particular for every $k>0$, by Proposition \ref{almgren.lower}, there exists $\alpha, \tilde{C}>0$ such that
\be\label{van.1}
\left(\frac{H(u,R)}{R^{2\overline{N}}}\right) r^{2(\overline{N}-k)} \leq \frac{H(u,r)}{r^{2k}} \leq \left(\frac{H(u,R)}{R^{2\underline{N}}}\right) r^{2(\underline{N}-k)},
\ee
with
$$
\underline{N} = e^{-\tilde{C}R^\alpha}(N(u,0^+)+1)-1 \quad\mbox{and}\quad
\overline{N} = e^{\tilde{C}R^\alpha}(N(u,R)+1)-1.
$$
Suppose first $\mathcal{V}(u,0)< N(u,0^+)$, so there exists $\eps>0$ such that $k := N(u,0^+) - \eps >\mathcal{V}(u,0)$. Let $R>0$ be such that
$$
(1-e^{-\tilde{C}R^\alpha})(N(u,0^+)+1) < \frac\eps2,
$$
where $\tilde{C},\alpha>0$ are introduced in Proposition \ref{almgren.lower}. Thus, we get $\underline{N}-k >\eps/2 $ and consequently by \eqref{van.1}
$$
\frac{H(u,r)}{r^{2k}} \leq \left(\frac{H(u,R)}{R^{2\underline{N}}}\right) r^{2(\underline{N}-k)}< C_2 r^\eps,
$$
for some constant $C_2>0$ depending only on $R>0$. The absurd follows immediately since $k>\mathcal{V}(u,0)$, namely
$$
+\infty= \limsup_{r\to 0^+}\frac{H(u,r)}{r^{2k}} \leq \left(\frac{H(u,R)}{R^{2\underline{N}}}\right) r^{2(\underline{N}-k)}< C_2 \lim_{r \to 0^+} r^\eps = 0.
$$
Similarly, if $\mathcal{V}(u,0)> N(u,0^+)$ consider $k=N(u,0^+)+\eps$, with $\eps>0$ sufficiency small so that $\mathcal{V}(u,0) > k$.
By the monotonicity result Proposition \ref{almgren.lower}, let $R>0$ be such that
$$
e^{\tilde{C}R^\alpha}(N(u,R)+1)- (N(u,0^+)+1) < \frac{\eps}{2}.
$$
Hence, since $\overline{N}-k <  -\eps/2$, we get by \eqref{van.1}
$$
\frac{H(u,r)}{r^{2k}} \geq \left(\frac{H(u,R)}{R^{2\overline{N}}}\right) r^{2(\overline{N}-k)}\geq C_2 r^{-\eps}
$$
for some constant $C_2>0$ depending only on $R>0$. The contradiction follows by Definition \ref{nu}.
\end{proof}
In particular, from the previous equivalences we deduce that, for every $k_1<N(X_0,u,0^+)<k_2$ there exist $C_1,C_2>0$ such that
\be\label{equiv}
C_2 r^{2k_2} \leq \norm{u}{X_0,r}^2 \leq C_1 r^{2k_1},
\ee
for $r \in (0,R)$, for some $R>0$ sufficiently small.\\

Finally, we can introduce the following notion of stratum of the nodal set.
\begin{definition}
Let $k< k_q$ we define
$$
\Gamma_k(u)  \coloneqq  \{X_0 \in \Gamma(u) \colon \mathcal{O}(u,X_0)=k\},
$$
where $\mathcal{O}(u,X_0)=\mathcal{V}(u,X_0)=N(X_0,u,0^+)$.
\end{definition}

While in the local case, in \cite{soaveweth, soavesublinear} the authors proved the existence of a generalized Taylor expansion of the solution near the nodal set by applying an iteration argument based on the results of \cite{fermi}, we apply a blow-up analysis in order to understand how the solutions behave near the nodal set $\Gamma(u)$.\\
Thus, we start by proving a convergence result for blow-up sequences based on the validity of the Almgren-type monotonicity formula. Hence, given $X_0 \in \Gamma(u)$, for any $r_k \downarrow 0^+$, we define as normalized  blow-up sequence
    \begin{equation*}\label{blow.up.almgren}
    u_k(X)= \ddfrac{u(X_0 + r_k X)}{\sqrt{H(X_0,u,r_k)}}\quad \mbox{for } X\in B^+_{X_0,r_k}=\frac{B_1^+ - X_0}{r_k},
    \end{equation*}
    such that
\be\label{blow.up.equation}
\begin{cases}
-L_a u_k=0 & \mbox{in } B_{X_0,r_k}^+\\
  -\partial^a_y u_k = \left(\ddfrac{r_k^{2s/(2-q)}}{\sqrt{H(X_0,u,r_k)}}\right)^{2-q} \left[\lambda_+ (u_k)_+^{q-1} - \lambda_- (u_k)_-^{q-1}\right] & \mbox{on } \partial^0 B^+_{X_0,r_k}.
\end{cases}
\ee
Let us introduce the notation
$$
0 < \alpha_{k} =  \left(\ddfrac{r_k^{2s/(2-q)}}{\sqrt{H(X_0,u,r_k)}}\right)^{2-q}< +\infty,
$$
Since we are considering the case $\mathcal{O}(u,X_0)< 2s/(2-q)$, the sequence $(\alpha_k)_k$ is bounded and converges to $0$ as $k \to \infty$.
\begin{theorem}\label{blowup.lower}
Let $X_0 \in \Gamma(u)$ be such that $\mathcal{O}(u,X_0)< k_q$ and $u_k$ be a normalized blow-up sequence centered in $X_0$ and associated with some $r_k \downarrow 0^+$. Then, there exists $p \in H^{1,a}_{\loc}(\R^{n+1})$ such that, up to a subsequence, $u_k\to p$ in $C^{0,\alpha}_{\loc}(\R^{n+1})$ for every $\alpha \in (0,1)$ and strongly in $H^{1,a}_{\loc}(\R^{n+1})$. In particular, the blow-up limit satisfy
\be\label{La.even}
\begin{cases}
-L_a p=0 & \mbox{in } \R^{n+1}_+\\
  -\partial^a_y p = 0 & \mbox{on } \R^n \times \{0\}.
\end{cases}
\ee
\end{theorem}
The proof will be presented in a series of lemmata.
\begin{lemma}\label{bounded.H1}
  Let $X_0 \in \Gamma(u)$ such that $\mathcal{O}(u,X_0)<k_q$. For any given $R>0$, we have  $$\norm{u_{k}}{R}\leq C
$$
where $C>0$ is a constant independent on $k>0$. Moreover $u_k \to p$ strongly in $H^{1,a}(B^+_R)$ for every $R>0$, for some $p\in H^{1,a}_\loc(\R^{n+1})$ such that $\norm{p}{L^{2,a}(\partial^+ B^+)}=1$.
\end{lemma}
\begin{proof}
Let us consider $\rho^2_k = H(X_0,u,r_k)$, then by definition of the blow-up sequence $u_k$, \eqref{derivative.H} and Proposition \ref{almgren.lower} we obtain
 \begin{align*}
    \int_{\partial^+ B^+_R}{y^a u^2_k \mathrm{d}\sigma} & = \frac{1}{\rho_k^2}\int_{\partial^+ B^+_{R}}{y^a u^2(X_0 + r_k X) \mathrm{d}\sigma} \\
&=\frac{1}{\rho_k^2 r_k^{n+a}}\int_{\partial B_{Rr_k}(X_0)}{y^a u^2 \mathrm{d}\sigma}\\
    &= R^{n+a}\frac{H(X_0,u,Rr_k)}{H(X_0,u,r_k) }\\
    & \leq R^{n+a}\left( \frac{R r_k}{r_k}\right)^{\!2\tilde{C}}
\end{align*}
which gives us $\norm{u_k}{L^{2,a}(\partial^+ B^+_R)}^2\leq C(R) R^{n+a}$. Instead, inspired by Corollary \ref{equivalence1}, let $$
k=\frac{2s-\alpha}{2-q},
$$
for some $\alpha>0$, then
\begin{align*}
  \norm{u_k}{R}^q  &= \frac{1}{\rho^q_k}  \norm{u}{X_0,r_k R}^q\\
 & \leq C \left[e^{\tilde{C}R}(N(X_0,u,R)+1)\right]^{2/q}\frac{H(X_0,u,Rr_k)}{H(X_0,u,R)}\rho_k^{2-q} (Rr_k)^{\alpha-2s}\\
 & \leq C(R) \rho_k^{2-q} (R r_k)^{\alpha -2s}
\end{align*}
and by Lemma \ref{lem.poin} and Proposition \ref{almgren.lower}, we infer
\begin{align*}
\frac{1}{R^{n+a-1}}\int_{B^+_R}{y^a \abs{\nabla u_k}^2\mathrm{d}X} &= E(X_0,u,R r_k) + \left(\frac{r_k^{\frac{2s}{2-q}}}{\sqrt{H(X_0,u, r_k)}}\right)^{2-q} \frac{1}{R^{n+a-1}}\int_{\partial^0 B^+_R }{F_{\lambda_+,\lambda_-}(u_k)\mathrm{d}x}\\
&\leq C N(X_0,u,R r_k)\frac{1}{R^{n+a}}\int_{\partial^+ B^+_R}{y^a u_k^2\mathrm{d}\sigma} + R^{1-a} \frac{r_k^{2s}}{\rho^{2-q}_k}\norm{u_k}{R}^q\\
&\leq C(R) N(X_0,u,Rr_k)+ C(R) R^{1-a} \frac{r_k^{2s}}{\rho^{2-q}_k}\rho_k^{2-q} (R r_k)^{\alpha -2s}\\
& \leq C(R) (1+R^\alpha),
\end{align*}
which finally implies the uniform bound.\\
Thus, up to a subsequence, we have proved the existence of a non trivial function $p \in H^{1,a}_\loc(\R^n)$ such that $\norm{p}{L^{2,a}(\partial B^+_1)}=1$ and $u_k \rightharpoonup p$ in $H^{1,a}(B^+_R)$ for every $R>0$.\\
On the other side, the strong convergence in $H^{1,a}(B^+_R)$ follows easily testing the equation for $u_k$ against $(u_k -p)\eta$, where $\eta\in C^\infty_c(B_R)$ is
an arbitrary cut-off function, and passing then to the limit. Indeed, we have
\begin{align*}
\int_{B^+_R}{y^a \eta \langle \nabla u_k , \nabla (u_k - p)\rangle\mathrm{d}X} = &\, \alpha_k\int_{\partial^0 B^+_R}{\eta(u_k - p ) \left[\lambda_+ (u_k)_+^{q-1} - \lambda_- (u_k)_-^{q-1}\right]\mathrm{d}x} + \\
& \, - \int_{B^+_R}{(u_k-p)\langle \nabla u_k , \nabla p\rangle \mathrm{d}X}.
\end{align*}
Since $u_k$ is uniformly bounded in $H^{1,a}(B^+_R)$, up to a subsequence,  we get that $u_k \to p$ strongly in $L^2(B^+_R)$ and in $L^p (\partial^0 B^+_R)$, for every $p \in [1,2^*)$. In the end, since $\alpha_k \to 0^+$ we get
\begin{align*}
\abs{\int_{B^+_R}{y^a \eta \langle \nabla u_k , \nabla (u_k - p)\rangle\mathrm{d}X}} \leq  & \, \alpha_k \norm{u_k - p}{L^q(\partial^0 B^+_R)}\norm{u_k}{L^{q}(\partial^0 B^+_R)}+\\
&\, + C \norm{u_k - p}{L^{2,a}(B^+_R)} \norm{u_k}{H^{1,a}(B^+_R)}
\end{align*}
Finally, since by weak convergence
$$
\lim_{k \to \infty}\int_{B^+_R}{y^a \eta \langle \nabla u_k , \nabla (u_k - p)\rangle\mathrm{d}X} =
\lim_{k \to \infty}\int_{B^+_R}{y^a \eta \left( \abs{\nabla u_k}^2 - \abs{\nabla p}^2\right)\mathrm{d}X},
$$
we reach the desired result.
\end{proof}
So far we have proved the existence of a nontrivial function $p\in H^{1,a}_{\loc}(\R^{n+1})\cap L^\infty_{\loc}(\R^{n+1})$ such that, up to a subsequence, we have $u_k \to p$ strongly in $H^{1,a}_{\loc}(\R^{n+1})$ and $L_a p = 0$ in $\mathcal{D}'(\R^{n+1})$. The next step is to prove that for $X_0 \in \Gamma(u) \cap \Sigma$ the convergence $u_k \to p$ is uniformly on copact sets of $\overline{\R^{n+1}}$ and strong in $C^{0,\alpha}_{\loc}$ for $\alpha \in (0,1)$.

\begin{lemma}\label{holder}
  For every $R > 0$ there exists $C > 0$, independent of $k$, such that
  $$
  \left[ u_k \right]_{C^{0,\alpha}(B_R)} = \sup_{X_1,X_2 \in \overline{B}_R}\frac{\abs{u(X_1)-u(X_2)}}{\abs{X_1-X_2}^\alpha} \leq C
  $$
  for every $\alpha \in (0,1)$.
\end{lemma}
\begin{proof}
    The proof follows essentially the ideas of the similar results in \cite{tvz1, tvz2,STT2020}:  the critical exponent $\alpha=1$ is related to a Liouville type theorem for $L_a$-harmonic function in $\R^{n+1}$ symmetric with respect to the characteristic manifold $\{y=0\}$, as given in \cite{vita2020}.
\end{proof}
As a first Corollary we deduce that the possible vanishing orders of $u$ in the case $\mathcal{O}(u,X_0)<2s/(2-q)$ are completely classified as the possible vanishing orders of $L_a$-harmonic function even with respect to $\{y=0\}$. More precisely
\begin{corollary}\label{blow-up.lim}
Let $X_0 \in \Gamma(u)$ be such that $k=\mathcal{O}(u,X_0)< k_q$. Then $k\in 1+\N$ and every blow-up limit centered at $X_0$ is a $k$-homogeneous solution of \eqref{La.even}.
\end{corollary}
\begin{proof}
Let $k= \mathcal{O}(u,X_0)<k_q$. By Theorem \ref{blowup.lower} we already know that given $(u_j)_j$ a normalized blow-up sequence centered in $X_0$ and associated to some $r_j\to 0^+$, it converges strongly in $H^{1,a}_\loc(\R^{n+1})$ and uniformly on every compact set of $\overline{\R^{n+1}_+}$ to some $p \in H^{1,a}_{\loc}(\R^{n+1})$ such that
  $$
\begin{cases}
-L_a p=0 & \mbox{in } \R^{n+1}_+\\
  -\partial^a_y p = 0 & \mbox{on } \R^n \times \{0\}.
\end{cases}
$$
On the other hand, by Corollary \ref{equivalence1} and Corollary \ref{equivalence2} we get $N(X_0,u,0^+)=k$. By the strong convergence of $(u_j)_j$, we have
$$
N(0,p,r)=\lim_{j\to \infty}N(0,u_j,r) =
\lim_{j\to \infty}N(X_0,u,r r_j) =
N(X_0,u,0^+) \quad\mbox{for every }r>0,
$$
where
$$
N(0,p,r) = \ddfrac{r\int_{B^+_r}{y^a \abs{\nabla p}^2\mathrm{d}X}}{\int_{\partial^+ B^+_r}{y^a p^2 \mathrm{d}\sigma}}.
$$
Since $p$ is a global $L_a$-harmonic function even with respect to $\{y=0\}$, by \cite[Lemma 4.7]{STT2020} we deduce that $p$ is $k$-homogeneous in $\R^{n+1}_+$ with $k=1+\N$.
\end{proof}
In order to conclude the local analysis near the points of the nodal set such that $\mathcal{O}(u,X_0)<2s/(2-q)$ we introduce the following Weiss-type monotonicity formula.
\begin{proposition}\label{weiss.lower}
  Let $X_0 \in \Gamma(u)$ be such that $k=\mathcal{O}(u,X_0)< k_q$. Given $\delta = 2s-(2-q)k >0$, there exist $R_1>0$ and $C_2>0$ such that
  $$
r \mapsto W_k(X_0,u,r) + C_2(n,s,q,\Lambda,k) r^{\delta-\eps}
  $$
 is monotone non-decreasing, for every $r \in (0,\min\{R_1,\mathrm{dist}(X_0,\partial^+ B^+_1)\})$ and $\eps<\delta$. In particular, we get
 \be\label{weiss.limit}
 W_k(X_0,u,0^+)= \lim_{r\to 0^+} W_k(X_0,u,r) = 0.
 \ee
\end{proposition}
\begin{proof}
For $k>0$, by Proposition \ref{weiss.mon} and Lemma \ref{lem.poin}, we get
$$
  \frac{d}{dr}W_{k}(X_0,u,r) \geq -\frac{C^s_{n,2} \Lambda}{q r^{n+a+2k}} \int_{\partial^0 B^+_r}{\abs{u}^q\mathrm{d}x}\geq -C r^{2s-1}\frac{\norm{u}{X_0,r}^q}{r^{2k}}
$$
where $C=C(n,q,s,\Lambda)$. By definition of $H^{1,a}$-vanishing order, for every $k_1< \mathcal{O}(u,X_0)$ there exists $R_1,C_1>0$ such that
$$
\norm{u}{X_0,r}^2 \leq Cr^{2k} \longrightarrow \frac{d}{dr}W_k(X_0,u,r) \geq -C_1 r^{2s-1-2k+qk_1},
$$
for every $r<R_1$. Since $k< k_q$, there exist $\delta >0$ such that
$$
k = \frac{2s-\delta}{2-q}.
$$
Thus, for every $\eps<\delta$, if we take $k_1=k-\eps/q$ we get that $r\mapsto W_k(X_0,u,r) + C_2 r^{\delta-\eps}$ is monotone non-decreasing, where $C_2$ does not depend on $\eps>0$.\\
Finally, since by Corollary \ref{equivalence1} and Corollary \ref{equivalence2} we have $k=\mathcal{O}(u,X_0)=N(X_0,u,0^+)$, we get
$$
W_k(X_0,u,0^+)=\lim_{r \to 0^+}\frac{H(X_0,u,r)}{r^{2k}}\left( N(X_0,u,r)-k\right) = 0.
$$
\end{proof}
\begin{proposition}\label{monneau.pot}
  Let $X_0 \in \Gamma(u)$ such that $k=\mathcal{O}(u,X_0)<k_q$. Given $\delta = 2s-(2-q)k >0$, there exist $R_1>0$ and $C_2>0$ such that, for every homogenous $L_a$-harmonic polynomial $p\in \mathfrak{sB}^a_k(\R^{n+1})$, the map
  $$
  r \mapsto \frac{H(X_0,u-p_{X_0},r)}{r^{2k}}= \frac{1}{r^{n+a+2k}}\int_{\partial^+ B^+_r(X_0)}{y^a \left(u -p_{X_0}\right)^2\mathrm{d}\sigma}
  $$
  satisfies
  $$
  \frac{d}{dr}\frac{H(X_0,u-p_{X_0},r)}{r^{2k}} \geq -C(1+\norm{p_{X_0}}{L^\infty(B^+_1)})r^{-1+\delta-\eps},
  $$
  for every $r \in (0,\min\{R_1,\mathrm{dist}(X_0,\partial^+ B^+_1)\}$ and $\eps<\delta$, with $p_{X_0}(X)=p(X-X_0)$.
\end{proposition}
\begin{proof}
First, since $k=\mathcal{O}(u,X_0)$ we already know $W_k(X_0,u,0^+)=0$. Now, let $w = u - p_{X_0}$, then on one hand we have
 \begin{align*}
 \frac{d}{dr}\left(\frac{1}{r^{n+a+2k}}\int_{\partial^+ B^+_r(X_0)}{y^a w^2 \mathrm{d}\sigma}\right)
 =& \frac{2}{r^{n+a+1+2k}}\int_{\partial^+ B^+_r(X_0)}{y^a w(\langle X-X_0, \nabla w\rangle -k w)\mathrm{d}\sigma}\\
 =& \frac{2}{r}W_k(X_0,w,r).
  \end{align*}
On the other hand, looking at the expression of the $k$-Weiss functional, we have
\begin{align*}
    W_k(X_0,u,r) =&\, W_k(X_0,w+p_{X_0},r)\\ 
     =&\, \frac{1}{r^{n+a-1+2k}}\left(\int_{B^+_r(X_0)}{y^a (\abs{\nabla w}^2 +2  \langle \nabla w,\nabla p\rangle)\mathrm{d}X}- \frac{k}{r}\int_{\partial^+ B^+_r(X_0)}{y^a (w^2 +2  w p) \mathrm{d}\sigma}\right) \\
     &+ \frac{1}{r^{n+a-1+2k}}\int_{\partial^+B^+_r(X_0)}{(w+p_{X_0})\partial^a_y (w+p_{X_0}) \mathrm{d}x}\\
     =&\, W_k(X_0,w,r) + \frac{1}{r^{n+a-1+2k}}\int_{\partial^0B^+_r(X_0)}{p_{X_0}\partial^a_y w\mathrm{d}x} +\\
     &+ \frac{2}{r^{n+a+2k}}\int_{\partial^+ B^+_r(X_0)}{y^a w( \langle \nabla p_{X_0},X-X_0\rangle - k p) \mathrm{d}\sigma}   \\
   =&\, W_k(X_0,u-p_{X_0},r)+ \frac{1}{r^{n+a-1+2k}}\int_{\partial^0B^+_r(X_0)}{p_{X_0}\partial^a_y u\mathrm{d}x} ,
  \end{align*}
  where $C=C(\lambda_+,\lambda_-)$ and in the second equality we used the $k$-homogeneity of $p_{X_0}\in \mathfrak{sB}^a_k(\R^{n+1})$.
Hence 
we finally infer
\begin{align*}
\frac{d}{dr} \frac{H(X_0,u-p_{X_0},r)}{r^{2k}} &=\frac{2}{r}W_k(X_0,u-p_{X_0},r)\\ &\geq \frac{2}{r}W_k(X_0,u,r)+\frac{2C}{r^{n+a+2k}}\int_{\partial^0 B^+_r(X_0)}{p_{X_0}\abs{u}^{q-2}u\mathrm{d}x}.
\end{align*}
On one side by Proposition \ref{weiss.lower} we have
$$
W_k(X_0,u,r) =W_k(X_0,u,r) -W_k(X_0,u,0^+) \geq -C_2(n,s,q,\Lambda,k) r^{\delta-\eps},
$$
with $\delta = 2s-(2-q)k >0$ and $\eps<\delta$. On the other, under the notations of Proposition \ref{weiss.lower}, for every $\eps \in (0,\delta)$ let us introduce
$$
k_1 = k -\frac{\eps}{q-1} <  k = \frac{2s-\delta}{2-q}.
$$
Then, by \eqref{equiv} we infer the existence of $R>0$ sufficiently small such that
\begin{align}\label{utile}
\begin{aligned}
\abs{\frac{C}{r^{n+a+2k}}\int_{\partial^0 B^+_r(X_0)}{p_{X_0}\abs{u}^{q-2}u\mathrm{d}x}}  &\leq
\norm{p_{X_0}}{L^\infty(B^+_r)} \frac{C}{r^{n+a+2k}} \left(\int_{\partial^0 B^+_r(X_0)}{\abs{u}^q\mathrm{d}x}\right)^{(q-1)/q}\abs{B_r}^{1/q}\\
& \leq \norm{p_{X_0}}{L^\infty(B^+)} \frac{C}{r^{a+k}} \norm{u}{H^{1,a}(B^+_r)}^{q-1}\\
& \leq C \norm{p_{X_0}}{L^\infty(B^+)} r^{2s-1-k+(q-1)k_1}\\
& \leq C \norm{p_{X_0}}{L^\infty(B^+)} r^{-1+\delta - \eps},
\end{aligned}
\end{align}
for $r \in (0,R)$. Hence, there exist $R_1>0$  and $C=C(n,s,q,\Lambda,k)$ such that
$$
r\mapsto \frac{H(X_0,u-p_{X_0},r)}{r^{2k}} + C(1+\norm{p_{X_0}}{L^\infty(B^+)})r^{\delta-\eps},
$$
is monotone nondecreasing $r \in (0,\min\{R_1,\mathrm{dist}(X_0,\partial^+ B^+_1)\})$ and $\eps<\delta$.
\end{proof}
For the sake of simplicity, we will use through the paper the following notation for the previous monotonicity formula
$$
M(X_0,u,p_{X_0},r)=\frac{H(X_0,u-p_{X_0},r)}{r^{2k}}. 
$$
Starting from these results, we will improve our knowledge of the blow-up convergence by proving the existence of a unique non trivial blow-up limit at every point of the nodal set $\Gamma(u)$, which will be called the tangent map $\varphi^{X_0}$ of $u$ at $X_0$.
\begin{lemma}\label{growth}
  Let $X_0 \in \Gamma(u)$ be such that $k=\mathcal{O}(u,X_0)< k_q$. Then, there exists $r_0>0$ and $C>0$ such that
  $$
  H(X_0,u,r)\leq C r^{2k} \quad\mbox{for } r \in (0,r_0).
  $$
\end{lemma}
\begin{proof}
  Let $k=\mathcal{O}(u,X_0)$ and $\delta = 2s/(2-q)-k$. By \eqref{derivative.H} and Proposition \ref{almgren.lower}, there exist $r_0>0, \alpha = \alpha(\delta,n,s,q)$ and $\tilde C=\tilde C(\delta,n,s,q)$ such that
  \begin{align*}
  \frac{d}{d\rho}\log\left(\frac{H(X_0,u,\rho)}{\rho^{2k}}\right) &= \frac{2}{\rho}\left(N(X_0,u,\rho)-k \right)\\
  &= \frac{2}{\rho}\left(e^{-\tilde{C}\rho^\alpha} e^{\tilde{C}\rho^\alpha}(N(X_0,u,\rho)+1)-1-k \right)\\
  &\geq 2(k+1)\frac{e^{-\tilde{C}\rho^\alpha} -1}{\rho},
  \end{align*}
  for every $\rho \in (0,r_0)$. Thus, given $r<r_0$ and integrating between $r$ and $r_0$ we get
  $$
  \frac{H(X_0,u,r)}{r^{2k}} \leq   \frac{H(X_0,u,r_0)}{r_0^{2k}} \exp\left(2(k+1)\int_0^{r_0}\frac{e^{-\tilde{C}\rho^\alpha}-1}{\rho}\mathrm{d}\rho\right)\leq C,
  $$
  as we claimed.
\end{proof}
\begin{lemma}\label{nondegeneracy}
Let $X_0 \in \Gamma(u)$ be such that $k=\mathcal{O}(u,X_0)< k_q$. Then, there exists $C>0$ such that
$$
\sup_{\partial B_r(X_0)}{\abs{u(X)}}\geq C r^{k}\quad \mathrm{for}\,\,0<r<R
$$
where $R=1-\mathrm{dist}(X_0,\partial^0 B_1)$.
\end{lemma}
\begin{proof}
Fix $X_0\in \Gamma(u)$ and suppose by contradiction, given a decreasing sequence $r_j\downarrow 0$, that
$$
\lim_{j\to \infty}\frac{H(X_0,u,r_j)^{1/2}}{r_j^{k}}=\lim_{j\to\infty}{\bigg(\frac{1}{r_j^{n+a+2k}}\int_{\partial^+ B^+_{r_j}(X_0)}{y^a u^2\,\mathrm{d}\sigma\bigg)^{1/2}}}\!\!=0.
$$
For $r_j\leq R=\min(r_0,\mathrm{dist}(X_0,\partial^0 B^+))$, consider the blow-up sequence
$$
u_j(X)=\frac{u(X_0+r_j X)}{\rho_j}\quad \mbox{where }\, \rho_j = H(X_0,u,r_j)^{1/2}, 
$$
centered in $X_0 \in \Gamma(u)$. By Theorem \ref{blowup.lower} and Corollary \ref{blow-up.lim} the sequence $(u_j)_j$ converges, up to a subsequence, strongly in $H^{1,a}_\loc(\R^{n+1})$ and uniformly on every compact set of $\R^{n+1}_+$ to some $L_a$-harmonic homogenous polynomial $p$ of degree $k$ symmetric with respect to $\{y=0\}$ such that $H(0,p,1)=1$.\\
Let us focus our attention on the functional $M(X_0,u,p_{X_0},r)$ with $p_{X_0}$ as above. Under the notations in Proposition \ref{monneau.pot}, we get
\begin{align*}
M(X_0,u,p_{X_0},0^+) &= \lim_{r\to 0} \frac{1}{r^{n+a+2k}}\int_{\partial^+ B_r^+(X_0)}{y^a(u-p_{X_0})^2\,\mathrm{d}\sigma}+ C(1+\norm{p_{X_0}}{L^\infty(B^+)})r^{\delta-\eps}\\
&=\lim_{r\to 0} \int_{\partial^+ B_1^+}{y^a\Big(\frac{u(X_0+rX)}{r^k}-p(X) \Big)^2\,\mathrm{d}\sigma}+ C(1+\norm{p_{X_0}}{L^\infty(B^+)})r^{\delta-\eps}\\
&= \int_{\partial^+ B_1^+}{y^a p^2\,\mathrm{d}\sigma}\\
&=\frac{1}{r^{n+a+2k}}\int_{\partial^+ B_r^+(X_0)}{y^a {p}_{X_0}^2\,\mathrm{d}\sigma},
\end{align*}
where in the third equality we used the assumption on the growth of $u$. By the monotonicity result of Proposition \ref{monneau.pot}, we obtain
$$
\frac{1}{r^{n+a+2k}}\int_{\partial^+ B_r^+(X_0)}{y^a(u- p_{X_0})^2\,\mathrm{d}\sigma}+ C(1+\norm{p_{X_0}}{L^\infty(B^+)})r^{\delta-\eps}\geq\frac{1}{r^{n+a+2k}}\int_{\partial^+ B_r^+(X_0)}{y^a p_{X_0}^2\,\mathrm{d}\sigma}
$$
and similarly
$$
\frac{1}{r^{n+a+2k}}\int_{\partial^+ B^+_r(X_0)}{y^a (u^2-2u p_{X_0})\,\mathrm{d}\sigma}+ C(1+\norm{p_{X_0}}{L^\infty(B^+)})r^{\delta-\eps}\geq 0.
$$
On the other hand, rescaling the previous inequality and using the notion of blow-up sequence $u_k$ defined as above, we obtain
$$
\frac{1}{r_j^{2k}}\int_{\partial^+ B_1^+}{y^a\left(H(X_0,u,r_j)u_j^2-2H(X_0,u,r_j)^{1/2}r_j^k  u_j p\right)\,\mathrm{d}\sigma}\geq -C(1+\norm{p_{X_0}}{L^\infty(B^+)})r_j^{\delta-\eps}
$$
and
$$
\int_{\partial^+ B_1^+}{y^a\left(\frac{H(X_0,u,r_j)^{1/2}}{r_j^k}u_j^2-2 u_j p\right)\,\mathrm{d}\sigma}\geq - C(1+\norm{p_{X_0}}{L^\infty(B^+)})\frac{r_j^{k+\delta-\eps}}{H(X_0,u,r_j)^{1/2}}.
$$
Since $\mathcal{V}(u,X_0)= \mathcal{O}(u,X_0)=k$, by Definition \ref{nu} we get
$$
\limsup_{j \to \infty} \frac{H(X_0,u,r_j)}{r_j^{2(k+\delta-\eps)}}=+\infty,
$$
and consequently, passing to the limit as $j\to \infty$ in the previous inequality, we obtain
$$
\int_{\partial^+ B_1^+}{y^a p^2\,\mathrm{d}\sigma}\leq 0
$$
in contradiction with $p \not\equiv 0$.
\end{proof}
\begin{theorem}\label{unique}
  Let $X_0 \in \Gamma(u)$ be such that $k=\mathcal{O}(u,X_0)< k_q$. Then there exists a unique nonzero $\varphi^{X_0} \in \mathfrak{sB}_k^a(\R^{n+1})$ blow-up limit such that
\begin{equation}\label{blow.up.homogenous}
  u_{X_0,r}(X) =\frac{u(X_0+rX)}{r^k} \longrightarrow  p(X).
\end{equation}
Moreover, we define as \emph{tangent map} of $u$ at $X_0$ the unique nonzero map $\varphi^{X_0} \in \mathfrak{sB}_k^a(\R^{n+1})$ that satisfies \eqref{blow.up.homogenous}.
\end{theorem}
\begin{proof}
Up to a subsequence $r_j\to 0^+$, we have that $u_{X_0,r_j}\to p$ in $\mathcal{C}^{0,\alpha}_{\loc}$ for every $\alpha \in (0,1)$. The existence of such limit follows directly from the growth estimate of Lemma \ref{growth} and, by Lemma \ref{nondegeneracy}, we have $p$ is not identically zero. Now, by Proposition \ref{weiss.lower}, for any $r>0$ we have
$$
W_k(0,p,r) = \lim_{j\to \infty}{W_k(0,u_{X_0,r_j},r)}= \lim_{j\to \infty}{W_k(X_0,u,r r_j)}= W_k(X_0,u,0^+)=0,
$$
where
$$
W_k(0,p,r) = \frac{1}{r^{2k}}\left[\frac{1}{r^{n+a-1}}\int_{B^+_r}{y^a \abs{\nabla p}^2\mathrm{d}X} - k \frac{1}{r^{n+a-1}}\int_{\partial^+ B^+_r}{y^a p^2\mathrm{d}\sigma}\right].
$$
In particular, by \cite[Proposition 5.2]{STT2020} it implies that the $p$ is $k$-homogeneous $L_a$-harmonic function even with respect to $\{y=0\}$ and consequently $p \in \mathfrak{sB}^a_k(\R^{n+1})$. Now, by Proposition \ref{monneau.pot}, the limit of the Monneau-type formula exists and can be computed by
\begin{align*}
M(X_0,u,p_{X_0},0^+) &= \lim_{j\to\infty}{M(X_0,u,p_{X_0},r_j)}\\
&= \lim_{j\to\infty}{M(0,u_{X_0,r_j},p,1)}\\
&= \lim_{j\to\infty}{\int_{\partial^+ B_1^+}{y^a(u_{X_0,r_j} -p)^2\,d\sigma}}=0.
\end{align*}
Moreover, let us suppose by contradiction that for any other sequence $r_i\to 0^+$ we have that the associated sequence  $(u_{X_0,r_i})_i$ converges to another blow-up limit, i.e. $u_{X_0,r_i}\to q\in \mathfrak{B}^a_k(\R^{n+1}), q\not\equiv p$, then
\begin{align*}
0=M(X_0,u,p_{X_0},0^+) &= \lim_{i\to \infty}M(X_0,u,p_{X_0},r_i)\\
&=\lim_{i\to\infty}\int_{\partial^+ B_1^+}{y^a (u_{r_i} - p)^2\,d\sigma}\\
&=\int_{\partial^+ B_1^+}{y^a (q- p)^2\,d\sigma}.
\end{align*}
As we claim, since $q$ and $p$ are both homogenous of degree $k$ they must coincide in $\R^{n+1}$.
\end{proof}
Thanks to the uniqueness and the non-degeneracy of the blow-up limit, we can also construct the generalized Taylor expansion of the solution on the nodal set
\begin{theorem}
\label{continuation}
Let $X_0 \in \Gamma(u)$ be such that $k=\mathcal{O}(u,X_0)< 2s/(2-q)$ and $\varphi^{X_0}$ be the tangent map of $u$ at $X_0$. Then
\begin{equation}\label{eq.continuation}
u(X)=\varphi^{X_0}(X-X_0) + o(\abs{X-X_0}^{k}).
\end{equation}
Moreover, the map $X_0 \mapsto \varphi^{X_0}$ from $\Gamma_k(u)$ to $\mathfrak{sB}_k^a(\R^{n+1})$ is continuous. 
\end{theorem}
\begin{proof}
Since $\mathfrak{sB}^a_k(\R^{n+1})$ is a convex subset of a finite-dimensional vector space, namely the space of all $k$-homogeneous polynomials in $\R^{n+1}$, all the norms on such space are equivalent and hence we can then endow $\mathfrak{sB}^a_k(\R^{n+1})$ with the norm of $L^{2,a}(\partial^+ B^+_1)$.\\
Fixed $X_0\in\Gamma(u)$, by Theorem \ref{unique} we have the following expansion
$$
u(X)=\varphi^{X_0}(X-X_0) + o(\abs{X-X_0}^{k}).
$$
where $\varphi^{X_0}$ is the unique tangent map of $u$ at $X_0$. Given $\varepsilon>0$, consider $r_\varepsilon=r_\varepsilon(X_0)$ such that
$$
M(X_0,u,\varphi^{X_0},r_\varepsilon)= \frac{1}{r^{n+a+2k}_\varepsilon}\int_{\partial^+ B_{r_\varepsilon}^+}{y^a\big(u(X_0+X)-\varphi^{X_0}(X)\big)^2\,d\sigma}<\varepsilon.
$$
There exists also $\delta_\varepsilon=\delta_\varepsilon(X_0)$ such that if $X_1\in \Gamma_k(u)\cap \Sigma$ and $\abs{X_1-X_0}<\delta_\varepsilon$ then
$$
\frac{1}{r^{n+a+2k}_\varepsilon}\int_{\partial^+ B_{r_\varepsilon}^+}{y^a\big(u(X_1+X)-\varphi^{X_0}(X))^2\,d\sigma}<2\varepsilon
$$
or similarly
$$
\int_{\partial^+ B^+_{1}}{y^a\left( \frac{u(X_1+r_\varepsilon X)}{r^k_\varepsilon} - \varphi^{X_0}(X)\right)^2\,d\sigma}<2\varepsilon
$$
From Proposition \ref{monneau.pot}, we have that $M(X_1,u,\varphi^{X_0},r)<2\varepsilon + C r_\eps^{\delta/2}$ for $r\in (0,r_\varepsilon)$, which implies
\begin{align*}
M(X_1,u,\varphi^{X_0},0^+)&=\lim_{r\to 0}{M(X_1,u,\varphi^{X_0},r)}\\
&=\lim_{r\to 0}\int_{\partial^+ B^+_1}{y^a\left( \frac{u(X_1+rX)}{r^k} - \varphi^{X_0}(X)\right)^2\,d\sigma}\\
&=\int_{\partial^+ B_1^+}{y^a\left( \varphi^{X_1}- \varphi^{X_0}\right)^2\,d\sigma}\leq 2\varepsilon + Cr_\eps^{\delta/2}.
\end{align*}
\end{proof}
Finally, we improve the convergence rate $o(\abs{X-X_0}^k)$ of the previous generalized Taylor's expansion into a quantitative bound of the form $O(|X-X_0|^{k+\delta})$ for some $\delta>0$.\\
Before to state the main result, we prove the validity of an Almgren-type monotonicity result for the difference between the solution $u$ and its tangent map $\varphi^{X_0}$ at $X_0$.
\begin{theorem}\label{basta}
Let $X_0 \in \Gamma(u)$ be such that $k=\mathcal{O}(u,X_0)< k_q$ and
$$
w(X)=u(X)-\varphi^{X_0}(X-X_0),
$$
with $\varphi^{X_0}$ the tangent map of $u$ at $X_0$. Then, there exist $r_0,\alpha>0$ and an absolutely continuous map $\Psi(r)$ such that $0\leq \Psi(r)\leq C r^{\alpha}$ and the map
$$
r \mapsto e^{\tilde{C}\Psi(r)}\!\left(N(X_0,w,r)+1\right)
$$
is monotone non-decreasing for $r \in (0,r_0)$. Consequently, there exists
$$
N(X_0,w,r)= \lim_{r\to 0^+}N(X_0,w,r).
$$
\end{theorem}
\begin{proof}
In order to simplify the notations, it is not restrictive to assume that $X_0=0$. Since $k=\mathcal{O}(u,0)<k_q$, by Lemma \ref{growth}, Lemma \ref{nondegeneracy} and Theorem \ref{unique} we already know that there exists $C_1,C_2>0$ such that
\be\label{aso}
C_1 r^k \leq \norm{u}{H^{1,a}(B_r)}\leq C_2 r^k.
\ee
for every $r\in (0,\overline{r})$.
Then, following the same computation of the last Section, we easily deduce by an integration by parts (see the proof of Proposition \ref{E.derivative}) that
$$
\frac{d}{dr}E(w,r) = \frac{2}{r^{n-1+a}}\int_{\partial^+ B^+_r}y^a(\partial_r w)^2 \mathrm{d}\sigma+ R(w,r)\quad \frac{d}{dr}H(w,r)=\frac{2}{r}E(w,r),
$$
where $E(w,r)=E(0,w,r), H(w,r)=H(0,w,r)$ are defined according to \eqref{E.H} and
$$
R(w,r)=\frac{1-n-a}{r^{n+a}}\int_{\partial^0 B^+_r}{w\partial^a_y w}\mathrm{d}x + \frac{1}{r^{n-1+a}}\int_{S^{n-1}_r}w\partial^a_y w \mathrm{d}\sigma - \frac{2}{r^{n+a}}\int_{\partial^0 B^+_r}\partial^a_y w \langle \nabla w, x\rangle \mathrm{d}x.
$$
Consequently, by the validity of the Cauchy-Schwarz inequality on $\partial^+ B^+_r$, the associated Almgren-type functional satisfies
\be\label{aso}
\frac{d}{dr}\log(N(w,r)+1) \geq  \frac{R(w,r)}{E(w,r)+H(w,r)}.
\ee
Let $\varphi\in \mathfrak{sB}^a_k(\R^{n+1})$ be the unique tangent map of $u$ at $0\in \Gamma(u)$ and consider the difference $w=u-\varphi \in H^{1,a}(B^+_r)$. By definition of tangent map, it satisfies
\be\label{fine}
\begin{cases}
  L_a w=0 & \mbox{in } B^+_r\\
  -\partial^a_y w= \lambda_+(w+\varphi)^{q-1}_+ - \lambda_-(w+\varphi)^{q-1}_- & \mbox{on }\partial^0 B^+_r.
\end{cases}
\ee
On one hand, we get
\begin{align*}
R(w,r)=&\,\frac{2-q}{q}\frac{1}{r^{n+a-1}}\int_{S^{n-1}_r}F_{\lambda_+,\lambda_-}(w+\varphi)\mathrm{d}\sigma -\frac{C^s_{n,q}}{qr^{n+a}}\int_{\partial^0 B^+_r}F_{\lambda_+,\lambda_-}(w+\varphi)\mathrm{d}x+\\
& + \frac{2s-n-2}{r^{n+a}}\int_{\partial^0 B^+_r}\varphi f(w+\varphi)\mathrm{d}x + \frac{1}{r^{n+a-1}}\int_{S^{n-1}_r}\varphi f(w+\varphi) \mathrm{d}\sigma
\end{align*}
where $f(t)=\lambda_+t_+^{q-1}-\lambda_-t_-^{q-1}$. On the other hand, by Lemma \ref{lem.poin} and \eqref{utile} we get
\begin{align*}
E(w,r)+H(w,r)&\geq \norm{u}{r}^2 -C r^{2s}\norm{u}{r}^q +\frac{1}{r^{n-1+a}}\int_{\partial^0 B^+_r}\varphi\partial^a_y u\mathrm{d}x\\
&\geq \norm{u}{r}^2 -C r^{2s}\left(\norm{u}{r}^q + \norm{\varphi}{L^\infty(B_1)}r^k\norm{u}{r}^{q-1}\right)
\end{align*}
In order to estimate the last remainder $R(w,r)$ we need the introduce the auxiliary function
$$
\psi(r)=r\left(\frac{1}{r^n}\int_{\partial^0 B^+_r}\abs{u}^q\mathrm{d}x\right)^h
$$
for $h\in (0,1)$ to be chosen later. A direct computation yields the identity
$$
\psi'(r)= \frac{\psi(r)}{r}\left(hn+1 + h r \frac{\int_{S^{n-1}_r}\abs{u}^q\mathrm{d}\sigma}{\int_{\partial^0 B^+_r}\abs{u}^q\mathrm{d}x}\right)
$$
which implies, by Lemma \ref{lem.poin}, that
$$
\frac{1}{r^{n-1}}\int_{S^{n-1}_r}\abs{u}^q\mathrm{d}\sigma \leq  \frac{\psi'(r)}{h} \norm{u}{r}^{q(1-h)}.
$$
Finally, we get
$$
\abs{ \frac{1}{r^{n+a-1}}\int_{S^{n-1}_r}\varphi f(w+\varphi)\mathrm{d}\sigma} \leq \frac{r^{2s-1}}{h}\norm{\varphi}{L^\infty(B_1)}r^k \psi'(r) \norm{u}{r}^{(q-1)(1-h)}
$$
and consequently
\be\label{R}
R(w,r)\geq -C r^{2s-1}\left(\norm{u}{r}^q + \norm{\varphi}{L^\infty(B_1)}r^{k}\norm{u}{r}^{q-1} + \norm{\varphi}{L^\infty(B_1)}r^{k} \psi'(r) \norm{u}{r}^{(q-1)(1-h)}\right)
\ee
for some $h\in (0,1)$. By \eqref{aso} and \eqref{utiledopo}, there exists $\alpha>0$ such that
\begin{align*}
\frac{d}{dr}\log(N(w,r)+1)&\geq -\frac{C}{r} \frac{r^{2s+kq}\left(1+\psi'(r) r^{kh(1-q)}\right)}{r^{2s-\alpha}\norm{u}{r}^q}\\
&\geq -\frac{C}{r} r^{\alpha}\left(1+\psi'(r) r^{kh(1-q)}\right).
\end{align*}
Hence, let
$$
\Psi(r)=\int_0^r r^{\alpha-1}(1+\psi'(t)t^{kh(1-q)})\mathrm{d}t.
$$
Then, by Lemma \ref{lem.poin} we first deduce $0\leq \psi(r)\leq Cr^{1+kqh}$ and then
\begin{align*}
0\leq \Psi(r) &= \int_0^r t^{\alpha-1}\left(1+\psi'(t)t^{kh(1-q)}\right)\mathrm{d}t\\
&= \frac{r^\alpha}{\alpha} + \left[\psi(r)r^{kh(1-q)+\alpha-1}\right]^r_0 - \int_0^r \psi(t)\frac{t^{kh(1-q)+\alpha-2}}{kh(1-q)+\alpha-1}\mathrm{d}t\\
&\leq \frac{r^\alpha}{\alpha} + r^{\alpha+kh} + C \int_0^r t^{kh+\alpha-1}\mathrm{d}t\\
& \leq C r^\alpha,
\end{align*}
for $r$ sufficiently small. As a result, we find that the function
$$
r \mapsto e^{C\Psi(r)}(N(w,r)+1)
$$
is absolutely continuous and increasing for $r\in (r_1, r_2)$, for some $0<r_1<r_2$. Following a standard argument, it follows that $H(w,r)$ is always strictly positive in the interval $(0, r_2)$, thanks to the monotonicity of the modified Almgren-type quotient. As a result, the modified Almgren-type formula is defined for all $r \in (0,r_2)$, and it can be extended for $r = 0$ by taking its limit for $r \to 0^+$.
\end{proof}
\begin{remark}\label{confronto}
  Notice that, under the notations of Theorem \ref{blow.uplimite}, the computations up to the final substitution of the estimates of the $H^{1,a}$-norm still hold in the critical case $\mathcal{O}(u,0)=k_q$ with $\mu=0$, i.e. $k_q\in \N$. Indeed, in the next section we will prove that if $k_q\in \N$ the blow-up limit $p$ is an homogeneous $L_a$-harmonic function symmetric with respect to $\{y=0\}$, and the function $w=u-p$ still satisfies \eqref{fine}. However, in this context the computations will lead to
  $$
\frac{d}{dr}\log(N(w,r_k)+1)\geq -\frac{C}{r_k} \frac{\alpha_k\left(1+\alpha_k^{1/(2-q)}(1+\psi' r^{kh(q-1)})\right)}{1- C\alpha_k\left(1+\alpha_k^{1/(2-q)}\right)}
\quad\mbox{with }\alpha_{k} =  \left(\ddfrac{r_k^{2s/(2-q)}}{\norm{u}{H^{1,a}(B_{r_k})}}\right)^{2-q}.
$$
By the dichotomy \eqref{casi}, even if $\mu=0$ yields to  $\alpha_k\to 0^+$, this is not enough to ensure the integrability of the right hand side of the previous inequality. As remarked in \cite{soavesublinear}, is possible that a sophisticated Fourier expansion finally lead to uniqueness: indeed it will imply that $r_k\mapsto \alpha_k(r_k)$ is Dini-continuous, which will be enough to ensure the validity of an Almgren-type monotonicity result.
\end{remark}
As a simple corollary of the monotonicity result in Proposition \ref{monneau.pot} for the Monneau-type formula, we easily deduce a lower bound for the Almgren-type formula evaluated on $w$.
\begin{corollary}\label{dalla}
  Let $X_0 \in \Gamma(u)$ be such that $\mathcal{O}(u,X_0)< k_q$. Then $N(X_0,u-\varphi^{X_0},0^+)\geq \mathcal{O}(u,X_0)$.
\end{corollary}
In order to improve the growth order of the remainder in \eqref{eq.continuation}, we start by proving a blow-up argument based on the validity of the previous Almgren-type monotonicity formula. Hence, given $X_0 \in \Gamma(u)$ and $w\in H^{1,a}_\loc(B_r)$ as in Theorem \ref{basta}, we consider the normalized  blow-up sequence $(w_k)_k$ centered in $X_0$ associated to some $r_k \downarrow 0^+$ (see \eqref{blow.up.almgren} for the definition of normalized blow-up sequence), such that
$$
\begin{cases}
-L_a w_k=0 & \mbox{in } B_{X_0,r_k}^+\\
  -\partial^a_y w_k = \alpha_k \left[\lambda_+ \left(\beta_k w_k + \varphi^{X_0} \right)_+^{q-1} - \lambda_- \left(\beta_k w_k + \beta_k\varphi^{X_0}\right)_-^{q-1}\right] & \mbox{on } \partial^0 B^+_{X_0,r_k}.
\end{cases}
$$
with
$$
\alpha_{k} =  \ddfrac{r_k^{2s +\mathcal{O}(u,X_0)(q-1)}}{\sqrt{H(X_0,w,r_k)}}, \quad \beta_k=\frac{\sqrt{H(X_0,w,r_k)}}{r_k^{\mathcal{O}(u,X_0)}}.
$$
\begin{lemma}\label{testa}
  Under the previous notations, let $0<k_1\leq k_2$ be such that $\mathcal{O}(u,X_0)\leq k_1<k_2$. Then, if $k_1\leq N(X_0,u-\varphi^{X_0},0^+)\leq k_2$ we infer
  $$
  \beta_k \to 0^+ \quad\mbox{and}\quad 0\leq \alpha_k \leq C r^{(2-q)\left(\frac{2s}{2-q}-\mathcal{O}(u,X_0)\right)},
  $$
  for some $C>0$ and $k$ sufficiently large.
\end{lemma}
\begin{proof}
First, by Proposition \ref{monneau.pot} we already know that $\beta_k\to 0^+$. Now, let $k_1,k_2>0$ be such that $\mathcal{O}(u,X_0)\leq k_1<k_2$ and $k_1\leq N(X_0,u-\varphi^{X_0},0^+)\leq k_2$. By \eqref{h.derivative} and Theorem \ref{basta} we have that if $k_1\leq N(X_0,u-\varphi^{X_0},0^+)\leq k_2$ then there exits $C_1,C_2,\overline{r}>0$ such that
$$
C_1 r^{k_1}\leq \sqrt{H(X_0,u-\varphi^{X_0},r)} \leq C_2 r^{k_2},
$$
for every $r \in (0,\overline{r})$. Thus
$$
\alpha_k \leq C r_k^{2s-k_1 +\mathcal{O}(u,X_0)(q-1)},
$$
for $k$ sufficiently large such that $r_k\leq \overline{r}$. Finally, by Corollary \ref{dalla}, if $k_1=\mathcal{O}(u,X_0)<2s/(2-q)$ we get
$$
\alpha_k \leq C r^{(2-q)\left(\frac{2s}{2-q}-\mathcal{O}(u,X_0)\right)},
$$
as we claimed.
\end{proof}
The following result finally improve the growth order by proving a quantitative bound of the form $O(\abs{X-X_0}^{k+\delta})$ for some $\delta \in \N, \delta>0$.
\begin{proposition}\label{basta2}
  Let $X_0 \in \Gamma(u)$ be such that $\mathcal{O}(u,X_0)< 2s/(2-q)$. Then $$N(X_0,u-\varphi^{X_0},0^+) \in \mathcal{O}(u,X_0) + \delta, \quad\mbox{for some }\delta \in \N, \delta >0.$$
\end{proposition}
\begin{proof}
  Let $w=u-\varphi^{X_0}$ and $(w_k)_k$ the normalized blow-up sequence centered at $X_0$ and associated to some $r_k\to 0^+$.
  As we did in Lemma \ref{bounded.H1}, exploiting the correlation between the normalized blow-up sequence with respect to the $L^{2,a}(\partial^+ B^+_1)$-norm and the validity of the Almgren-type monotonicity formula, it is easy to see that $(w_k)_k$ is uniformly bounded in $H^{1,a}_\loc(\R^{n+1})$ and it converges, up to subsequence, to some  $p\in H^{1,a}_{\loc}(\R^{n+1})\cap L^\infty_{\loc}(\R^{n+1})$ such that $\norm{p}{L^{2,a}(\partial^+ B^+_1)}=1$.\\
  On the other hand, since $\mathcal{O}(u,X_0)<2s/(2-q)$, by Lemma \ref{testa} we get that both the sequences $(\alpha_k)_k$ and $(\beta_k)_k$ approach zero as $k$ goes to infinity. Therefore, following the same contradiction argument of Lemma \ref{holder}, the sequence $(w_k)_k$ is uniformly bounded in $C^{0,\alpha}_\loc(\R^{n+1})$ for every $\alpha \in (0,1)$ and it converges uniformly on every compact set to a global solution of \eqref{La.even}. Moreover, by the strong convergence of $(w_k)_k$, we have
$$
N(0,p,r)=\lim_{k\to \infty}N(0,w_k,r) =
\lim_{k\to \infty}N(X_0,w,r r_k) =
N(X_0,w,0^+) \quad\mbox{for every }r>0,
$$
where
$$
N(0,p,r) = \ddfrac{r\int_{B^+_r}{y^a \abs{\nabla p}^2\mathrm{d}X}}{\int_{\partial^+ B^+_r}{y^a p^2 \mathrm{d}\sigma}}.
$$
Hence, $p$ is a homogeneous $L_a$-harmonic function even with respect to $\{y=0\}$ of order $N(0,p,1)$. By \cite[Lemma 4.7]{STT2020} we first get that $N(0,p,1)\in \N$ while by Theorem \ref{unique} we deduce that $N(0,p,1) > \mathcal{O}(u,X_0)$. Since $N(0,p,1)=N(X_0,w,0^+)$ we finally get the claimed result.
\end{proof}
Thanks to this classification, we can improve the growth order of the remainder in \eqref{eq.continuation}.
\begin{corollary}\label{holderblow}
  Let $X_0 \in \Gamma(u)$ be such that $k=\mathcal{O}(u,X_0)< k_q$ and $\varphi^{X_0}$ be the tangent map of $u$ at $X_0$. Then
$$
u(X)=\varphi^{X_0}(X-X_0) + O(\abs{X-X_0}^{k+\delta}),
$$
for some $\delta \in \N,\delta>0$.
Moreover, the map $X_0 \mapsto \varphi^{X_0}$ from $\Gamma_k(u)$ to $\mathfrak{sB}_k^a(\R^{n+1})$ is H\"{o}lder continuous.
\end{corollary}
Having established Theorem \ref{continuation} and Proposition \ref{basta2}, we can finally show the validity of the first part of Theorem \ref{caldo} and Theorem \ref{hau}.
\begin{proof}[Proof of Theorem \ref{caldo}]
  Let us consider the case $\mathcal{V}(u,X_0)<k_q$. By Corollary \ref{equivalence1} and Corollary \ref{equivalence2} we already know that
$$
\mathcal{O}(u,X_0) = \mathcal{V}(u,X_0) = N(X_0,u,0^+).
$$
Therefore the results of this Section hold true also for the case $\mathcal{V}(u,X_0)< k_q$. In particular, by Corollary \ref{blow-up.lim}, we know that $\mathcal{V}(u,X_0)$ must be a positive integer and, by Theorem \ref{continuation} and Proposition \ref{basta2}, there exists $\delta \in \N,\delta>0$ such that \eqref{eq.continuationa} holds true.
\end{proof}
Finally, by applying \cite[Theorem 8.5]{MR1637972}, a variant of the classical Federer’s dimension reduction principle (for which we refer to \cite[Appendix A]{Simon83}), and the Whitney's extension theorem (we refer to \cite{MR2514337} and the reference therein) we can easily estimate the Hausdorff dimension of the singular strata.
\begin{proof}[Proof of Theorem \ref{hau}]
First, since $\Gamma(u) = \mathcal{T}(u)$ for those values of $s\in (0,1)$ and $q \in [1,2)$ such that $k_q\leq 1$, let us concentrate on the opposite case. Since on $\mathcal{R}(u)\cup \mathcal{S}(u)$ all the notions of vanishing order coincide, i.e.
$$
\mathcal{O}(u,X_0) = \mathcal{V}(u,X_0) = N(X_0,u,0^+)<\frac{2s}{2-q},
$$
we can easily adapt the general approach of \cite{STT2020} by using the validity of the Almgren-type monotonicity formula. More precisely, by a straightforward application of Corollary \ref{holderblow} and the implicit function theorem, we already deduce that 
$$
\mathcal{R}(u)= \left\{X\in \Gamma(u) \colon N(X_0,u,0^+)=1\right\},
$$
which is relatively open in $\Gamma(u)$ and it is a $(n-1)$-dimensional regular set of class $C^{1,\alpha}$. Moreover, by the upper semi-continuity of $X_0\mapsto N(X_0,u,0^+)$, the proof of the Hausdorff estimate
$$
\mathrm{dim}_{\mathcal{H}}\mathcal{S}(u) \leq n-1
$$
follows the one of \cite[Theorem 6.3]{STT2020}).\\
On the other hand, it is possible to apply step by step the proof of \cite[Theorem 7.7]{STT2020} and \cite[Theorem 7.8]{STT2020} (using Corollary \ref{holderblow} instead of \cite[Theorem 5.12]{STT2020} and the generalized formulation of the Whitney's extension theorem in \cite{MR2514337} for $C^{m,\omega}$-functions), obtaining the desired result for the stratification of the singular set. The crucial idea is that the Whitney's extension allows to study the structure of the nodal set just by using the generalized Taylor expansion \eqref{eq.continuationa} without the high-order differentiability of the function itself.
\end{proof}
\section{Blow-up analysis for $\mathcal{O}(u,X_0)=k_q$}\label{6}
The previous analysis terminates the study of the nodal set in those points where the local behaviour of the solutions resemble the one of the $s$-harmonic functions. In this Section we will complete our study by considering the threshold case $\mathcal{O}(u,X_0)=2s/(2-q)$. The following result is the second part of Theorem \ref{caldo}.
\begin{theorem}\label{blow.uplimite}
Let $q \in [1,2), \lambda_+, \lambda_- > 0$ and $u \in H^{1,a}_\loc(B_1), u\neq 0$ be a solution of \eqref{system} and $X_0 \in \Gamma(u)$.\\
If $\mathcal{O}(u, X_0) = k_q$, then for every sequence $r_k \to 0^+$ we have, up to a subsequence, that
$$
\frac{u(X_0 + r_k X)}{\norm{u}{X_0,r_k}}
\to \overline{u} \quad\mbox{in } C^{0,\alpha}_\loc(\R^{n+1}),
$$
for every $\alpha \in (0,\min(1,k_q))$,
where $\overline{u}$ is a $k_q$-homogeneous non-trivial solution to
\be\label{limite}
\begin{cases}
  L_a \overline{u}=0 & \mbox{in } \R^{n+1}_+ \\
  -\partial^a_y \overline{u} = \mu\left(\lambda_+ (\overline{u}_+)^{q-1} - \lambda_- (\overline{u}_-)^{q-1}\right) &\mbox{on } \R^n \times \{0\},
\end{cases}
\ee
for some $\mu \geq 0$. Moreover, the case $\mu = 0$ is possible if and only if $k_q\in \N$.
\end{theorem}
The proof will be presented in a series of lemmata. Let $X_0 \in \Gamma(u)$ be such that $\mathcal{O}(u,X_0)=k_q$ and $r_k \to 0^+$. Moreover, we introduce the normalized blow-up sequence
\be\label{sequence}
u_k(X) = \frac{u(X_0 + r_k X)}{\norm{u}{X_0,r_k}}\quad\mbox{with }   X\in B^+_{X_0,r_k}=\frac{B_1^+ - X_0}{r_k},
\ee
for $0<r_k<R<\mathrm{dist}(X_0,\partial B_1)$. Then $\norm{u_k}{0,1} = 1$ and
$$
\begin{cases}
-L_a u_k=0 & \mbox{in } B_{X_0,r_k}^+\\
  -\partial^a_y u_k = \left(\ddfrac{r_k^{2s/(2-q)}}{\norm{u}{X_0,r_k}}\right)^{2-q} \left[\lambda_+ (u_k)_+^{q-1} - \lambda_- (u_k)_-^{q-1}\right] & \mbox{on } \partial^0 B^+_{X_0,r_k}.
\end{cases}
$$
By Proposition \ref{2s/(2-q)} (in particular by \eqref{eq.2s/(2-q)}), there exists $C>0$ such that
$$
0 < \alpha_{k} =  \left(\ddfrac{r_k^{2s/(2-q)}}{\norm{u}{X_0,r_k}}\right)^{2-q}\leq  C,
$$
for every $r_k<R$. As we pointed out in the previous Sections, the $H^{1,a}$-normalization seems to be more suitable for the critical case $\mathcal{O}(u,X_0)=2s/(2-q)$ and it overcomes the lack of monotonicity of the Almgren-type formula. The following is a compactness result for the blow-up sequence.
\begin{lemma}\label{compact}
For every $R>0$, there exists $k_R>0$ such that, for every $k>k_R$, the sequence $(u_k)_k$ is uniformly bounded in $H^{1,a}(B_{R}^+)$ and, up to a subsequence, it converges strongly in $L^{2,a}(B_{R}^+)$ and $H^{1,a}_\loc(B_{R}^+)$.
\end{lemma}
\begin{proof}
The strong convergence in $H^{1,a}(B_R^+)$ of $(u_k)_k$ is a straightforward consequence of the uniform bound in $H^{1,a}(B_R^+)$. Indeed, suppose there exists $k_R>0$ such that, for every $k>k_R$, we already know that $(u_k)_k$ weakly converges in $H^{1,a}(B_{R}^+)$ and strongly in $L^{2,a}(B_{R}^+)$ and $L^p(\partial^0 B^+_R)$ for $p \in [1,2^\star)$, by the uniform bound in $H^{1,a}(B_R^+)$.\\
Finally, the strong convergence follows easily testing the equation against $(u_k-u)\eta$, where $\eta \in C^\infty_c(B_{R})$, and passing to the limit as $k\to +\infty$. Namely, we get
\begin{align*}
\int_{B^+_R}y^a \eta \langle \nabla u_k,\nabla (u_k-u)\rangle \mathrm{d}X =& -\int_{B^+_R}y^a (u_k-u)\langle\nabla u_k, \nabla \eta \rangle \mathrm{d}X+\\ &+ \alpha_k \int_{\partial^0 B^+_R}{\eta (u_k - u)(\lambda_+ (u_k)_+^{q-1} - \lambda_- (u_k)_-^{q-1})\mathrm{d}x}.
\end{align*}
Since $(u_k)_k$ is uniformly bounded in $H^{1,a}(B_R)$ and it converges strongly in $L^{2,a}(B^+_R)$, the first term in the right hand side tends to $0$ as $k \to \infty$. Similarly, since $\alpha_k$ is bounded and $u_k \to u$ strongly in $L^p(\partial^0 B^+_R)$ for $p \in [1,2^\star)$, the second term vanishes too. Finally, regarding the left hand side, by the weak convergence we get
$$
\int_{B^+_R}y^a \eta \langle \nabla u_k,\nabla (u_k-u)\rangle \mathrm{d}X = \int_{B^+_R}y^a \eta \left(\abs{\nabla u_k}^2-\abs{\nabla u}^2\right) \mathrm{d}X + o(1)
$$
as $k$ goes to $+\infty$, which leads to the claimed result.\\

Hence, it remains to prove the validity of a uniform bounds in $H^{1,a}$.
 By definition of $(u_k)_k$, since
 $$
 \norm{u_k}{0,R}=\frac{\norm{u}{X_0,r_kR}}{\norm{u}{X_0,r_k}}
 $$
 the first part of the result follows if there exists $k_R, C_R>0$ such that
 $$
 \norm{u}{X_0,r_kR}\leq C_R \norm{u}{X_0,r_k}, \quad\mbox{for every }k\geq k_R.
 $$
 Thus, suppose by contradiction that, up to a subsequence, for $r_k\searrow 0$ it results
 $$
 \frac{\norm{u}{X_0,r_kR}}{\norm{u}{X_0,r_k}} \to +\infty.
 $$
 We claim, in such case, that
 \be\label{ssurdo}
 \frac{\norm{u}{X_0,r_kR}}{(r_k R)^{2s/(2-q)}}\to +\infty
 \ee
 as $k\to \infty$. If not, by \eqref{eq.2s/(2-q)}, we would have that
 $$
 \norm{u}{X_0,r_kR}\leq C (r_k R)^{2s/(2-q)} \leq C R^{2s/(2-q)}\norm{u}{X_0,r_k},
 $$
 against the absurd hypothesi. Thus, by Lemma \ref{lem.poin}, we get for every $r>0$ that
$$
 \frac{1}{r^{n+a-1}}\int_{\partial^0 B^+_r(X_0)} F_{\lambda_+,\lambda_-}(u)\mathrm{d}x
  \leq C \Lambda r^{2s} \norm{u}{X_0,r}^q =C \Lambda \left(\frac{r^{2s/(2-q)}}{\norm{u}{X_0,r}}\right)^{2-q}\norm{u}{X_0,r}^2,
$$
 where $\Lambda = \max\{\lambda_+,\lambda_-\}$. In particular, it implies
 $$
0\leq \frac{1}{(r_k R)^{n+a-1}\norm{u}{X_0,r_k R}^2}\int_{\partial^0 B^+_{r_k R}(X_0)} F_{\lambda_+,\lambda_-}(u)\mathrm{d}x
   \leq C \Lambda \left(\frac{(r_k R)^{2s/(2-q)}}{\norm{u}{X_0,r_k R}}\right)^{2-q} \longrightarrow 0,
 $$
 as $k\to \infty$. On the other hand, since
$$
W_{k,t}(X_0,u,r) = \frac{\norm{u}{X_0,r}^2}{r^{2k}}\left[1-(k+1)\frac{H(X_0,u,r)}{\norm{u}{X_0,r}^2}-\frac{t}{q}\frac{1}{r^{n+a-1}\norm{u}{X_0,r}^2}\int_{\partial^0 B^+_r(X_0)}F_{\lambda_+,\lambda_-}(u)\mathrm{d}x\right],
$$
by the monotonicity of $r \mapsto W_{k_q,2}(X_0,u,r)$, we deduce that
 \begin{align*}
 C &\geq W_{k_q,2}(X_0,u,r_k R)\\
 & \geq \frac{\norm{u}{X_0,(r_k R)}^2}{(r_k R)^{2k_q}}\left[1-(k_q+1)\frac{H(X_0,u,r_k R)}{\norm{u}{X_0,r_k R}^2}-\frac{2}{q}\frac{1}{(r_k R)^{n+a-1}\norm{u}{X_0,r_k R}^2}\int_{\partial^0 B^+_{r_k R}(X_0)}F_{\lambda_+,\lambda_-}(u)\mathrm{d}x\right]\\
 & \geq \frac{\norm{u}{X_0,(r_k R)}^2}{(r_k R)^{2k_q}}\left[\frac34-(k_q+1)\frac{H(X_0,u,r_k R)}{\norm{u}{X_0,r_k R}^2}\right]
 \end{align*}
 for $k$ sufficiently large. Together with \eqref{ssurdo}, it implies
 $$
 \frac{H(X_0,u,r_k R)}{\norm{u}{X_0,r_k R}^2} \geq \frac{1}{2(k_q+1)}
 $$
 as $k$ sufficiently large. If we consider the sequence
 $$
 v_k(X)= \frac{u(X_0+r_k RX)}{\norm{u}{X_0,r_k R}}
 $$
 since it is uniformly bounded in $H^{1,a}(B_1)$ and it satisfies $$
\begin{cases}
-L_a v_k=0 & \mbox{in } B_{X_0,r_k R}^+\\
  -\partial^a_y v_k = \left(\ddfrac{(r_k R)^{2s/(2-q)}}{\norm{u}{X_0,r_k R}}\right)^{2-q} \left[\lambda_+ (v_k)_+^{q-1} - \lambda_- (v_k)_-^{q-1}\right] & \mbox{on } \partial^0 B^+_{X_0,r_k R},
\end{cases}
$$
we deduce from the first part of the proof that, up to a subsequence, it converges strongly in $L^{2,a}(B_1), L^{2,a}(\partial B_1)$ and in $H^{1,a}_\loc(B_1)$ to a function $\overline{v} \in H^{1,a}(B_1)$. Moreover, by \eqref{ssurdo}, it solves
$$
\begin{cases}
-L_a \overline{v}=0 & \mbox{in } \R^{n+1}_+\\
  -\partial^a_y \overline{v} = 0 & \mbox{on } \R^{n}\times \{0\}.
\end{cases}
$$
Now, on one side, by the strong convergence in $L^{2,a}(\partial B_1)$ we get
$$
H(0,\overline{v},1) = \lim_{k \to \infty} H(0,v_k,1) = \lim_{k \to \infty} \frac{H(X_0,u,r_k R)}{\norm{u}{X_0,r_k R}^2} \geq \frac{1}{2(k_q+1)},
$$
i.e. $\overline{v} \not \equiv 0$ on $\partial^0 B^+_1$. On the other, by the absurd assumption, we have
$$
\norm{\overline{v}}{0,1/R} = \lim_{k\to \infty} \norm{v_k}{0,1/R}=\lim_{k\to \infty}\frac{\norm{u}{X_0,r_k}}{\norm{u}{X_0, r_k R}}=0,
$$
which implies that $\overline{v}\equiv 0$ on $\partial^0 B^+_{1/R}$. The contradiction follows by the unique continuation property for $L_a$-harmonic function even with respect to $\{y=0\}$.
\end{proof}

\begin{lemma}
Under the previous notations, the sequence $(u_k)_k$ is uniformly bounded in $C^{0,\alpha}_\loc(\R^{n+1}_+)$ for every $\alpha \in (0,\min(1,k_q))$. Moreover, up to a subsequence, it converges uniformly on every compact set of $\R^{n+1}_+$.
\begin{proof}
    The proof follows essentially the ideas of the similar results in \cite{tvz1, tvz2, STT2020} and the result of Lemma \ref{holder}. Indeed, the critical exponent $\min(1,k_q)$ is related to a Liouville type theorem for homogeneous solutions of our problem.
\end{proof}
%
\end{lemma}
So far we have proved the strong convergence of the blow-up sequence $(u_k)_k$ in $H^{1,a}_\loc(\R^{n+1})$ and uniformly on every compact set, to a function $\overline{u} \in H^{1,a}_\loc(\R^{n+1})\cap L^\infty_\loc(\overline{\R^{n+1}_+})$. The next step is to prove the homogeneity of the blow-up limit and the complete characterization of the possible limits.
\begin{proof}[Conclusion of the proof of Theorem \ref{blow.uplimite}]
   Since by Proposition \ref{2s/(2-q)} there exists $C>0$ such that $\alpha_{k}\in (0,C)$, up to a subsequence, we have either
   \be\label{casi}
\frac{\norm{u}{X_0,r_k}}{r_k^{k_q}} \to l \in (0,+\infty) \quad\mbox{or}\quad    \frac{\norm{u}{X_0,r_k}}{r_k^{k_q}}\to +\infty.
   \ee
   First, suppose that the limit $l$ is finite. By Lemma \ref{compact}, together with a diagonal argumet, we get that $u_k \to \overline{u}$ strongly in $H^{1,a}_\loc(\R^{n+1}_+)$ and uniformly on every compact set. It is also clear that the limit $\overline{u}$ solves \eqref{limite} with  $\mu = l^{-2/k_q}$ and $\overline{u} \not\equiv 0$ since, by strong
$H^{1,a}(B_1^+)$-convergence, we have $\norm{\overline{u}}{0,1}=1$. Now, since it remains to prove that $\overline{u}$ is homogeneous, notice that for any $R>0$ we have
\be\label{weiss.cambio} W_{k_q,2}(0,u_k,R)
= \frac{r_k^{2 k_q}}{\norm{u}{X_0,r_k}^2}
 W_{k_q,2}(X_0,u,r_k R).
\ee
Passing to the limit as $k \to \infty$, we deduce by the uniform convergence that
$$
W_{k_q,2}(0,\overline{u},R) = \lim_{k\to \infty} \frac{r_k^{2 k_q}}{\norm{u}{X_0,r_k}^2}
 W_{k_q,2}(X_0,u,r_k R) =\frac{1}{l^2} W_{k_q,2}(X_0,u,0^+),
$$
for any $R>0$, namely we provide that $R\mapsto W_{k_q,2}(0,\overline{u},R)$ is constant and, by Corollary \ref{homo}, it follows that $\overline{u}$ is $k_q$-homogeneous.\\
Let us deal with the second case in \eqref{casi}. We already know, following the same arguments for the case $l\in (0,+\infty)$ up to the validity of a Weiss-type monotonicity result, that, up to a subsequence, $(u_k)_k$ convergence uniformly on every compact set, to a function $\overline{u}\in H^{1,a}_\loc(\R^{n+1})\cap L^\infty_\loc(\overline{\R^{n+1}_+})$ which satisfies
\be\label{sarmonice}
\begin{cases}
-L_a \overline{u}=0 & \mbox{in } \R^{n+1}_+\\
  -\partial^a_y \overline{u} = 0 & \mbox{on } \R^n \times \{0\}.
\end{cases}
\ee
Now, even if \eqref{weiss.cambio} still holds true, we can not conclude that $\overline{u}$ is $k_q$-homogeneous as before.  Instead, by \eqref{weiss.cambio} and the monotonicity of $R \mapsto W_{k_q,2}(X_0,u,R)$, we get
$$
W_{k_q,2}(0,u_k,R)
\leq \frac{r_k^{2 k_q}}{\norm{u}{X_0,r_k}^2}
 W_{k_q,2}(X_0,u,R_0).
 $$
with $R_0 \in (0, \mathrm{dist}(X_0, \partial B_1))$ arbitrarily chosen and $k$ sufficiently large. By the previous estimate, we have
\begin{align*}
\frac{1}{R^{n+a-1}}\int_{B^+_R}{y^a\abs{\nabla u_k}^2\mathrm{d}X} \leq&\,\,
\frac{k_q}{R^{n+a}}\int_{B^+_R}{y^a u_k^2\mathrm{d}X}+
\frac{(r_k R)^{2 k_q}}{\norm{u}{X_0,r_k}^2}
 W_{k_q,2}(X_0,u,R_0)+\\
&\,+\frac{\alpha_k}{R^{n+a-1}}\int_{\partial^0 B^+_R}{F_{\lambda_+,\lambda_-}(u_k)\mathrm{d}x},
 \end{align*}
where the terms in the right hand side go to zero since $\alpha_k \to 0^+$ and $\norm{u}{X_0,r_k}/r_k^{k_q}\to +\infty$. Finally, passing to the limit as $k \to \infty$, we get
\be\label{boh1}
\frac{1}{R^{n+a-1}}\int_{B^+_R}{y^a\abs{\nabla \overline{u}}^2} \leq
k_q\frac{1}{R^{n+a}}\int_{B^+_R}{y^a \overline{u}^2},
\ee
for every $R>0$. On the other hand, since $\mathcal{O}(u, X_0) = k_q$, we get
$$
  \limsup_{r\to 0^+} \frac{1}{r^{2\alpha}} \norm{u}{H^{1,a}(B_r(X_0))}^2 = \begin{cases}
      0, & \mbox{if } 0< \alpha < k_q \\
      +\infty, & \mbox{if } \alpha >k_q.
    \end{cases}
  $$
By Lemma \ref{compact} and \eqref{boh1}, for every $\alpha>0$ we have
\begin{align*}
\frac{1}{R^{2\alpha}} \norm{\overline{u}}{H^{1,a}(B_R)}^2 & \leq \frac{1+k_q}{R^{2\alpha}}\frac{1}{R^{n+a}}\int_{B^+_R}y^a \overline{u}^2\\
& = \lim_{k \to \infty} \frac{1+k_q}{R^{2\alpha}}\frac{1}{(Rr_k)^{n+a}}\int_{B^+_{Rr_k}(X_0)}y^a u^2\\
& = (1+k_q)\lim_{k \to \infty} \frac{H(X_0,u,Rr_k)}{(Rr_k)^{2\alpha}} r_k^{2\alpha}\\
& \leq (1+k_q)r_0^{2\alpha}\limsup_{k \to \infty} \frac{1}{(Rr_k)^{2\alpha}} \norm{u}{X_0, r_k R}^2,
\end{align*}
which yields that $\mathcal{O}(\overline{u},0) \geq k_q$. Since we already know that $\overline{u}$ is a weak solution of \eqref{sarmonice}, by \cite[Lemma 4.7]{STT2020} we get that
$$
\frac{1}{R^{n+a-1}}\int_{B^+_R}{y^a\abs{\nabla \overline{u}}^2} \geq
k_q\frac{1}{R^{n+a}}\int_{B^+_R}{y^a \overline{u}^2},
$$
which implies with \eqref{boh1} that $k_q \in 1+\N$ and that $\overline{u}$ is $k_q$-homogeneous in $\R^{n+1}_+$.
\end{proof}
Having established the compactness of the blow-up sequence for those point such that $\mathcal{O}(u,X_0)=k_q$, we can finally prove the equivalence between the two notion of vanishing order.
\begin{corollary}\label{ottimo}
  For every $X_0 \in \Gamma(u)$, we have $\mathcal{O}(u,X_0)=\mathcal{V}(u,X_0)$.
\end{corollary}
\begin{proof}
  Since we already proved in Proposition \eqref{equivalence1} the previous equivalence for the case $\mathcal{O}(u,X_0)<k_q$, let us focus on the case $\mathcal{O}(u,X_0)=k_q$ and let us prove that
  $$
  0 < \liminf_{r\to 0^+}\frac{H(X_0,u,r)}{\norm{u}{H^{1,a}(B_r(X_0))}^2} \leq 1
  $$
  Since the upper estimate follows by the definition of the norm in $H^{1,a}(B_r(X_0))$, suppose by contradiction that there exists $r_k \to 0^+$ such that
  \be\label{boh2}
  \frac{H(X_0,u,r)}{\norm{u}{H^{1,a}(B_r(X_0))}^2} \to 0^+.
  \ee
  Since $\mathcal{O}(u,X_0)=k_q$, the normalized blow-up sequence
  $$
  u_k(X)=\frac{u(X_0 + r_k X)}{\norm{u}{X_0,r_k}}
  $$
  converges, up to a subsequence, to an homogenous non-trivial solution $\overline{u}$ of \eqref{limite} in $\R^{n+1}$. On the other hand, by \eqref{boh2} we get
  $$
  \int_{\partial^+ B_1^+}y^a \overline{u}^2 = \lim_{k\to \infty}  \int_{\partial^+ B_1^+} y^a u_k^2 = \lim_{k\to \infty} \frac{H(X_0,u,r_k)}{\norm{u}{X_0,r_k}^2} \to 0.
  $$
  By homogeneity, it implies that $\overline{u}\equiv 0$ on $\R^{n+1}$, a contradiction.
\end{proof}
Up to the previous Corollary, we knew that Theorem \ref{boundsopra} was valid for the $H^{1,a}$-vanishing order.
Finally, we can state the proof of the result for the classic vanishing order $\mathcal{V}(u,X_0)$.
\begin{proof}[Proof of Theorem \ref{boundsopra}]
By Proposition \ref{2s/(2-q)} we already know that the maximum admissible $H^{1,a}$-vanishing order is equal to $k_q=2s/(2-q)$. If $\mathcal{O}(u,X_0)<k_q$, by Corollary \ref{equivalence1} and Corollary \ref{equivalence2} we already know that
$$
\mathcal{O}(u,X_0) = \mathcal{V}(u,X_0) = N(X_0,u,0^+).
$$
Therefore by Corollary \ref{blow-up.lim} we know that $\mathcal{V}(u,X_0)$ must be a positive integer.\\
If instead $\mathcal{O}(u,X_0)=k_q$, by Corollary \ref{ottimo} we finally deduce that $\mathcal{V}(u,X_0)=k_q$, as we claimed.
\end{proof}
\section{One-dimensional $k_q$-homogeneous solution}\label{7}
By Theorem \ref{boundsopra} we already know that for those values of $s\in (0,1), q \in [1,2)$ such that $k_q \leq 1$ it holds
$
\Gamma(u)= \mathcal{T}(u)$ with
$$
    \mathcal{T}(u)= \left\{ X \in \Gamma(u) \colon \mathcal{V}(u,X)=k_q\right\}.
$$
In this Section, we prove the existence of $k_q$-homogeneous solutions of \eqref{limite} whose traces on $\R^n\times\{0\}$ are one-dimensional, for those values of the parameters $s$ and $q$ such that $k_q<1$.\\
Thanks to the Federer's reduction principle, this result allows to control the Hausdorff dimension of $\mathcal{T}(u)$ and to prove that the nodal set is a collection of point with vanishing order $k_q$ and Hausdorff dimension less or equal than $(n-1)$, in contrast with the case $s=1$.\\
The classification of $k_q$-homogeneous solution depending only on two-variables $(x_1,y)$ is the starting point for a possible improvement of flatness approach via a viscosity formulation of the sublinear set $\mathcal{T}(u)$. Moreover, we think that this strategy can be easily extended to the case $k_q>1$ by taking care of the classification of $L_a$-harmonic polynomial in \cite{STT2020}. The main result is the following.
\begin{theorem}\label{esempio}
  For every $s \in (0,1), q \in [1,2)$ and $\lambda_+,\lambda_->0$ there exists a $k_q$-homogeneous function $u$ such that $u(0,0)=0$ and
    \be\label{omog}
\begin{cases}
-L_a u = 0 & \mathrm{in }\quad\R^{2}_+\\
-\partial^a_y u = \lambda_+ (u_+)^{q-1} - \lambda_- (u_-)^{q-1} &\mathrm{on }\quad\R \times \{0\}.
\end{cases}
\ee
\end{theorem}
In particular, by exploiting the homogeneity of $u$, the previous problem is equivalent to consider
\begin{equation}\label{eige1d}
\begin{cases}
  -(\sin^a(\theta) \varphi')' = \mu \sin^a(\theta) \varphi & \mbox{in } (0,\pi) \\
  -\partial^a_\theta \varphi(0) = \lambda_+ (\varphi_+(0))^{q-1} - \lambda_- (\varphi_-(0))^{q-1}\\
  -\partial^a_\theta \varphi(\pi) = \lambda_+ (\varphi_+(\pi))^{q-1} - \lambda_- (\varphi_-(\pi))^{q-1},
\end{cases}
\end{equation}
with $\mu=k_q\left(k_q+1-2s\right)$ and $u(X)=\abs{X}^k\varphi(X\abs{X}^{-1})$.
In order to simplify the proofs, we will consider first the case $\lambda_+=\lambda_-$ and we plan to prove existence of solutions of \eqref{omog} whose traces on $\R\times \{0\}$ are either of the form
$$
u(x,0)=A\left(x_+^{k_q} - x_-^{k_q}\right) \quad\mbox{or}\quad u(x,0)=A\abs{x}^{k_q}.$$
In the end, this result will implies the existence of solution of \eqref{eige1d} such that $\varphi(\theta)=\varphi(\pi-\theta)$.\vspace{0.5cm}\\
In the following Lemma we prove the existence of $T\in (0,\pi)$ such that there exists a positive eigenfunction $\varphi$ in $(0,T)$ which satisfies $\varphi(T)=0$ and the non-homogeneous Neumann condition in $\theta=0$.
\begin{lemma}\label{T}
Given $T\in(0,\pi)$ and
  $$
X= \left\{u \in H^{1,a}\left(\left(0,T\right)\right) \colon u(T)=0\right\},
$$
let us consider the mixed Dirichlet-Neumann eigenvalue associated to $(0,T)$
$$
\lambda_M(T)=\min\left\{\frac{\int_0^{T}\sin^a(\theta)(u')^2}{\int_0^{T}\sin^a(\theta)u^2}\colon u\in X \setminus\{0\}, \partial^a_\theta u(0)=0\right\}.
$$
Then, if $\mu < \lambda_M(T)$ there exists an unique positive function $\varphi \in X$ such that
\be\label{prob}
\begin{cases}
  -(\sin^a(\theta) \varphi')' = \mu \sin^a(\theta) \varphi & \mathrm{in }\quad (0,T) \\
  -\partial^a_\theta \varphi(0) = \lambda_+ (\varphi_+(0))^{q-1}.
\end{cases}
\ee
\end{lemma}
\begin{proof}
  Under the previous notations, let us consider the minimization problem $\min_{\varphi \in X} J(\varphi) $ with
$$
J(u) = \frac12\int_0^{T}{\sin^a(\theta)\left( (u')^2-k_q\left(k_q + 1-2s \right)u^2 \right)\mathrm{d}\theta} - \frac{F_{\lambda_+,0}(u)(0)}{q}.
$$
Since $q\in [1,2)$, for every $u\in X$ there exists $\overline{t}>0$ small enough such that $J(tu)<0$ for every $t \in (0,\overline{t})$.\\
Notice that critical point of $J$ in $X$ are solution of \eqref{prob}, i.e. for every $\phi \in X$ we get
\begin{align*}
\mathrm{d}J(u)[\phi] =& \int_0^{T}{\sin^a(\theta)\left( u'\phi'-k_q\left(k_q + 1-2s \right)u \phi \right)\mathrm{d}\theta} +\\
& -\left(\lambda_+ (u_+(0))^{q-1}\phi_+(0)\right)\\
= & -\int_0^{T}{\left((\sin^a(\theta) u')'+k_q\left(k_q + 1-2s \right)u \right)\phi\mathrm{d}\theta} +\\
& -\partial^a_\theta u(0)\phi(0)-\left(\lambda_+ (u_+(0))^{q-1}\phi_+(0)\right).
\end{align*}
By the Sobolev embedding, for every $n>2s, q \in [1,2)$ it holds
$$
\int_{S^{n-1}}g^q \mathrm{d}\sigma_x\leq \tilde{C}\abs{\partial^0 B^+}^{\frac{2n-(n-2s)q}{2n}}\left(\int_{S^n_+}{\sin^a(\theta) \abs{\nabla_S g}^2 \mathrm{d}\sigma_X} + (k_q^2+n+2k_q-2s)\int_{S^n_+}\sin^a(\theta) g^2\mathrm{d}\sigma_X\right)^{q/2}
$$
with
$$
\tilde{C}=\frac{(n+k_q q)(C_{n,s}N_s)^{q/2}}{(n+2k_q-2s)^{q/2}},\quad
N_s = 2^{2s-1}\frac{\Gamma(s)}{\Gamma(1-s)},\quad C_{n,s}= \frac{2^{-2s}}{\pi^s}\left(\frac{\Gamma(\frac{n-2s}{2})}{\Gamma(\frac{n+2s}{2})}\right) \left( \frac{\Gamma(n)}{\Gamma(n/2)}\right)^{\frac{2s}{n}}.
$$
Thus, for $n=1$ and  we get
\begin{align*}
J(u)\geq &\,\,\frac12\int_{0}^{T}{\sin^a(\theta) \left( (u')^2-\mu u^2 \right)\mathrm{d}\theta}+\\
&\,- \frac{\lambda_+}{q}C\left(\int_{0}^{T}{\sin^a(\theta)(u')^2\mathrm{d}\theta} + (k_q^2+1+2k_q-2s)\int_{0}^{T}\sin^a(\theta)u^2\mathrm{d}\theta\right)^{q/2}
\end{align*}
with
$$
C =\left(\frac{\Gamma(s)}{\Gamma(1-s)}\right)^{q/2}
\frac{1}{\pi^{qs}}\left(\frac{\Gamma(\frac{1-2s}{2})}{\Gamma(\frac{1+2s}{2})}\right)^{q/2}
\frac{1+k_q q}{(1+2k_q-2s)^{q/2}}2^{1-(1-s)q}.
$$
Moreover, since by the Poincaré inequality in $X$ we have
$$
\int_0^{T}\sin^a(\theta)u^2 \mathrm{d}\theta \leq C_p \int_0^{T}\sin^a(\theta)(u')^2\mathrm{d}\theta,
$$
for some positive constant $C_p$, we get
\begin{align*}
J(u)\geq &\,\, \frac12\left(1-C_p\mu\right)\int_{0}^{T}{\sin^a(\theta)(u')^2\mathrm{d}\theta}+\\ &\,- \frac{\Lambda ((k_q^2+1+2k_q-2s)C_p +1)^{q/2}}{q}C\left(\int_{0}^{T}{\sin^a(\theta)(u')^2}\mathrm{d}\theta\right)^{q/2}.
\end{align*}
Finally, since
$$
\frac{1}{C_p} = \min_{u\in X}\ddfrac{\int_0^{T}\sin^a(\theta)(u')^2\mathrm{d}\theta}{\int_0^{T}\sin^a(\theta)u^2\mathrm{d}\theta}=\lambda_M(T),
$$
we get $C_p\mu<1$, which implies that $J$ is bounded
from below and coercive. Since $X$ is weakly closed, the direct method of the calculus of variations implies the existence of a
minimizer $u$ which solves \eqref{prob}. Moreover, we can prove that $u$ is positive: indeed, since if $u$ is a minimizer the same holds also for $\abs{u}$, we can already suppose that $u\geq 0$. Now the strong maximum principle implies
that either $u>0$ or $u \equiv 0$, but the latter options can be easily ruled out observing that $J(u)<0$.\\
Finally, if we suppose there exists two different solutions $\varphi_1,\varphi_2$ of \eqref{prob}, it is straightforward to see that there exists a linear combination $w=\varphi_1-C\varphi_2$, with $C>0$ such that $\varphi_1^{q-1}(0)=C\varphi_2^{q-1}(0)$ and
\be\label{assu}
 -\partial_\theta^a w(0)= -\partial^a_\theta \varphi_1(0) + C \partial^a_\theta \varphi_2(0) = \lambda_+(\varphi_1(0)^{q-1}-C\varphi_2(0)^{q-1})=0  .
\ee
Moreover
$$
\begin{cases}
  -(\sin^a(\theta) w')' = \mu \sin^a(\theta) w & \mbox{in } (0,T) \\
  w(T)=0, \partial^a_\theta w(0) = 0.
\end{cases}
$$
Necessary $w$ must vanishes identically in $(0,T)$: indeed, if not either the function is strictly positive in $(0,T)$ or it changes sign in $(0,T)$, both in contradiction with the assumption $\mu < \lambda_M(T)$. Hence, $\varphi_1\equiv C\varphi_2$ in $[0,T]$, which contradicts the definition of $C$.
\end{proof}
\begin{theorem}
  Let $k_q<1$, then for every $\lambda_+>0$ there exist only two $k_q$-homogeneous solutions $u_1,u_2 \in H^{1,a}(\R^{2}_+)$ of
  $$
\begin{cases}
-L_a u = 0 & \mathrm{in }\quad\R^{2}_+\\
-\partial^a_y u = \lambda_+ \abs{u}^{q-2}u &\mathrm{on }\quad\R \times \{0\},
\end{cases}
  $$
such that
\be\label{boh3}
u_1(x,0)=A_1\left(x_+^{k_q} - x_-^{k_q}\right) \quad\mbox{or}\quad u_2(x,0)=A_2\abs{x}^{k_q},\ee
for some positive constants $A_1,A_2$ depending only on $s,q$ and $\lambda_+$.
\end{theorem}
\begin{proof}
Notice first that the condition $k_q<1$ immediately implies $s \in (0,1/2)$. Since we plan to prove the existence of a $k_q$-homogeneous function, it is obvious that its trace must be of the form \eqref{boh3}. Moreover, if we suppose by contradiction that there exist two solutions $u$ and $v$ with the same type of traces (either like $u_1(\cdot,0)$ or $u_2(\cdot,0)$) then, it must exist a constant $C>0$ such that $u^{q-1}_\pm(x,0)=Cv^{q-1}_\pm(x,0)$ in $\R$. Consequently, the function $w=u-Cv$ is a $k_q$-homogeneous solution of
$$
\begin{cases}
L_a w = 0 & \mathrm{in }\quad\R^{2}_+\\
-\partial^a_y w = 0 &\mathrm{on }\quad \R \times \{0\}.
\end{cases}
$$
By the classification of \cite[Lemma 4.7]{STT2020} we already know that either $k_q\in 1 + \N$ or $w\equiv 0$. Since $k_q<1$, necessary $w\equiv 0$, in contradiction with the choice of $C>0$.\\

In order to construct two functions with these features, let us consider the symmetric and antisymmetric solution of the eigenvalue problem associated the traces on $S^1$ of $u$.\\
 Hence, for the antisymmetric case, fixed $T=\pi/2$, by Lemma \ref{T} there exists $\varphi \in H^{1,a}(0,\pi/2)$ such that $\varphi(\pi/2)=0$ and
 $$
 \begin{cases}
  -(\sin^a(\theta) \varphi')' = \mu \sin^a(\theta) \varphi & \mathrm{in }\quad (0,\pi/2) \\
  -\partial^a_\theta \varphi(0) = \lambda_+ (\varphi_+(0))^{q-1}.
\end{cases}
 $$
 Hence, we define
 $$
 \varphi_1(\theta) = \begin{cases}
                       \varphi(\theta) & \mbox{if } \theta\in (0,\pi/2) \\
                       -\varphi(\pi-\theta) & \mbox{if }\theta \in (\pi/2,\pi)
                     \end{cases},
$$
an antisymmetric solution of \eqref{eige1d} with $\lambda_+=\lambda_-$.
 On the other hand, let us consider the symmetric eigenfunction $\phi$ defined as
 \be\label{caso.sym}
 \begin{cases}
   -(\sin^a(\theta)\phi')'=\lambda_1(T)\sin^a(\theta)\phi & \mbox{in } (T,\pi-T)  \\
   \phi>0 & \mbox{in }(T,\pi-T)\\
   \phi(T)=0=\phi(\pi-T),
 \end{cases}
 \ee
 for $T \in (0,\pi/2)$, where $\lambda_1(T)$ is the fist eigenvalue associated to $(T,\pi-T)$. By monotonicity of the eigenvalue with respect to the set inclusion, we already know that $T \mapsto \lambda_1(T)$ is increasing and it satisfies
 $$
 \lim_{T \to 0^+}\lambda_1(T) = 2s \quad\mbox{and}\quad \lambda_1(\arctan( \sqrt{2(1-s)}))=2.
 $$
 Thus, since $s<1/2$, there exists $T^* \in (0,\arctan( \sqrt{2(1-s)}))$ such that $\lambda_1(T^*)=2s/(2-q)$. Furthermore, by applying Lemma \ref{T} with $T=T^*$, there exists a function $\psi\in H^{1,a}(0,T^*)$ such that $\psi(T^*)=0$ and
 $$
 \begin{cases}
  -(\sin^a(\theta) \psi')' = \mu \sin^a(\theta) \psi & \mathrm{in }\quad (0,T^*) \\
  -\partial^a_\theta \psi(0) = \lambda_+ (\psi_+(0))^{q-1}.
\end{cases}
 $$
Finally, let $C>0$ be such that $-C\phi'(T^*)=\psi'(T^*)$, then if we define $$
\varphi_2(\theta) = \begin{cases}
                                              \psi(\theta) & \mbox{if } \theta\in (0,T) \\
                       -C\phi(\theta) & \mbox{if }\theta \in (T,\pi-T)\\
                       \psi(\pi-\theta) & \mbox{if }\theta \in (\pi-T,\pi)
                     \end{cases},
 $$
 we get a symmetric solution of \eqref{eige1d} with
Thus, the solutions $u_i$ are defined as the homogeneous extension of $\varphi_i$ in $\overline{\R^{n+1}_+}$
 $$
 u_i(X)= \abs{X}^{2s/(2-q)}\varphi_i\left(\frac{X}{\abs{X}}\right),
 $$
 which gives the claimed result.
\end{proof}
Finally, by applying the Federer's reduction principle in the form of \cite[Theorem 8.5]{MR1637972}, we can conclude the proof of Theorem \ref{hau} as a byproduct of the results of this Section.
\begin{proof}[Conclusion of the proof of Theorem \ref{hau}]
   Let us consider the class of functions $\mathcal{F}$ defined as
$$
\mathcal{F}=\Bigg\{ u \in L^\infty_\loc(\R^{n+1})\setminus \{ 0\}\Bigg\vert \,\,
\begin{aligned}
 &u \mbox{ solves }\eqref{omog}\mbox{ in }B_r(X_0), \mbox{ for some }r \in \R,\, X_0 \in \R^{n}\times\{0\}\\
 &\mbox{for some }\lambda_+,\lambda_-,\mu >0
 \end{aligned}
 \Bigg\}.
$$
endowed with the topology associated to the uniform convergence and
$$
\overline{\mathcal{S}}\colon u \mapsto \mathcal{T}(u).
$$
We already know that $\mathcal{F}$ is close under rescaling, translation and normalization. Moreover, by Theorem \ref{blow.uplimite} the hypothesis of the existence of a blow-up limit in $\mathcal{F}$ is satisfied, as well as the singular set assumption. Thus, the Federer's reduction principle \cite[Theorem 8.5]{MR1637972} is applicable and it implies the existence of an integer $d \in [0,n]$ such that
$$
\mathrm{dim}_\mathcal{H} \mathcal{T}(u) \leq d,
$$
for every function $u\in \mathcal{F}$.
Suppose by contradiction that $d = n$, this would implies the existence of $\varphi \in \mathcal{F}$ such that $\overline{\mathcal{S}}(\varphi) = \R^n$ i.e., $\varphi \equiv 0$ on $\R^n$. Thus $\varphi \equiv 0$ on the whole $\R^{n+1}$, which contradicts the fact the $0 \not\in \mathcal{F}$.
Actually, since Theorem \ref{esempio} ensures the existence of a $(n-1)$-linear subspace $E\subset \R^n$ and a $k_q$-homogeneous function $\varphi \in \mathcal{F}$ such that $\overline{\mathcal{S}}(\varphi)=E$, we get $d=n-1$.
\end{proof}
\bibliography{Biblio}

\begin{thebibliography}{10}

\bibitem{allengarcia}
M.~Allen and M.~S.~V. Garcia.
\newblock The fractional unstable obstacle problem.
\newblock {\em Nonlinear Analysis}, 193:111459, 2020.
\newblock Nonlocal and Fractional Phenomena.

\bibitem{allenlindgrenpetro}
M.~Allen, E.~Lindgren, and A.~Petrosyan.
\newblock The two-phase fractional obstacle problem.
\newblock {\em SIAM J. Math. Anal.}, 47(3):1879--1905, 2015.

\bibitem{allenpetro}
M.~Allen and A.~Petrosyan.
\newblock A two-phase problem with a lower-dimensional free boundary.
\newblock {\em Interfaces Free Bound.}, 14(3):307--342, 2012.

\bibitem{CS2007}
L.~Caffarelli and L.~Silvestre.
\newblock An extension problem related to the fractional {L}aplacian.
\newblock {\em Comm. Partial Differential Equations}, 32(7-9):1245--1260, 2007.

\bibitem{fermi}
L.~A. Caffarelli and A.~Friedman.
\newblock The free boundary in the {T}homas-{F}ermi atomic model.
\newblock {\em J. Differential Equations}, 32(3):335--356, 1979.

\bibitem{MR1637972}
X.-Y. Chen.
\newblock A strong unique continuation theorem for parabolic equations.
\newblock {\em Math. Ann.}, 311(4):603--630, 1998.

\bibitem{MR2944369}
E.~Di~Nezza, G.~Palatucci, and E.~Valdinoci.
\newblock Hitchhiker's guide to the fractional {S}obolev spaces.
\newblock {\em Bull. Sci. Math.}, 136(5):521--573, 2012.

\bibitem{MR943927}
H.~Donnelly and C.~Fefferman.
\newblock Nodal sets of eigenfunctions on {R}iemannian manifolds.
\newblock {\em Invent. Math.}, 93(1):161--183, 1988.

\bibitem{fallfelli2}
M.~M. Fall and V.~Felli.
\newblock Unique continuation property and local asymptotics of solutions to
  fractional elliptic equations.
\newblock {\em Comm. Partial Differential Equations}, 39(2):354--397, 2014.

\bibitem{fallfelli1}
M.~M. Fall and V.~Felli.
\newblock Unique continuation properties for relativistic {S}chr\"{o}dinger
  operators with a singular potential.
\newblock {\em Discrete Contin. Dyn. Syst.}, 35(12):5827--5867, 2015.

\bibitem{MR2514337}
C.~Fefferman.
\newblock Extension of {$C^{m,\omega}$}-smooth functions by linear operators.
\newblock {\em Rev. Mat. Iberoam.}, 25(1):1--48, 2009.

\bibitem{MR833393}
N.~Garofalo and F.-H. Lin.
\newblock Monotonicity properties of variational integrals, {$A_p$} weights and
  unique continuation.
\newblock {\em Indiana Univ. Math. J.}, 35(2):245--268, 1986.

\bibitem{MR882069}
N.~Garofalo and F.-H. Lin.
\newblock Unique continuation for elliptic operators: a geometric-variational
  approach.
\newblock {\em Comm. Pure Appl. Math.}, 40(3):347--366, 1987.

\bibitem{MR4018099}
N.~Garofalo and X.~Ros-Oton.
\newblock Structure and regularity of the singular set in the obstacle problem
  for the fractional {L}aplacian.
\newblock {\em Rev. Mat. Iberoam.}, 35(5):1309--1365, 2019.

\bibitem{MR1305956}
Q.~Han.
\newblock Singular sets of solutions to elliptic equations.
\newblock {\em Indiana Univ. Math. J.}, 43(3):983--1002, 1994.

\bibitem{MR1639155}
Q.~Han, R.~Hardt, and F.-H. Lin.
\newblock Geometric measure of singular sets of elliptic equations.
\newblock {\em Comm. Pure Appl. Math.}, 51(11-12):1425--1443, 1998.

\bibitem{MR1090434}
F.-H. Lin.
\newblock Nodal sets of solutions of elliptic and parabolic equations.
\newblock {\em Comm. Pure Appl. Math.}, 44(3):287--308, 1991.

\bibitem{Nekvinda}
A.~Nekvinda.
\newblock Characterization of traces of the weighted {S}obolev space
  {$W^{1,p}(\Omega,d^\epsilon_M)$} on {$M$}.
\newblock {\em Czechoslovak Math. J.}, 43(118)(4):695--711, 1993.

\bibitem{ruland2}
A.~R\"{u}land.
\newblock Unique continuation for fractional {S}chr\"{o}dinger equations with
  rough potentials.
\newblock {\em Comm. Partial Differential Equations}, 40(1):77--114, 2015.

\bibitem{ruland1}
A.~R\"{u}land.
\newblock On quantitative unique continuation properties of fractional
  {S}chr\"{o}dinger equations: doubling, vanishing order and nodal domain
  estimates.
\newblock {\em Trans. Amer. Math. Soc.}, 369(4):2311--2362, 2017.

\bibitem{MR3857504}
A.~R\"{u}land.
\newblock Unique continuation for sublinear elliptic equations based on
  {C}arleman estimates.
\newblock {\em J. Differential Equations}, 265(11):6009--6035, 2018.

\bibitem{silvestrepaper}
L.~Silvestre.
\newblock Regularity of the obstacle problem for a fractional power of the
  {L}aplace operator.
\newblock {\em Comm. Pure Appl. Math.}, 60(1):67--112, 2007.

\bibitem{Simon83}
L.~Simon.
\newblock {\em Lectures on geometric measure theory}, volume~3 of {\em
  Proceedings of the Centre for Mathematical Analysis, Australian National
  University}.
\newblock Australian National University, Centre for Mathematical Analysis,
  Canberra, 1983.

\bibitem{STT2020}
Y.~Sire, S.~Terracini, and G.~Tortone.
\newblock On the nodal set of solutions to degenerate or singular elliptic
  equations with an application to s-harmonic functions.
\newblock {\em Journal de Mathématiques Pures et Appliquées}, 2020.

\bibitem{vita2020}
Y.~{Sire}, S.~{Terracini}, and S.~{Vita}.
\newblock {Liouville type theorems and regularity of solutions to degenerate or
  singular problems part I: even solutions}.
\newblock {\em arXiv e-prints}, page arXiv:1904.02143, Apr. 2019.

\bibitem{soavesublinear}
N.~Soave and S.~Terracini.
\newblock The nodal set of solutions to some elliptic problems: sublinear
  equations, and unstable two-phase membrane problem.
\newblock {\em Adv. Math.}, 334:243--299, 2018.

\bibitem{soavesingular}
N.~Soave and S.~Terracini.
\newblock The nodal set of solutions to some elliptic problems: singular
  nonlinearities.
\newblock {\em J. Math. Pures Appl. (9)}, 128:264--296, 2019.

\bibitem{soaveweth}
N.~Soave and T.~Weth.
\newblock The unique continuation property of sublinear equations.
\newblock {\em SIAM J. Math. Anal.}, 50(4):3919--3938, 2018.

\bibitem{tvz2}
S.~Terracini, G.~Verzini, and A.~Zilio.
\newblock Uniform {H}\"older regularity with small exponent in
  competition-fractional diffusion systems.
\newblock {\em Discrete and Continuous Dynamical Systems- Series A},
  34(6):2669--2691, 2014.

\bibitem{tvz1}
S.~Terracini, G.~Verzini, and A.~Zilio.
\newblock Uniform {H}\"older bounds for strongly competing systems involving
  the square root of the laplacian.
\newblock {\em Journal of the European Mathematical Society},
  18(12):2865--2924, 2016.

\bibitem{yijing}
Y.~{Wu}.
\newblock {A non-local one-phase free boundary problem from obstacle to
  cavitation}.
\newblock {\em arXiv e-prints}, page arXiv:1810.05535, Oct. 2018.

\bibitem{yang}
R.~Yang.
\newblock Optimal regularity and nondegeneracy of a free boundary problem
  related to the fractional {L}aplacian.
\newblock {\em Arch. Ration. Mech. Anal.}, 208(3):693--723, 2013.

\end{thebibliography}
\bibliographystyle{abbrv}

\end{document}